\documentclass[11pt]{article}
\usepackage[T1]{fontenc}
\usepackage{lmodern}
\usepackage{amsmath}
\usepackage{amsfonts}
\usepackage{amssymb}
\usepackage{amsthm}
\usepackage{graphicx}
\usepackage{etoolbox,xstring,mfirstuc,textcase}
\usepackage{empheq}
\usepackage{indentfirst}
\usepackage{cite} 
\usepackage{mathrsfs}
\usepackage{cases}
\usepackage{graphics}
\usepackage{xcolor}  
\usepackage{bm}
\usepackage{changes}
\usepackage{lineno}
\usepackage{cases}
\usepackage{enumerate}
\usepackage{enumitem}
\definechangesauthor[name={R1 cusse}, color=orange]{r1}
\definechangesauthor[name={Ty cusse}, color=purple]{ty}
\usepackage[colorlinks=true,linkcolor=blue,citecolor = blue]{hyperref}
\newtheorem{theorem}{\textbf{Theorem}}[section]
\newtheorem{lemma}{\textbf{Lemma}}[section]
\newtheorem{proposition}{\textbf{Proposition}}[section]

\newtheorem{remark}{\textbf{Remark}}[section]

\allowdisplaybreaks[4] 
\numberwithin{equation}{section}

\def \u {\boldsymbol{u}}
\def \D {\mathcal{D}}
\def \v {\boldsymbol{v}}

\def \V {\boldsymbol{V}_{\sigma}}
\def \H {\bm{H}}

\def \d {\: \mathrm{d}}
\def \dt {\: \mathrm{d}t}
\def \dx {\: \mathrm{d}x}
\def \G {\mathcal{G}}
\def \tot {\mathrm{tot}}

\begin{document}
\title{Global Weak Solutions of a Navier--Stokes--Cahn--Hilliard System for Incompressible Two-phase Flows with Thermo-induced Marangoni Effects
\footnote{Published as L.-X. Chen and H. Wu, Global weak solutions of a Navier--Stokes--Cahn--Hilliard system for incompressible two-phase flows with thermo-induced Marangoni effects, Advances in Differential Equations, \textbf{31}(11\&12) (2026), 989--1046.}}
	
\author{
{Lingxi Chen}\footnote{School of Mathematical Sciences,
Fudan University, Handan Road 220, Shanghai 200433, P. R. China.
Email:  \texttt{chenlingxi@gmail.com}}
\ \ and \ \
{Hao Wu}\footnote{Corresponding author. School of Mathematical Sciences,
Fudan University, Handan Road 220, Shanghai 200433, P. R. China.
Email:  \texttt{haowufd@fudan.edu.cn}}
}
	
\date{}
	
\maketitle

\begin{abstract}
\noindent
We study a diffuse-interface model that describes the dynamics of two-phase incompressible flows driven by the thermo-induced Marangoni effect. The hydrodynamic system consists of the Navier--Stokes equations for the fluid velocity, the convective Cahn--Hilliard equation for the phase-field variable, and a convective heat equation for the (relative) temperature. For the initial-boundary value problem in two and three dimensions with variable viscosity, mobility, thermal diffusivity, and a physically relevant singular potential, we establish the existence of global weak solutions. The proof relies on an implicit-explicit time discretization scheme that preserves the $L^\infty$-bounds of both the phase-field variable and the temperature. When the spatial dimension is two, we prove the uniqueness of weak solutions for the case with matched densities under suitable assumptions on the initial temperature, mobility, and thermal diffusivity.
\end{abstract}

\section{Introduction}
\label{Introduction}

The Marangoni effect, named after Carlo Marangoni,   describes the fluid motion generated along an interface due to gradients in surface tension \cite{Marangoni}. This inhomogeneity of surface tension typically arises from spatial variations in temperature (thermocapillarity) or solute concentration (solutocapillarity) \cite{Marangoni-coffee,Marangoni_turbulence}. When thermal effects are considered, for most fluids, the surface tension decreases with increasing temperature \cite{Marangoni_review}, and consequently, the resulting tangential stress induces a convective flow from the warmer regions (lower surface tension) to the cooler regions (higher surface tension). This ubiquitous phenomenon was recognized as a primary driving mechanism that contributes to fluid motion alongside buoyancy effects in the classic B\'{e}nard experiments \cite{Benard1901}. Marangoni convection can influence the spreading behavior near interfaces and lead to the appearance of dynamic organized patterns. It has become a fundamental subject of research, with diverse applications in crystal growth/welding, electron beam melting, nanotechnology, and biology (see, e.g., \cite{Marangoni_crystal,Marangoni_welding,Marangoni_electron,NB20}).

In this study, we consider the evolution of a non-isothermal fluid mixture with two incompressible immiscible viscous Newtonian fluids under the influence of thermally induced surface tension gradients. In the sharp interface formulation (see, e.g., \cite{sharp1,sharp2}), the transition region between immiscible constituents is idealized as a zero-thickness hypersurface separating two distinct bulk domains. Although this treatment provides a clear mathematical representation of phase boundaries, it becomes challenging during topological changes such as droplet pinch-off and coalescence, where vanishing length scales can generate singular interfacial dynamics, see \cite{Eggers1997,Charles1960}.
This difficulty has motivated alternative interfacial descriptions that can regularize the sharp-boundary idealization and provide a more flexible framework for describing interfacial deformation and topological transitions. Phase-field models have emerged as an efficient method for studying interfacial dynamics and phase transitions in multi-component fluids, see \cite{phasefield1,phasefield4} and the references therein. In this approach, the sharp interface is regularized as a thin transition layer where physical variables vary continuously but steeply, i.e., the so-called diffuse interface. The diffuse interface regularization avoids explicit interface tracking while naturally incorporates thermodynamic driving forces, facilitating both the mathematical analysis and numerical implementation for multi-phase flows \cite{phasefield1,phasefield2,phasefield3}.

Let $\Omega \subset \mathbb{R}^d$ ($d = 2$ or $3$) be a bounded domain with a $C^3$-boundary $\partial \Omega$. We consider the following hydrodynamic phase-field system:
\begin{align}
	&\partial_t(\rho(\phi)\u) + \mathrm{div}(\u \otimes(\rho(\phi)\u +\mathbf{J})) - \mathrm{div}\,(2 \nu(\phi,\theta) D \u) + \nabla p
\notag \\
&\qquad = - \mathrm{div}\, \boldsymbol{\sigma} + \bm{f}_{\mathrm{b}}(\phi, \theta),
	\label{NSCHM}\\
	&\mathrm{div}\, \u = 0,
	\label{INCOMPRESS}\\
	&\partial_t \phi +\u \cdot \nabla \phi = \mathrm{div} (m(\phi,\theta) \nabla \mu),
	\label{CH1}\\
	&\mu = -\Delta \phi + W^{\prime}(\phi),
	\label{CH2}\\
	&\partial_t \theta +\u \cdot \nabla \theta=  \mathrm{div}\,(\kappa(\phi,\theta) \nabla \theta),
	\label{BOUSSINESQ}
\end{align}	
in $\Omega \times (0,\infty)$, subject to the boundary and initial conditions:
\begin{alignat}{4}
    &\u=\mathbf{0}, \quad
	&&\partial_{\mathbf{n}} \phi = \partial_{\mathbf{n}} \mu =0, \quad
	&&\theta=\theta_{\mathrm{b}}, \quad
	&&\text{on} ~ \partial \Omega \times (0,\infty),
	\label{boundary}\\
	&\u|_{t=0}=\u_0, \quad
	&&\phi|_{t=0}=\phi_0, \quad
	&&\theta|_{t=0} = \theta_0, \quad
	&&\text{in} ~ \Omega.
	\label{initial}
\end{alignat}
In \eqref{boundary}, $\mathbf{n}=\mathbf{n}(x)$ denotes the outward unit normal vector on $\partial\Omega$ and $\partial_{\mathbf{n}}$ denotes the outward normal derivative on the boundary, $\theta_{\mathrm{b}}=\theta_{\mathrm{b}}(x)$ is a given function on $\partial\Omega$.
The state variables of the system \eqref{NSCHM}--\eqref{BOUSSINESQ} are denoted by $(\u, p, \phi, \mu, \theta)$. More precisely, $\u:\Omega\times(0,\infty)\to\mathbb{R}^d$ denotes the volume-averaged velocity of the binary fluid mixture and $p:\Omega\times(0,\infty)\to\mathbb{R}$ the pressure, which are governed by a modified Navier--Stokes system  \eqref{NSCHM}--\eqref{INCOMPRESS}; $\phi:\Omega\times(0,\infty)\to [-1,1]$ is the phase-field variable and $\mu:\Omega\times(0,\infty)\to\mathbb{R}$ the chemical potential, together forming the convective Cahn--Hilliard equation \eqref{CH1}--\eqref{CH2}; and $\theta:\Omega\times(0,\infty)\to\mathbb{R}$ denotes the relative temperature (with respect to a given constant ambient temperature), which satisfies the convective heat transport equation \eqref{BOUSSINESQ}.

The coupled system \eqref{NSCHM}--\eqref{initial} presents a nontrivial coupling between fluid dynamics, phase separation process, and thermal convection, leading to complex interfacial phenomena \cite{Marangoni_interface,Marangoni_turbulence}. It not only generalizes the phase-field model for thermo-induced Marangoni effects derived in \cite{EnVarA1,EnVarA2} via the energetic variational approach, but also incorporates the thermodynamically consistent framework proposed in \cite{AGGmodel} for the general scenario with unmatched densities, based on the utilization of a volume-averaged velocity $\bm{u}$ that keeps the binary fluid mixture incompressible (cf. \eqref{INCOMPRESS}). The phase-field variable $\phi$ serves as a conserved order parameter that represents the difference in volume fractions of the fluid mixture, such that $\{\phi=-1\}$ represents fluid $1$ and $\{\phi=1\}$ represents
fluid $2$. Then the average density $\rho$ is assumed to be the typical linear form
\begin{align}
\rho(\phi)=\frac{\rho_2-\rho_1}{2}\phi+\frac{\rho_1+\rho_2}{2},
\label{density}
\end{align}
where the positive constants $\rho_1$, $\rho_2$ denote the homogeneous positive density of the unmixed components of the fluid. For simplicity, the density variation with respect to the temperature is explicitly considered in the buoyancy force $\bm{f}_{\mathrm{b}}$ and assumed to satisfy the linearized thermal expansion approximation (see e.g., \cite{EnVarA2}), that is,
\begin{align}
\bm{f}_{\mathrm{b}}(\phi, \theta) = -\rho(\phi)(1-\alpha\theta)g \mathbf{e}_d.
\label{bouy}
\end{align}
Here, $\alpha$ is the coefficient of thermal expansion, $g$ is the gravitational acceleration, and $\mathbf{e}_d$ denotes the unit upward vector, i.e., $\mathbf{e}_2=(0,1)^T$, $\mathbf{e}_3=(0,0,1)^T$. In the equation of momentum balance \eqref{NSCHM}, the matrix-valued function $D \u := \frac{1}{2}(\nabla \u + \nabla \u^T)$ denotes the symmetric gradient of $\u$, while the relative flux related to the diffusion of the fluid components is given by
\begin{align}
\mathbf{J}=-\frac{\rho_2-\rho_1}{2}m(\phi,\theta)\nabla \mu.
\label{Jflux}
\end{align}
The Cauchy stress tensor $\boldsymbol{\sigma}$ is defined as
\begin{align}
\label{Cauchy_stress_tensor}
	\boldsymbol{\sigma} =  \lambda(\theta) ( \nabla \phi \otimes \nabla \phi ) +  \lambda(\theta)
	\left( \frac{1}{2}|\nabla \phi|^2 + W(\phi)\right)  \mathbb{I}_d,
\end{align}
where $\mathbb{I}_d$ denotes the $d$-dimensional unit matrix, and the temperature-dependent surface tension coefficient $\lambda$ follows the empirical E\"{o}tv\"{o}s law:
\begin{align}
	\lambda(\theta) = \lambda_0(a - b \theta), \label{Eotvos}
\end{align}
with constant coefficients $\lambda_0,a,b>0$. Since the surface tension depends on temperature, its spatial variation along the interface induces tangential stresses that drive interfacial fluid motion, i.e., the Marangoni convection. The mixing energy of binary fluids is given by the classical Ginzburg--Landau type free energy (cf. \cite{phasefield1,Cahn1958})
\begin{align}
\label{mix}
E_{\mathrm{mix}}(\phi)= \int_\Omega \frac12|\nabla \phi|^2+ W(\phi)\,\mathrm{d}x,
\end{align}
in which the gradient term contributes to the free-energy excess of the interfacial region and the nonlinear function $W=W(\phi)$ denotes the bulk energy density with a double well structure. Here, we set the width of the diffuse interface to $1$ for simplicity, since we do not consider the sharp interface limit in this study. A physically relevant example of $W$ is the logarithmic (Flory--Huggins) potential \cite{Cahn1958}:
\begin{align}
\label{Wphi}
	W(\phi) = \frac{A}{2} \left[ (1+\phi)\ln(1+\phi) + (1-\phi)\ln(1-\phi) \right] - \frac{A_c}{2} \phi^2, \quad \phi \in (-1,1),
\end{align}
with constant coefficients $A, A_c$ satisfying $0 < A < A_c$. The chemical potential $\mu$ in the Cahn--Hilliard equation \eqref{CH1}--\eqref{CH2} is defined as the variational derivative of the mixing energy $E_{\mathrm{mix}}$ with respect to $\phi$ (subject to the homogeneous Neumann boundary condition $\partial_\mathbf{n}\phi=0$). Its gradient provides a driving force for the diffusion mechanism \cite{phasefield2,CHreview}. The heat equation \eqref{BOUSSINESQ} gives a simplification of the thermal energy equation. In general, temperature dependent extensions of the Cahn--Hilliard equation for phase separation processes are rather involved, since different transport properties of the temperature and relations of the free energy can lead to different dynamical models, see \cite{GL2015,Liuextend,Wang2020,Sun2024} and the references therein.

In this study, we allow physical coefficients such as fluid viscosity $\nu$, diffusion mobility $m$ and thermal diffusivity $\kappa$ to depend both on the order parameter $\phi$ and temperature $\theta$. The temperature dependence of the fluid viscosity and thermal diffusivity is physically important in the investigation of the detailed motion in certain non-isothermal flows, see, for instance, \cite{LB96} and the references therein. On the other hand, for multi-phase fluids, the dependence of structural coefficients on the mixture composition is physically meaningful and should be taken into account in order to accurately model the complex fluid dynamics, since different components may have distinct physical properties, see, e.g., \cite{EnVarA1,viscosity-phi}.

Several simplified versions of the coupled system \eqref{NSCHM}--\eqref{initial} have been studied in the literature. For instance, assuming matched densities, constant mobility, and only temperature dependence of the viscosity and thermal diffusivity, in \cite{me} the author proved the existence and uniqueness of global weak/strong solutions to the initial boundary value problem in two dimensions. When the Marangoni effect is neglected, still assuming matched densities and constant coefficients, problem \eqref{NSCHM}--\eqref{initial} reduces to the so-called Cahn--Hilliard--Boussinesq system, which has been extensively analyzed. We refer to \cite{Zhao_regularity,MBoussinesq,Zhao_longtime1,Zhao_longtime2} for results on well-posedness, regularity, and longtime behavior of solutions. Research on related models such as the Cahn--Hilliard--Oberbeck--Boussinesq system can be found in \cite{Oberbeck}, where weak and very weak solutions in two dimensions were investigated. When the Cahn--Hilliard equation is replaced by a second order Allen--Cahn equation (with a regular potential like $W(\phi)=\frac14(\phi^{2}-1)^2$), existence of global weak solutions, existence and uniqueness of strong solutions and long-time behavior of globally bounded solutions have been analyzed under suitable assumptions on the structural data, see \cite{HWUXX2013,HW2017,lopes-AC,lopes18} for detailed discussions. In addition, the local-in-time existence of weak solutions was established in \cite{lopes} for the Allen--Cahn variant with a temperature-dependent interfacial thickness. Recently, the authors in \cite{Abels2025} investigated a non-isothermal Navier--Stokes--Allen--Cahn system for  incompressible two-phase flows of equal densities. Despite the high nonlinearity in the thermodynamically consistent heat equation, they established local well-posedness of the initial-boundary value problem and the existence of global ``entropic'' weak solutions. In addition, they analyzed the sharp interface limit, showing that an entropic weak solution to the phase-field model converges to a distributional (or BV) solution to a non-isothermal Navier--Stokes/mean curvature flow.

All of the aforementioned works assumed matched densities for the fluid components and sometimes adopted the Boussinesq approximation for the thermally induced variation of the density. On the other hand, when the thermal effect is neglected, problem \eqref{NSCHM}--\eqref{initial} reduces to the Abels--Garcke--Gr\"{u}n system (AGG for short, see \cite{AGGmodel}), which is a thermodynamically consistent and frame-indifferent diffuse interface model for isothermal viscous incompressible two-phase flows with unmatched densities. We refer to \cite{JMFM} for the existence of global weak solutions in two and three dimensions for a non-degenerate mobility $m=m(\phi)$. The corresponding result in the case of degenerate mobility was established in \cite{AGG-P}. Concerning local strong well-posedness in three dimensions, we refer to \cite{AGG3dstrong}, see also \cite{AW21} for the case with a regular potential. When the spatial dimension is two, the existence of strong solutions locally in time for bounded domains and globally in time for periodic boundary conditions was obtained in \cite{AGG2dstrong}. Finally, we refer to \cite{longtime_ann} for global regularity and asymptotic stabilization of global weak solutions in three dimensions as well as the existence of global strong solutions in two-dimensional bounded domains. Extensions to multi-component systems can be found in the recent work \cite{mobility-multi-phase}.

To the best of our knowledge, there is no analytical result on the hydrodynamic system \eqref{NSCHM}--\eqref{initial}, which extends the isothermal AGG model by coupling with the energy transport equation \eqref{BOUSSINESQ} and incorporating thermal effects on both density (buoyancy forces) and surface tension (Marangoni effects). In this study, we establish the following results for problem \eqref{NSCHM}--\eqref{initial} in the general scenario with unmatched densities and variable viscosity, mobility, as well as thermal diffusivity:
\begin{itemize}
 \item[(1)] the existence of global weak solutions in both two and three dimensions, which are uniformly bounded in time (see Theorem \ref{weaksolution-3d});
 \item[(2)] the uniqueness of global weak solutions in two dimensions, under suitable assumptions on the structural data (see Theorem \ref{weaksolution-2d}).
\end{itemize}

The problem \eqref{NSCHM}--\eqref{initial} satisfies two fundamental properties, i.e., mass conservation and energy balance, which serve as the basis for the subsequent analysis. Integrating \eqref{CH1} over $\Omega$, using the boundary conditions \eqref{boundary} and the incompressibility condition \eqref{INCOMPRESS}, after integration by parts, we obtain
\begin{align}
\frac{\mathrm{d}}{\mathrm{d}t} \int_\Omega\phi \,\mathrm{d}x =0,\quad \forall\, t>0,\label{mass}
\end{align}
which implies that the mass of the binary fluids is conserved for all time.
Next, for sufficiently smooth solutions, multiplying \eqref{NSCHM} by $\bm{u}$, \eqref{CH1} by $\lambda_0 a \mu$ and \eqref{BOUSSINESQ} by $\theta$, integrating over $\Omega$ (assuming for simplicity here $\theta_\mathrm{b}=0$), and adding the resultants together, we formally arrive at the following energy identity
\begin{align}
&\frac{\mathrm{d}}{\mathrm{d} t} \mathcal{E}_{\mathrm{tot}} (\bm{u}, \phi, \theta) + \int_\Omega \big( 2\nu(\phi, \theta)|D\bm{u}|^2 + \lambda_0 am(\phi,\theta)|\nabla \mu|^2 + \kappa(\phi, \theta)|\nabla \theta|^2\big)\,\mathrm{d}x
\notag \\
&\quad = \int_\Omega \mathrm{div}\big[\lambda_0 b\theta(\nabla \phi\otimes \nabla \phi)\big]\cdot \bm{u}\,\mathrm{d}x
+ \int_\Omega \bm{f}_{\mathrm{b}}(\phi, \theta)\cdot \bm{u}\,\mathrm{d}x,\quad \forall\, t>0,
\label{energy-a}
\end{align}
where the total energy $\mathcal{E}_{\mathrm{tot}}$ is given by
\begin{align}
\label{energy-b}
\mathcal{E}_{\mathrm{tot}} (\bm{u}, \phi, \theta)
=\int_\Omega \left[\frac12\rho(\phi)|\bm{u}|^2 + \lambda_0 a\left(\frac12|\nabla \phi|^2+ W(\phi)\right) + \frac12 |\theta|^2\right]\,\mathrm{d}x.
\end{align}
We observe that the Marangoni-driven capillary term $\mathrm{div} (\lambda(\theta) \nabla \phi \otimes \nabla \phi)$ introduces a highly nonlinear coupling between the temperature and the phase-field variable, which significantly complicates the analysis. In view of \eqref{energy-a}, the Marangoni effect and the buoyancy force cause the loss of a standard energy dissipation structure for the system, compared to the corresponding isothermal case \cite{AGGmodel}. This poses significant difficulties in obtaining global weak solutions that are expected to be uniformly bounded in time. In addition, though physically meaningful, the nonhomogeneous Dirichlet boundary condition for the temperature introduces additional obstacles to the derivation of the necessary estimates. To achieve our goal, we adopt an implicit-explicit time-discretization scheme that maintains the uniform boundedness of the temperature along time steps (see \eqref{time-discretization-u}--\eqref{time-discretization-theta}). This discretization scheme is inspired by \cite{JMFM} for the isothermal AGG model. By working directly with a singular potential like \eqref{Wphi}, it also keeps the physical range of the phase-field variable, i.e., $\phi\in [-1,1]$. The uniform $L^{\infty}$-bounds of the approximate solutions $(\phi,\theta)$ help us to handle some troublesome terms that are highly nonlinear. Furthermore, by exploring the coupling structure of the system, we derive a suitable discrete energy inequality that enables us to control the energy growth and obtain uniform-in-time estimates for the approximate solutions (see \eqref{discrete-energy-inequality}). Our analysis does not require any assumption of smallness on the initial temperature, thereby improving the conditions in \cite{HWUXX2013,HW2017} for the related Navier--Stokes--Allen--Cahn system (with matched densities). In the two-dimensional setting, we obtain further regularity property of $\theta$ using the H\"older estimate \eqref{GiorginiHolder} (cf. \cite{me}). Based on this improved regularity property, we can establish the uniqueness of global weak solutions by assuming that the fluid mixture has matched densities, the mobility depends only on the concentration, and the thermal diffusivity depends solely on the temperature. Whether uniqueness holds under more general assumptions remains an open question; we refer the readers to Remarks \ref{uniqueness-AGG} and \ref{weaksolution-2d'} for detailed discussions.

Building on the existence of global weak solutions, it is natural to study the long-time behavior. We expect that in three dimensions, results similar to those in \cite{longtime_ann} can be achieved (i.e., the longtime separation property, eventual regularity, weak-strong uniqueness, and convergence to a single equilibrium), provided that the mobility is assumed to be a positive constant. Moreover, in light of the recent contribution \cite{GP2025}, we can apply the method therein to prove that every global weak solution converges to a single equilibrium as $t \to \infty$ in the general case of a non-degenerate and non-constant mobility. These issues will be addressed in a forthcoming study.

\smallskip
\textit{Plan of the paper.} In Section \ref{main-res}, we present the notation, mathematical tools, and main results. Section \ref{section-of-weak-3d} is devoted to the proof of the existence of a global weak solution in both two and three dimensions. In Section \ref{section-of-weak-2d}, we prove the existence and uniqueness of global weak solutions with improved regularity for the temperature in two dimensions.

\section{Main Results}
\label{main-res}

\subsection{Notation}
\label{Preliminaries}

We first introduce the notation and conventions used throughout this work.

Let $B$ be a Banach space whose norm is denoted by $\|\cdot\|_{B}$. We denote its dual by $B^{\prime}$ and the dual product by $\langle \cdot, \cdot \rangle_{B}$. If $B$ is a Hilbert space, we write $(\cdot,\cdot)_B$ for the associated inner product. The bold letter $\bm{B}$ denotes the generic space of vectors or matrices, with each component belonging to $B$.
Let $I \subset [0, \infty]$ be an open set. We denote by $L^p(I; B)$ the space that consists of Bochner measurable $p$-integrable functions (if $p \in [1,\infty)$) or essentially bounded functions (if $p =\infty$). In the case $I=(t_1, t_2)$, we simply write $L^p(t_1,t_2;B)$.
Moreover, $L^p_{\mathrm{uloc}}([t_0,\infty);B)$ denotes the uniformly local variant of $L^p(t_0,\infty;B)$ consisting of all strongly measurable functions $f:[t_0,\infty)\to B$ such that
$$
\|f\|_{L^p_{\mathrm{uloc}}([t_0,\infty);B)}=\sup_{t\geqslant t_0}\|f\|_{L^p(t,t+1;B)}<\infty.
$$
For any $T\in (t_0,\infty)$, we simply have $L^p_{\mathrm{uloc}}([t_0,T);B)= L^p(t_0,T;B)$. For $p\in [1,\infty]$, $T\in (0,\infty)$, $f\in W^{1,p}(0,T;B)$ if and only if $f$, $\frac{\mathrm{d}}{\mathrm{d}t} f\in  L^{p}(0,T;B)$, where $\frac{\mathrm{d}}{\mathrm{d}t} f$ denotes the vector-valued distributional derivative of $f$. The uniformly local space
$W^{1,p}_{\mathrm{uloc}}([t_0,\infty);B)$ is defined by replacing $L^p(0,T; B)$ with
$L^p_{\mathrm{uloc}}([t_0,\infty);B)$. If $p=2$, we simply set $H^1(0,T;B) = W^{1,2}(0,T;B)$ and
$H^1_{\mathrm{uloc}}([t_0,\infty);B) = W^{1,2}_{\mathrm{uloc}}([t_0,\infty);B)$. Let $I=[0,T]$ if $T \in (0,\infty)$ or $I=[0,\infty)$ if $T=\infty$. Then $C(I;B)$ denotes the Banach space of all continuous functions $f:I \to B$ equipped with the supremum norm. We also denote by $C_{w}(I;B)$ the topological vector space of all weakly continuous functions $f:I \to B$. In addition, we denote by $C^\infty_0(0, T; B)$ the abstract vector space of all smooth functions $f:(0, T)\to B$ with $\mathrm{supp} f\subset\subset (0,T)$. Given any two matrices $X=(X_{ij})_{i,j=1}^{d}$, $Y=(Y_{ij})_{i,j=1}^{d}\in \mathbb{R}^{d\times d}$, we have
$(XY)_{ik}=\sum_{j=1}^dX_{ij}Y_{jk}$. Then we denote the Frobenius inner product by
$X:Y=\mathrm{trace}(Y^TX)=\sum_{i,k=1}^dX_{ik}Y_{ik}$ and its associated norm $|X|=\sqrt{X:X}$.

Let $\Omega \subset \mathbb{R}^d$, $d\in \{2,3\}$, be a bounded domain with sufficiently smooth boundary $\partial \Omega$. We write $C(\overline{\Omega})$ (or $C^{\alpha}(\overline{\Omega})$ with $\alpha\in (0,1)$) for the set of continuous (or $\alpha$-H\"{o}lder continuous) functions defined in the closure of $\Omega$. The symbol $C_0^{\infty}(\Omega)$ refers to the space of functions that are differentiable infinitely many times
and compactly supported in $\Omega$.
For the standard Lebesgue and Sobolev spaces in $\Omega$, we use the notation $L^{p}(\Omega)$, $W^{k,p}(\Omega)$ for any $p \in [1,\infty]$ and $k\in \mathbb{N}$,
equipped with the corresponding norms $\|\cdot\|_{L^{p}(\Omega)}$, $\|\cdot\|_{W^{k,p}(\Omega)}$, respectively. Besides, $W^{k,p}_0(\Omega)$ represents the closure of $C^\infty_0(\Omega)$ in $W^{k,p}(\Omega)$, and $W^{-k,p^\prime}(\Omega) = (W^{k,p}_0(\Omega))^\prime$
refers to the corresponding dual space, where $p^\prime$ is the dual exponent of $p$.
When $p = 2$, these spaces are Hilbert spaces and we use the standard conventions $H^{k}(\Omega) := W^{k,2}(\Omega)$, $H^{k}_0(\Omega) := W^{k,2}_0(\Omega)$ and $H^{-1}(\Omega)=(H^1_0(\Omega))'$.
For $s\geqslant 0$ and $p \in [1,\infty)$, we denote by $H^{s,p}(\Omega)$ the Bessel-potential spaces and by $W^{s,p}(\Omega)$ the Slobodeckij spaces. We have $H^{s,2}(\Omega) = W^{s,2}(\Omega)$
for all $s$, but for $p\neq 2$ the identity $H^{s,p}(\Omega) = W^{s,p}(\Omega)$ is only true if
$s\in \mathbb{N}$. In addition, for $s\in \mathbb{N}$, $H^{s,p}(\Omega)$ and $W^{s,p}(\Omega)$  coincide with the usual Sobolev spaces. The corresponding function spaces on the
boundary $\partial\Omega$ are defined using local charts. For convenience, we define
$$
H_{\mathbf{n}}^2(\Omega) =
\{ f \in H^2(\Omega) : \partial_{ \mathbf{n}} f = 0 \text{ on } \partial \Omega \},
$$
and denote the inner product and norm of $L^2(\Omega)$ by $(\cdot, \cdot)$ and $\|\cdot\|$ without ambiguity.

For every $f\in (H^1(\Omega))^{\prime}$, its generalized mean value over $\Omega$ is defined as
$\overline{f}=|\Omega|^{-1}\langle f,1\rangle_{H^1(\Omega)}$. If $f\in L^1(\Omega)$, then the mean value is given by $\overline{f}=|\Omega|^{-1}\int_\Omega f \,\mathrm{d}x$. For any $M\in \mathbb{R}$, we set
$$
 L^2_{(M)}(\Omega):= \left\{f\in L^2(\Omega):\overline{f} =M\right\}.
$$
The orthogonal projection $P_0: L^2(\Omega)\to L^2_{(0)}(\Omega)$ is defined as $P_0 f= f- \overline{f}$. Besides, we introduce the linear spaces
\begin{align*}
&
V_{(0)} := \left\{f \in H^1(\Omega) : \overline{f} = 0	\right\},
\quad
V_{(0)}^{\prime} := \{f \in (H^1(\Omega))^{\prime} : \overline{f} =0\},
\end{align*}
and recall the well-known Poincar\'{e}--Wirtinger inequality:
\begin{equation*}
\left\|f-\overline{f}\right\|\leqslant C_{PW} \|\nabla f\|,\quad \forall\,
f\in H^1(\Omega),
\end{equation*}
where the positive constant $C_{PW}$ depends only on $\Omega$. Consider the Neumann problem
\begin{align*}
	\begin{cases}
		-\Delta u = f, \quad &\text{in}~\Omega, \\
		\partial_\mathbf{n} u = 0, \quad &\text{on}~\partial \Omega.
	\end{cases}
\end{align*}
We define the solution operator $\G : V_{(0)}^{\prime} \to V_{(0)}$ in the following sense:
for every $f \in V_{(0)}^{\prime}$, $\G f \in V_{(0)}$ is the unique function satisfying
\begin{align}\label{Pre-G}
	\langle f , v \rangle_{H^1(\Omega)} = ( \nabla \G f, \nabla v), \quad \forall\, v \in V_{(0)}.
\end{align}
For $f \in V_{(0)}^{\prime}$, it is easy to check that $\| \nabla \G f \|$ is a norm on $V_{(0)}^{\prime}$ equivalent to the standard one. Thus, we also write it as $\|\cdot\|_{V_{(0)}^{\prime}}$. In order to handle the nonconstant mobility,  we also consider the following Neumann problem
\begin{align}\label{Gq}
	\begin{cases}
		-\mathrm{div} ( m(q) \nabla u) = f, \quad &\text{in}~\Omega, \\
		m(q) \partial_\mathbf{n} u = 0, \quad &\text{on}~\partial \Omega.
	\end{cases}
\end{align}
Here, $q: \Omega \to [-1,1]$ is a measurable function.
Define the solution operator $\G_q$ as
\begin{align}\label{Pre-Gq}
	\langle f , v \rangle_{H^1(\Omega)} = ( m(q) \nabla \G_q f, \nabla v), \quad \forall\, v \in V_{(0)}.
\end{align}
As it has been shown in \cite[Section 2]{New_Results_Mobility}, for $f \in V_{(0)}^{\prime}$, $\| \nabla \G_q f \|$ and $\| \nabla \G f \|$ are equivalent norms to the $(H^1(\Omega))^{\prime}$-norm. Further properties of the operator $\G_q$ can be found in, e.g., \cite{New_Results_Mobility}.
Suppose that $m \in C^1([-1,1])$ and $q \in H^2(\Omega)$. Then for any $f \in L^2_{(0)}(\Omega)$, we find $\G_q f\in H_{\mathbf{n}}^2(\Omega)$ and the following estimate when the spatial dimension is two (see \cite[Equation (2.15)]{New_Results_Mobility}):
\begin{align}\label{GqH2}
	\| \G_q f \|_{H^2(\Omega)} \leqslant
	C \big( \| \nabla q \| \| q \|_{H^2(\Omega)} \| \nabla \G_q f \| + \| f \| \big), \quad d=2.
\end{align}
Here, the positive constant $C$ depends only on $\Omega$ and the function $m$.

Next, we recall some function spaces for the Navier--Stokes equations. Let $\bm{C}^\infty_{0,\sigma}(\Omega)$ be the space of divergence-free vector fields in $(C^\infty_0(\Omega))^d$. We define $\bm{H}_{\sigma}$ and $\bm{V}_{\sigma}$ as the closure of $\bm{C}^\infty_{0,\sigma}(\Omega)$ with respect to the $\bm{L}^2$ and $\bm{H}^1$ norms, respectively. The space $\bm{V}_{\sigma}$ is equipped with the scalar product
 $(\bm{u},\bm{v})_{\bm{V}_{\sigma}}:=(\nabla \bm{u},\nabla \bm{v})$ for all $\bm{u},\, \bm{v}  \in \bm{V}_{\sigma}$ and the norm $\|\bm{u}\|_{\bm{V}_{\sigma}}=\|\nabla \bm{u}\|$. We recall the classical Korn's inequality
\begin{align}
 \|\nabla \bm{u}\| \leqslant \sqrt{2} \|D\bm{u}\|\leqslant \sqrt{2}\|\nabla \bm{u}\|,\quad \forall\, \bm{u}\in \bm{V}_{\sigma}.
 \label{Korn}
\end{align}
Besides, Poincar\'e's inequality entails that
\begin{align}
 \|\bm{u}\| \leqslant C_P \|\nabla \bm{u}\|,\quad \forall\, \bm{u}\in \bm{V}_{\sigma},
 \label{Poin}
\end{align}
where the positive constant $C_P$ depends only on $\Omega$.
It is well known that $\bm{L}^2(\Omega)$ can be decomposed into $\bm{H}_{\sigma}\oplus\bm{G}(\Omega)$, where $\bm{G}(\Omega):=\{\bm{f}\in\bm{L}^2(\Omega): \exists\, g\in H^1(\Omega),\ \bm{f}=\nabla g\}$. Then we introduce the Helmholtz--Leray projection in the space of divergence-free functions $\bm{P}:\bm{L}^2(\Omega)\to \bm{H}_{\sigma}$. Recall  the Stokes operator $\bm{S}: \bm{V}_{\sigma}\cap\bm{H}^2(\Omega)\to\bm{H}_{\sigma}$ such that
\begin{equation}
(\bm{S}\bm{u},\bm{\zeta})=(\nabla \bm{u},\nabla\bm{\zeta}),\quad  \forall\, \bm{\zeta} \in \bm{V}_{\sigma},
\nonumber
 \end{equation}
 with the domain $\mathcal{D}(\bm{S})= \bm{V}_{\sigma}\cap\bm{H}^2(\Omega)$ (see, e.g., \cite[Chapter III]{Sohr}). The operator $\bm{S}$ is a canonical isomorphism from $\bm{V}_{\sigma}$ to $\bm{V}_{\sigma}^{\prime}$ with the inverse $\bm{S}^{-1}:\bm{V}_{\sigma}^{\prime}\to\bm{V}_{\sigma}$. For any $\bm{f}\in \bm{V}_{\sigma}^{\prime}$, there is a unique $\bm{u}=\bm{S}^{-1}\bm{f}\in\bm{V}_{\sigma}$ such that
\begin{equation}
(\nabla(\bm{S}^{-1}\bm{f}),\nabla \bm{\zeta})=\langle\bm{f},\bm{\zeta}\rangle_{\bm{V}_{\sigma}},
\quad \forall\, \bm{\zeta} \in \bm{V}_{\sigma},
\nonumber
\end{equation}
and $\|\nabla(\bm{S}^{-1}\bm{f})\|=\langle\bm{f},\bm{S}^{-1}\bm{f} \rangle_{\bm{V}_{\sigma}}^{\frac{1}{2}}$ is an equivalent norm on $\bm{V}_{\sigma}^{\prime}$.
In addition, there exists a positive constant $C$ such that $\|\bm{u}\|_{\bm{H}^2(\Omega)}\leqslant C\|\bm{S} \bm{u}\|$ for any $\bm{u}\in \mathcal{D}(\bm{S})$.

Throughout this paper, the symbols $C$, $C_i$, $i\in \mathbb{N}$, denote generic positive constants that may depend on coefficients of the system, norms of the initial data, $\Omega$ and time. Their values may change from line to line, and specific dependencies will be pointed out when necessary.

%

\subsection{Statement of results}

First, we introduce the assumptions used in the subsequent analysis.

\begin{itemize}
\item[$\mathbf{(A0)}$] $\Omega \subset \mathbb{R}^d$, $d\in\{2,3\}$, is a bounded domain with $C^3$-boundary $\partial \Omega$.

\item[$\mathbf{(A1)}$] The viscosity $\nu: \mathbb{R}^2 \to \mathbb{R}^+$ belongs to $C^1(\mathbb{R}^2)$ and
\begin{align*}
			\text{there exists}\ \underline{\nu} >0\ \text{such that}\ \  \nu(s_1,s_2)\geqslant \underline{\nu},  \quad  \forall\, (s_1,s_2) \in \mathbb{R}^2.
\end{align*}
\item[$\mathbf{(A2)}$] The mobility $m:\mathbb{R}^2\to\mathbb{R}^+$ belongs to $C^1(\mathbb{R}^2)$ and
\begin{align*}
			\text{there exists}\ \underline{m}>0\ \text{such that}\ \ m(s_1,s_2)\geqslant \underline{m}, \quad  \forall\,  (s_1,s_2) \in \mathbb{R}^2.
\end{align*}
\item[$\mathbf{(A3)}$] The thermal diffusivity $\kappa:\mathbb{R}^2\to\mathbb{R}^+$  belongs to $C^2(\mathbb{R}^2)$ and
\begin{align*}
			\text{there exists}\ \underline{\kappa}>0\ \text{such that}\ \ \kappa(s_1,s_2)\geqslant \underline{\kappa}, \quad  \forall\,  (s_1,s_2) \in \mathbb{R}^2.
\end{align*}
\item[$\mathbf{(A4)}$]
The singular potential $W$ belongs to the class of functions $C([-1,1])\cap C^{2}(-1,1)$ and it can be decomposed into the following form
	\begin{equation}
	W(s)=W_1(s)+W_2(s).\nonumber
	\end{equation}
The singular convex part $W_1$ fulfills
	\begin{equation}
	\lim_{s\to \pm 1} W_1^{\prime}(s)=\pm \infty ,\quad \text{and}\ \  W_1^{\prime \prime}(s)\geqslant c_0,\quad \forall\, s\in (-1,1),\nonumber
	\end{equation}
for some strictly positive constant $c_0$. Besides, we make the extension $W_1(s)=+\infty$ for any $s\notin[-1,1]$. Concerning the regular (possibly concave) part $W_2$, we assume that
$$
\text{there exists}\ c_W >0 \ \ \text{such that}\ \ |W^{\prime \prime}_2(s)|\leqslant
        c_W,\quad \forall\, s\in\mathbb{R}.
$$
\item[$\mathbf{(A5)}$] The averaged density $\rho$ is given by \eqref{density}, the relative flux $\mathbf{J}$ satisfies \eqref{Jflux} and the buoyancy force $\bm{f}_{\mathrm{b}}$ takes the form of \eqref{bouy}. The Cauchy stress tensor $\bm{\sigma}$ is defined as in \eqref{Cauchy_stress_tensor} with the surface tension coefficient $\lambda$ determined by \eqref{Eotvos}. $\rho_1$, $\rho_2$, $\lambda_0$, $a$, $b$, $\alpha$, $g$ are given positive constants.
\end{itemize}

\begin{remark}\label{rem:W}\rm
It is straightforward to check that the logarithmic potential \eqref{Wphi} with continuous extension to the interval $[-1,1]$ satisfies the structural assumption $\mathbf{(A4)}$.
Moreover, for any function $W$ that satisfies $\mathbf{(A4)}$, we can write
\begin{align}
W(s)= F(s)- c_W s^2,\label{modif-W}
\end{align}
with $F(s)= W_1(s)+W_2(s) + c_W s^2$.
Then $F\in C([-1,1])\cap C^{2}(-1,1)$ satisfies
\begin{equation}
	\lim_{s\to \pm 1} F^{\prime}(s)=\pm \infty ,\quad \text{and}\ \  F^{\prime \prime}(s)\geqslant c_0 + c_W > c_W,\quad \forall\, s\in (-1,1).\nonumber
	\end{equation}
We also have $F(s)=+\infty$ for any $s\notin[-1,1]$.
In the subsequent analysis, we shall always use the reformulation \eqref{modif-W} for any given potential $W$ that satisfies  $\mathbf{(A4)}$.
\end{remark}

\begin{remark}\rm
Since both the phase-field variable $\phi$ and the temperature $\theta$ are bounded as in the definition of solutions, in $\mathbf{(A1)}$--$\mathbf{(A3)}$, we can alternatively assume that (cf. \cite{LB99,HW2017})
$$
\nu(s_1,s_2)>0, \quad m(s_1,s_2)>0,\quad  \kappa(s_1,s_2)>0, \quad  \forall\, (s_1,s_2) \in \mathbb{R}^2.
$$
\end{remark}

We are now in a position to state the main results.

\begin{theorem}\label{weaksolution-3d}
	Let $d=2,3$. Suppose that the assumptions $\mathbf{(A0)}$--$\mathbf{(A5)}$ are satisfied. For any initial data
    $\u_0 \in \bm{H}_{\sigma}$, $\phi_0 \in H^{1}(\Omega)$, $\|\phi_0\|_{L^{\infty}(\Omega)} \leqslant 1$, $|\overline{\phi_0}|<1$,
	$\theta_0 \in L^{\infty}(\Omega)$ and the boundary datum $\theta_{\mathrm{b}}\in H^\frac52(\partial \Omega)$, problem \eqref{NSCHM}--\eqref{initial} admits a global weak solution $(\u, \phi, \mu, \theta)$ on $[0,\infty)$ such that
	\begin{align*}
		&\u \in L^{\infty}(0, \infty ; \bm{H}_{\sigma})
		\cap
		L^2_{\mathrm{uloc}}([0,\infty) ; \V), \\
		&\phi \in L^{\infty}(0, \infty ; H^1(\Omega))
		\cap
		L^4_{\mathrm{uloc}}([0, \infty); H^2_\mathbf{n}(\Omega))
		\cap
		L^2_{\mathrm{uloc}}([0, \infty) ; W^{2, p}(\Omega))
		\cap
		H^1_{\mathrm{uloc}}([0, \infty) ; (H^1(\Omega))^{\prime}), \\
		&\phi \in L^{\infty}(\Omega \times(0, \infty)),\ \text { with } |\phi(x, t)|<1
		\text { a.e. in } \Omega \times(0, \infty),\\
		&\mu \in L^2_{\mathrm{uloc}}([0, \infty) ; H^1(\Omega)), \quad F'(\phi)\in L^2_{\mathrm{uloc}}([0,\infty); L^p(\Omega)), \\
		&\theta \in L^{\infty}( \Omega \times (0, \infty) )
		\cap
		L^2_{\mathrm{uloc}} ([0, \infty) ; H^1(\Omega))
		\cap
		H^1_{\mathrm{uloc}} ([0, \infty) ; H^{-1}(\Omega)),
	\end{align*}
	where $p \in [2,\infty)$ if $d=2$, and $p=6$ if $d=3$.
	The solution fulfills the following weak formulations
	\begin{align}
		&-\int_0^\infty (\rho \u, \partial_t\v )\,\mathrm{d}t -
		\int_0^\infty [(\rho \u\otimes\u, \nabla \v) +  (\u\otimes \mathbf{J}, \nabla \v)]\,\mathrm{d}t
        +\int_0^\infty (2\nu(\phi,\theta) D \u, D \v) \,\mathrm{d}t \notag \\
		&\quad = \int_0^\infty (\lambda_0 a \mu \nabla \phi, \v)\,\mathrm{d}t
        - \int_0^\infty (\lambda_0 b \theta (\nabla \phi \otimes \nabla \phi), \nabla \v)\,\mathrm{d}t
        + \int_0^\infty (\bm{f}_{\mathrm{b}}(\phi, \theta), \v )\,\mathrm{d}t,
		\label{weaku} \\
		&-\int_0^\infty ( \phi, \partial_t \zeta )\,\mathrm{d}t -
		\int_0^\infty  (\phi \u)\cdot \nabla \zeta\,\mathrm{d}t +
		\int_0^\infty (m(\phi,\theta) \nabla \mu, \nabla \zeta)\,\mathrm{d}t
        = 0,
		\label{weakphi} \\
		&-\int_0^\infty ( \theta, \partial_t \xi )\,\mathrm{d}t -
		\int_0^\infty (\theta \u )\cdot \nabla \xi\,\mathrm{d}t +
		\int_0^\infty(\kappa(\phi,\theta) \nabla \theta, \nabla \xi)\,\mathrm{d}t
        = 0,
		\label{weaktheta}
	\end{align}
	for all test functions $\v \in C_{0}^{\infty}(0,\infty; \bm{C}^\infty_{0,\sigma}(\Omega))$,
	$\zeta\in C_{0}^{\infty}(0,\infty;C^\infty(\overline{\Omega}))$,
    and $\xi \in C_{0}^{\infty}(0,\infty;C^\infty_0(\Omega))$.
    Moreover, it holds
	\begin{align}\label{weakmu}
		\mathbf{J}=-\frac{\rho_2-\rho_1}{2}m(\phi,\theta)\nabla \mu\quad\text{and}\quad
        \mu = - \Delta \phi + W^{\prime}(\phi), \quad \text{ a.e. in } \Omega\times (0,\infty).
	\end{align}
	The initial conditions \eqref{initial} are satisfied almost everywhere in $\Omega$ and the boundary condition $\theta=\theta_{\mathrm{b}}$ is satisfied almost everywhere on $\partial\Omega\times (0,\infty)$. Furthermore, for almost all $t\geqslant 0$, the following estimate holds for $\theta$:
	\begin{align}\label{maximumprinciple}
	\min\Big\{\operatorname*{ess\,inf}_{\Omega}\theta_0,\,\min_{\partial\Omega}\theta_{\mathrm{b}}\Big\} \leqslant \theta(x,t) \leqslant \max\Big\{\operatorname*{ess\,sup}_{\Omega}\theta_0,\,\max_{\partial\Omega}\theta_{\mathrm{b}}\Big\},\quad \text{a.e. in}\ \overline{\Omega}.
	\end{align}
\end{theorem}

\begin{remark}\rm
Indeed, we have
$$
\u \in C_w([0, \infty) ; \bm{H}_{\sigma}),\quad \phi\in C_w([0, \infty) ; H^1(\Omega)), \quad \theta\in C([0, \infty) ; L^2(\Omega))
$$
so that the initial data can be attained.
\end{remark}

\begin{remark}\rm
In the weak formulation \eqref{weaku}, the vectorial term $\mathrm{div} ( \lambda(\theta) ( \frac{1}{2}|\nabla \phi|^2 + W(\phi))  \mathbb{I}_d )$ is absorbed into the pressure, and thus vanishes after being tested by the divergence-free function $\v$.
Furthermore, we rewrite the term $\int_{\Omega} \lambda(\theta) (\nabla \phi \otimes \nabla \phi): \nabla \v \dx $ using the definition \eqref{Eotvos} and the fact $\mathrm{div}\,\v=0$ such that
	\begin{align*}
		\int_{\Omega} \lambda(\theta) (\nabla \phi \otimes \nabla \phi): \nabla \v \dx
        &=
		\lambda_0 a (\nabla \phi \otimes \nabla \phi, \nabla \v) - \lambda_0 b (\theta \nabla \phi \otimes \nabla \phi, \nabla \v)
        \\
		&= -\lambda_0 a \left(\nabla \phi \Delta \phi + \nabla \frac{|\nabla \phi|^2}{2}, \v\right)
		- \lambda_0 b (\theta \nabla \phi \otimes \nabla \phi, \nabla \v)
        \\
		&= \lambda_0 a (-\Delta \phi \nabla \phi + W^{\prime}(\phi) \nabla \phi, \v)
		- \lambda_0 b (\theta \nabla \phi \otimes \nabla \phi, \nabla \v)
        \\
		&= \lambda_0 a (\mu \nabla \phi, \v)
		- \lambda_0 b (\theta \nabla \phi \otimes \nabla \phi, \nabla \v).
	\end{align*}
On the right-hand side of the above equality, the first term corresponds to the usual Korteweg force (see, e.g., \cite{phasefield1}), and the second term arises due to the thermo-induced Marangoni effect (see, e.g., \cite{EnVarA1,EnVarA2}).
\end{remark}

\begin{remark}\rm
We note that the regularity assumption on the boundary data $\theta_{\mathrm{b}}$ can be relaxed. Here, we assume $\theta_{\mathrm{b}} \in H^{\frac52}(\partial \Omega)$ for the purpose of the estimate in \eqref{need-H5/2}, and this condition can be weakened to $\theta_{\mathrm{b}} \in H^{\frac32}(\partial \Omega)$ by a standard density argument. This is possible because the $C^3$-regularity of the boundary $\partial \Omega$ ensures that any $\theta_{\mathrm{b}} \in H^{\frac32}(\partial \Omega)$ is the limit of a sequence $\{\theta_{\mathrm{b}}^k\} \subset H^{\frac52}(\partial \Omega)$ (see, e.g., \cite{Sobolev-manifold}). By solving the extension problem \eqref{Theta_b} for each approximation $\theta_{\mathrm{b}}^k$ and then passing to the limit, all subsequent arguments follow with straightforward adaptations.
\end{remark}
\medskip

Next, in the two-dimensional case, we establish the uniqueness of global weak solutions under the assumptions that the fluid mixture has matched densities, the thermal diffusivity depends only on the temperature, the mobility depends only on the concentration, and the initial temperature is more regular.

\begin{theorem}\label{weaksolution-2d}
Let $d=2$. Suppose that the assumptions $\mathbf{(A0)}$--$\mathbf{(A5)}$ are satisfied.
For any initial data
	$\u_0 \in \H_\sigma$,
	$\phi_0 \in H^{1}(\Omega)$,
	$\|\phi_0\|_{L^{\infty}(\Omega)} \leqslant 1$,
	$|\overline{\phi_0}|<1$,
	$\theta_0 \in C^{\gamma}(\overline{\Omega}) \cap H^1(\Omega)$ with $\gamma \in (0,1)$ and $\theta_0|_{\partial\Omega}=\theta_{\mathrm{b}}\in H^\frac32(\partial \Omega)$, problem \eqref{NSCHM}--\eqref{initial} admits a global weak solution $(\u, \phi, \mu, \theta)$ on $[0,\infty)$ satisfying the additional regularity properties (cf. Theorem \ref{weaksolution-3d}):
	\begin{align*}
		\theta \in L^{\infty}(0, \infty ; H^1(\Omega) \cap C^{\beta}(\overline{\Omega}) )
		\cap
		L^2_{\mathrm{uloc}} ([0, \infty) ; H^2(\Omega))
		\cap
		H^1_{\mathrm{uloc}} ([0, \infty) ; L^2(\Omega)),
	\end{align*}
	for some $ \beta \in (0, \gamma]$. Assume, in addition,
	$$ \rho_1=\rho_2, \quad
	m(\phi, \theta) \equiv m(\phi), \quad
	\kappa(\phi, \theta) \equiv \kappa(\theta).
	$$
Then the global weak solution is unique.
\end{theorem}
\smallskip

\begin{remark}\rm \label{uniqueness-AGG}
In the general case of unmatched densities, the uniqueness of weak solutions in two dimensions remains an open question even without the coupling of thermal effects (see \cite{AGG2dstrong,longtime_ann} for the AGG model).
\end{remark}

\begin{remark}\rm \label{weaksolution-2d'}
In the case where mobility and thermal diffusivity depend on both $\phi$ and $\theta$, we need the following additional regularity property to ensure uniqueness:
$$
\theta_1 \in L_{\mathrm{uloc}}^4([0, \infty); H^2(\Omega) ),
$$
where $(\u_1, \phi_1, \mu_1, \theta_1)$, $(\u_2, \phi_2, \mu_2, \theta_2)$ are global weak solutions obtained in Theorem \ref{weaksolution-2d}.
However, it is not known whether $\theta_1$ can satisfy this condition.
\end{remark}

\section{Existence of Global Weak Solutions}
\label{section-of-weak-3d}

In this section, we prove Theorem \ref{weaksolution-3d} on the existence of global weak solutions to problem \eqref{NSCHM}--\eqref{initial}.
The proof is based on an implicit-explicit time-discretization scheme in the spirit of \cite{JMFM,Wangxiaoming}, with suitable modifications according to new structures of the coupled system. Without loss of generality, we present the argument for the three-dimensional case while pointing out necessary changes in two dimensions.

\subsection{Preliminaries}
To begin with, let us recall some properties of subgradients related to the Ginzburg--Landau free energy (see \cite{Abels2009, Abels2007}).
Concerning the convex part of the free energy (cf. Remark \ref{rem:W})
\begin{align}
E(\phi) =
\int_{\Omega}
\left( \frac{1}{2} |\nabla \phi|^{2} + F(\phi) \right) \mathrm{d}x,
\label{G-L}
\end{align}
we define
\begin{align}\label{convex-part-operator}
	\left.\widetilde{E}(\phi)=
	\left\{\begin{array}{ccc}
		E(\phi),
        &\text{for } \phi \in \D(\widetilde{E}),\\
        +\infty,
        &\text{else},
        \end{array}\right.
	\right.
\end{align}
with the effective domain
$$
\D(\widetilde{E}) := \left\{ \phi \in H^1(\Omega)\ :\ -1 \leqslant \phi \leqslant 1 \text{ a.e. in } \Omega \right\}.
$$
Then we define the subgradient
$$
\partial \widetilde{E}(\phi) = - \Delta \phi + F^{\prime}(\phi),
$$
with the domain
$$
\mathcal{D}(\partial\widetilde{E})
=\left\{\phi \in H^{2}_\mathbf{n}(\Omega)\ : \
F^{\prime}(\phi)\in L^{2}(\Omega),\quad F^{\prime\prime}(\phi)|\nabla \phi|^{2}\in L^{1}(\Omega)\right\}.
$$
Similarly, for any given $M\in (-1,1)$, we define
\begin{align}\label{convex-part-operator-M}
	\left.\widetilde{E}_M(\phi)=
	\left\{\begin{array}{ccc}
		E(\phi),
        &\text{for } \phi \in \D(\widetilde{E}_M),\\
        +\infty,
        &\text{else},
        \end{array}\right.
	\right.
\end{align}
with the effective domain
$$
\D(\widetilde{E}_M) := \left\{ \phi \in H^1(\Omega)\cap L^2_{(M)}(\Omega)\ :\ -1 \leqslant \phi \leqslant 1 \text{ a.e. in } \Omega \right\},
$$
as well as the subgradient
$$
\partial \widetilde{E}_M(\phi) = - \Delta \phi + P_0(F^{\prime}(\phi)),
$$
with the domain
$$
\mathcal{D}(\partial\widetilde{E}_M)
=\left\{\phi \in H^{2}_\mathbf{n}(\Omega)\cap L^2_{(M)}(\Omega)\ : \
F^{\prime}(\phi)\in L^{2}(\Omega),\quad F^{\prime\prime}(\phi)|\nabla \phi|^{2}\in L^{1}(\Omega)\right\}.
$$
According to \cite[Theorem 5]{Abels2009},
the following estimates hold
\begin{align}
& \| \phi \|_{H^2(\Omega)}^2 + \| F^{\prime}(\phi) \|^2 +
\int_{\Omega} F^{\prime \prime}(\phi) | \nabla \phi |^2 \dx \leqslant
C \left( \| \partial \widetilde{E}(\phi) \|^2 + \| \phi \|^2 + 1 \right),
\label{es-pE-1}
\end{align}
for all $ \phi\in \D(\partial \widetilde{E})$, and
\begin{align}
& \| \phi \|_{H^2(\Omega)}^2 + \| F^{\prime}(\phi) \|^2 +
\int_{\Omega} F^{\prime \prime}(\phi) | \nabla \phi |^2 \dx \leqslant
C \left( \| \partial \widetilde{E}_M(\phi) \|^2 + \| \phi \|^2 + 1 \right),
\label{es-pE-2}
\end{align}
for all $\phi\in \D(\partial \widetilde{E}_M)$. The positive constant $C$ in \eqref{es-pE-1} (resp. \eqref{es-pE-2}) is independent of $\phi\in \D(\partial\widetilde{E})$ (resp. $\phi\in \D(\partial\widetilde{E}_M)$).

Next, we recall some results for the following Neumann problem with singular nonlinearity, which will be useful in the subsequent analysis (see e.g., \cite{Abels2009,Conti2020,HwuHeleshaw,inequality}):
\begin{align}
\begin{cases}
- \Delta \phi + F^{\prime}(\phi) = f,\quad \text{in}\ \Omega,\\
\partial_{\mathbf{n}} \phi=0,\qquad \qquad\quad \ \text{on}\ \partial\Omega.
\end{cases}
\label{sing-phi}
\end{align}
\begin{lemma}\label{lem:sing-phi}
Let $\Omega$ be a bounded $C^2$-domain in $\mathbb{R}^d$, $d=2,3$. Assume that $F$ is determined by \eqref{modif-W} with the corresponding assumptions.

\begin{enumerate}[label=\textup{(\arabic*)}]
\item For any $f\in L^2(\Omega)$, problem \eqref{sing-phi} admits a unique strong solution $\phi\in H^2_{\mathbf{n}}(\Omega)$ with $F^{\prime}(\phi)\in L^2(\Omega)$, satisfying the equation $- \Delta \phi + F^{\prime}(\phi) = f$ almost everywhere in $\Omega$. In particular, $\|\phi\|_{L^\infty(\Omega)}\leqslant 1$.

\item Assume, in addition, $f\in L^p(\Omega)$ with $p \in [2,\infty)$ if $d=2$ and $p\in [2,6]$ if $d=3$, then we have
\begin{align}
& \|\phi\|_{W^{2, p}(\Omega)}+ \|F^{\prime}(\phi)\|_{L^p(\Omega)} \leqslant C(p)(1+\|f\|_{L^p(\Omega)}),
\label{es-wpphi}
\end{align}
where the positive constant $C(p)$ depends on $\Omega$ and $p$.
\end{enumerate}
\end{lemma}
%

\subsection{The implicit-explicit time-discretization scheme}

Let us now present the implicit-explicit time-discretization scheme for problem \eqref{NSCHM}--\eqref{initial}.
In order to handle the nonhomogeneous Dirichlet boundary condition for $\theta$, we choose $\Theta_{\mathrm{b}}$ as the harmonic extension of the boundary datum $\theta_{\mathrm{b}}$, which is inspired by e.g., \cite[Section 2.4]{MBoussinesq} and \cite[Section 2]{LB99}:
\begin{equation}
\begin{cases}
-\Delta \Theta_{\mathrm{b}}=0,\quad \text{in}\ \Omega,\\
\Theta_{\mathrm{b}}=\theta_{\mathrm{b}},\quad\quad \, \text{on}\ \partial\Omega.
\end{cases}
\label{Theta_b}
\end{equation}
Since $\theta_{\mathrm{b}}\in H^\frac{5}{2}(\partial\Omega)$, we find
$$
\Theta_{\mathrm{b}}\in H^3(\Omega)\quad \text{and}\quad
\|\Theta_{\mathrm{b}}\|_{H^3(\Omega)}\leqslant C\|\theta_{\mathrm{b}}\|_{H^\frac{5}{2}(\partial\Omega)},
$$
for some constant $C>0$ depending only on $\Omega$. Introducing the new variable
$$\vartheta=\theta - \Theta_{\mathrm{b}},$$
we can reformulate the heat equation \eqref{BOUSSINESQ} as follows (cf. \cite{LB99}):
\begin{align}
& \partial_t \vartheta + \u \cdot \nabla \vartheta - \mathrm{div}\,(\kappa(\phi,\vartheta+ \Theta_{\mathrm{b}}) \nabla \vartheta)
= -\u \cdot \nabla \Theta_{\mathrm{b}} + \mathrm{div}\,(\kappa(\phi,\vartheta+ \Theta_{\mathrm{b}}) \nabla \Theta_{\mathrm{b}}).
\label{Eq:vartheta}
\end{align}
Moreover, $\vartheta$ satisfies the following homogeneous Dirichlet boundary condition and the initial condition
\begin{align}
\vartheta=0\quad \text{on}\ \partial\Omega\times (0,T),\qquad \vartheta|_{t=0}=\vartheta_0:=\theta_0- \Theta_{\mathrm{b}}\quad \text{in}\ \Omega.
\label{vartheta:bdini}
\end{align}
Due to the Sobolev embedding theorem $H^2(\Omega)\hookrightarrow L^\infty(\Omega)$ for $d=2,3$, we have $\vartheta_0\in L^\infty(\Omega)$ since $\theta_0\in L^\infty(\Omega)$ and $\Theta_{\mathrm{b}}\in H^2(\Omega)$.

For any given positive integer $N$, we take the time step as $h=\frac{1}{N}$.
Next, for every $k \in \mathbb{N}$, let $\u^k\in \H_\sigma$, $\phi^k\in H^2_\mathbf{n}(\Omega)$ with $\phi^k\in [-1,1]$ almost everywhere in $\Omega$, $\overline{\phi^k}\in (-1,1)$, and $\vartheta^k\in H^2(\Omega) \cap  H^1_0(\Omega)$. Define $\rho^k=\rho(\phi^k)$ by \eqref{density}. Then, we look for a quadruple
$$
(\u, \phi, \mu, \vartheta)
=\big(\u^{k+1}, \phi^{k+1}, \mu^{k+1}, \vartheta^{k+1}\big)
\in \bm{V}_\sigma \times \mathcal{D}(\partial\widetilde{E}) \times H^2_{\mathbf{n}}(\Omega) \times ( H^2(\Omega)\cap H^1_0(\Omega))
$$
as a solution of the following time-discrete system at the time step $k+1$:
\begin{align}
	&\left( \frac{\rho\u - \rho^k\u^k}{h} , \v \right) +
	\big(\mathrm{div}(\rho^k \u \otimes \u), \v\big) + (\mathrm{div}(\u\otimes \mathbf{J}), \v)
    \notag \\
	&\qquad  +
	\big(2 \nu(\phi^k, \vartheta^k+ \Theta_{\mathrm{b}}) D \u , D \v\big) - \lambda_0 a (\mu \nabla \phi^k, \v)
    \notag \\
    &\quad = -
	\lambda_0 b \big((\vartheta^k+ \Theta_{\mathrm{b}}) (\nabla \phi \otimes \nabla \phi), \nabla \v\big) +
	\big(\bm{f}_{\mathrm{b}}(\phi^k, \vartheta^k+ \Theta_{\mathrm{b}}), \v\big), \quad &&\forall\, \v \in \V,
\label{time-discretization-u}\\
    & \mathbf{J}= -\frac{\rho_2-\rho_1}{2}m(\phi^k,\vartheta^k+ \Theta_{\mathrm{b}})\nabla \mu, \quad &&\text{a.e. in } \Omega,
    \label{time-discretization-J}\\
	&  \frac{\phi - \phi^k}{h}  +
	 \u \cdot \nabla \phi^k  = \mathrm{div}(m(\phi^k,\vartheta^k+ \Theta_{\mathrm{b}}) \nabla \mu), \quad &&\text{a.e. in } \Omega,
\label{time-discretization-phi} \\
	& \mu  + c_W (\phi+\phi^k) = - \Delta \phi + F^{\prime}(\phi), \quad &&\text{a.e. in } \Omega,
\label{time-discretization-mu} \\
	&  \frac{\vartheta - \vartheta^k}{h}  +
	  \u \cdot \nabla \vartheta - \mathrm{div}  ( \kappa(\phi^k,\vartheta^k+ \Theta_{\mathrm{b}}) \nabla \vartheta)
\notag \\
&\quad =  - \u \cdot \nabla \Theta_{\mathrm{b}} +  \mathrm{div} ( \kappa(\phi^k,\vartheta^k+ \Theta_{\mathrm{b}}) \nabla \Theta_{\mathrm{b}}), \quad &&\text{a.e. in } \Omega.
\label{time-discretization-theta}
\end{align}
\begin{remark}\rm
In the subsequent analysis, we shall use the notation $\theta, \theta^k$ such that
\begin{align}
\theta= \vartheta +\Theta_{\mathrm{b}}\quad \text{and}\quad \theta^k= \vartheta^k +\Theta_{\mathrm{b}}.
\label{notation-theta}
\end{align}
For the initial time step $k=0$, the given data are determined by $$(\u^0,\phi^0,\vartheta^0)=(\u_0,\phi_0^N,\vartheta_0^N),$$
where $\phi_0^N\in H^2_{\mathbf{n}}(\Omega)$ is a suitable regularization of the initial datum $\phi_0 \in \D(\widetilde{E})$ satisfying $\phi_0^N\in [-1,1]$ almost everywhere in $\Omega$  and $\overline{\phi_0^N}=\overline{\phi_0}$, while $\vartheta_0^N \in  H^2(\Omega)\cap H^1_0(\Omega)$ is a suitable regularization of the initial datum $\vartheta_0\in L^\infty(\Omega)$ with an $L^\infty$-bound independent of $N$. Detailed construction is presented in Section \ref{dis-existence}.
\end{remark}

\begin{remark}\rm
Both explicit and implicit formulations are adopted in the discrete approximation \eqref{time-discretization-u}--\eqref{time-discretization-theta}.
Our aim is two-fold: (1) at each time step, the resulting discrete problem can be solved by a Leray--Schauder fixed point argument, (2) the discrete solution satisfies a sufficiently simple energy inequality that can yield uniform-in-time estimates (see  \eqref{discrete-energy-inequality} below). To this end, all variable-dependent ``coefficients'' are treated explicitly to reduce the order of nonlinearity for unknown variables. Moreover, we use the implicit form $\u\cdot \nabla \theta$ in \eqref{time-discretization-theta}, which enables us to maintain an $L^\infty$-estimate for $\theta$ as well as its shifted variant $\vartheta$ (see Lemma \ref{Stampacchia-truncation-method-remark}). Since we work directly with the singular potential $F$, which guarantees the bound $\|\phi^k\|_{L^\infty(\Omega)}\leqslant 1$ for all $k$ in induction, instead of \eqref{time-discretization-mu} (like in \cite{JMFM}) we can also use the well-known convex-splitting formulation such that $\mu  + 2c_W \phi^k = - \Delta \phi + F^{\prime}(\phi)$.
\end{remark}

Applying the Stampacchia truncation method for equation \eqref{time-discretization-theta}, we obtain the following result:
\begin{lemma}\label{Stampacchia-truncation-method-remark}
Suppose that $\phi^k, \theta^k\in L^\infty(\Omega)$. If $(\u,\vartheta)\in \V \times H^1_0(\Omega)$ is a weak solution to \eqref{time-discretization-theta}, then we have $\theta, \vartheta\in L^\infty(\Omega)$ and the following estimate holds
\begin{align}\label{maximumprinciple-discrete}
	\min\Big\{\operatorname*{ess\,inf}_{\Omega}\theta^k,\,\min_{\partial\Omega}\theta_{\mathrm{b}}\Big\} \leqslant \theta(x) \leqslant \max\Big\{\operatorname*{ess\,sup}_{\Omega}\theta^k,\,\max_{\partial\Omega}\theta_{\mathrm{b}}\Big\},\quad \text{a.e. in}\ \overline{\Omega}.
	\end{align}
\end{lemma}
\begin{proof}
The proof follows an argument similar to that for \cite[Lemma 3.1]{LB96}. We sketch it for the sake of completeness.
The weak solution $\vartheta$ satisfies
\begin{align}
& \left(\frac{\vartheta - \vartheta^k}{h},\xi\right)  +
	  (\u \cdot \nabla \vartheta,\xi) + \left( \kappa(\phi^k,\vartheta^k+ \Theta_{\mathrm{b}}) \nabla \vartheta,\nabla \xi\right)
\notag \\
&\quad =  - (\u \cdot \nabla \Theta_{\mathrm{b}},\xi) -\left( \kappa(\phi^k,\vartheta^k+ \Theta_{\mathrm{b}}) \nabla \Theta_{\mathrm{b}},\nabla \xi\right), \quad  \forall\, \xi\in H^1_0(\Omega). \label{time-discretization-theta-w}
\end{align}
Set $K_{\mathrm{u}}=\max\Big\{\operatorname*{ess\,sup}_{\Omega}\theta^k,\,\max_{\partial\Omega}\theta_{\mathrm{b}}\Big\}$. Since $\theta\in H^1(\Omega)$,  we see that $\theta^+:= \max\{\theta-K_{\mathrm{u}},0\}$ satisfies $\theta^+\in H^1_0(\Omega)$. Next, testing \eqref{time-discretization-theta-w} by $\xi=\theta^+$, using \eqref{notation-theta} and integration by parts,
we obtain
\begin{align*}
& \frac{1}{h}\int_\Omega (\theta-K_{\mathrm{u}}) \theta^+\,\mathrm{d}x - \frac{1}{h}\int_\Omega (\theta^k-K_{\mathrm{u}}) \theta^+\,\mathrm{d}x
+ \int_\Omega (\u\cdot \nabla (\theta-K_{\mathrm{u}}))\theta^+ \,\mathrm{d}x
\\
&\quad + \int_\Omega \kappa(\phi^k,\theta^k)\nabla (\theta-K_{\mathrm{u}}) \cdot \nabla \theta^+\,\mathrm{d}x=0,
\end{align*}
which implies
\begin{align*}
\frac{1}{h}\int_\Omega (\theta^+)^2\,\mathrm{d}x
+ \int_\Omega \kappa(\phi^k,\theta^k)|\nabla \theta^+|^2 \,\mathrm{d}x
= \frac{1}{h}\int_\Omega (\theta^k-K_{\mathrm{u}}) \theta^+\,\mathrm{d}x
\leqslant 0.
\end{align*}
Here, we have used the facts that $\theta^k-K_{\mathrm{u}}\leqslant 0$ almost everywhere in $\Omega$ and $\u$ is divergence free.
This yields $|\theta^+|^2 =0$ almost everywhere in $\Omega$. Since $\theta^+\in H^1_0(\Omega)$, we can conclude that $\theta^+=0$ almost everywhere in $\overline{\Omega}$, i.e., the right-hand side of \eqref{maximumprinciple-discrete} holds. The left-hand side of \eqref{maximumprinciple-discrete} can be proven in a similar way. This yields $\theta\in L^\infty(\Omega)$. By the definition of $\vartheta$, we also get $\vartheta\in L^\infty(\Omega)$.
\end{proof}

Next, we show that the discrete solution $\phi$ satisfies the property of mass conservation.
\begin{lemma}\label{mass-dis}
Assume that $(\u, \phi, \mu)
\in \bm{V}_\sigma \times \mathcal{D}(\partial\widetilde{E}) \times H^2_{\mathbf{n}}(\Omega)$ satisfies the equation \eqref{time-discretization-phi} with a given function $\phi^k\in H^1(\Omega)$, then it holds
\begin{equation}
\overline{\phi}=\overline{\phi^k}.\label{mass-dis-a}
\end{equation}
\end{lemma}
\begin{proof}
Taking the spatial average on both sides of the equation \eqref{time-discretization-phi} along
with the divergence theorem, we obtain the conclusion.
\end{proof}

\begin{remark}\rm
The properties \eqref{maximumprinciple-discrete} and \eqref{mass-dis-a} hold for all $k\in \mathbb{N}$, as long as the discrete solutions exist.
By induction, we easily get
\begin{align}
&  \min\Big\{\operatorname*{ess\,inf}_{\Omega}\theta^0,\,\min_{\partial\Omega}\theta_{\mathrm{b}}\Big\} \leqslant \theta^k(x) \leqslant \max\Big\{\operatorname*{ess\,sup}_{\Omega}\theta^0,\,\max_{\partial\Omega}\theta_{\mathrm{b}}\Big\},\quad \text{a.e. in}\ \overline{\Omega},
\label{maximumprinciple-discrete-b}
\\
&  \overline{\phi^k}=\overline{\phi^0},
\label{mass-dis-b}
\end{align}
for any positive integer $k$.
\end{remark}

The following lemma gives a refinement of Lemma \ref{lem:sing-phi} under the specific choice $f=\mu+c_W(\phi+\phi^k)$ as in the discrete approximation \eqref{time-discretization-mu}.
\begin{lemma}\label{lem:mu-es-k}
Assume that $(\phi, \mu)\in \mathcal{D}(\partial\widetilde{E})\times H^1(\Omega)$ satisfies the equation \eqref{time-discretization-mu} with $\phi^k\in H^1(\Omega)$, $\phi^k\in [-1,1]$ almost everywhere, and $\overline{\phi} = \overline{\phi^k} \in (-1,1)$. Then we have
\begin{align}
& \|\partial \widetilde{E}(\phi)\|\leqslant C(\|\mu\|+1),
\label{dpm-es1}
\\
& |\overline{\mu}|\leqslant C(\|\nabla \mu\|+1),
\label{dpm-es2}
\\
& \|\phi\|_{W^{2,p}(\Omega)}+ \|F^{\prime}(\phi)\|_{L^p(\Omega)} \leqslant C(\|\nabla \mu\|+1),
\label{dpm-es3}
\\
& \|\phi\|_{H^2(\Omega)}^2 \leqslant C\|\nabla \mu\|\|\nabla \phi\| + C,
\label{dpm-es4}
\end{align}
where $C$ is a positive constant that may depend on $\Omega$ and $\overline{\phi}$, $p \in [2,\infty)$ if $d=2$ and $p\in [2,6]$ if $d=3$.
\end{lemma}
\begin{proof}
By definition, from $\phi\in \mathcal{D}(\partial\widetilde{E})$ we infer that $\phi\in [-1,1]$ almost everywhere in $\Omega$.
This, combined with the assumption $\phi^k\in [-1,1]$ almost everywhere in $\Omega$ easily yields \eqref{dpm-es1}. Next, from $(\mathbf{A4})$ and the assumption $\overline{\phi}\in (-1,1)$, we can obtain the estimate \eqref{dpm-es2} using the same argument as that for \cite[Theorem 3.6]{Conti2020}. To obtain \eqref{dpm-es3}, we apply \eqref{es-wpphi} with $f=\mu + c_W (\phi+\phi^k)$ together with the Poincar\'{e}--Wirtinger inequality and \eqref{dpm-es2} such that
\begin{align*}
\|\phi\|_{W^{2,p}(\Omega)}+ \|F^{\prime}(\phi)\|_{L^p(\Omega)}
&\leqslant C(1+\|\mu\|_{L^p(\Omega)}+\|\phi\|_{L^p(\Omega)}+\|\phi^k\|_{L^p(\Omega)})
\\
&\leqslant C(1+ \|\mu-\overline{\mu}\|_{L^p(\Omega)} +|\overline{\mu}|)
\leqslant C(\|\nabla \mu\|+1).
\end{align*}
Concerning the last conclusion \eqref{dpm-es4}, we introduce the globally Lipschitz function $h_j:\mathbb{R}\to \mathbb{R}$ as in \cite[Appendix A]{Conti2020}: for every integer $j\geqslant 2$,
\begin{align*}
h_j(s)=
\begin{cases}
-1+\frac{1}{j},\quad \ \text{if}\ s<-1+\frac{1}{j},\\
s,\qquad\qquad  \text{if}\ s\in [-1+\frac{1}{j},1-\frac{1}{j}],\\
1-\frac{1}{j},\qquad \,\text{if}\ s>1-\frac{1}{j}.
\end{cases}
\end{align*}
Define $\phi_j= h_j\circ \phi$. Since $\phi\in H^1(\Omega)$, then $\phi_j\in H^1(\Omega)$ and $\nabla \phi_j=\nabla \phi\cdot \chi_{[-1+\frac{1}{j},1-\frac{1}{j}]}(\phi)$.
Testing  equation \eqref{time-discretization-mu} by $-\Delta \phi$, we have
\begin{align}
\|\Delta \phi\|^2 + (F'(\phi_j),-\Delta\phi) = - (\mu,\Delta\phi) - c_W(\phi+\phi^k, \Delta \phi)
+ (F'(\phi)-F'(\phi_j), \Delta \phi).
\label{es-H2-phi-a}
\end{align}
As in the proof of \cite[Lemma A.1]{Conti2020}, we infer from the convexity of $F$ that
\begin{align*}
& (F'(\phi_j),-\Delta\phi) = (F''(\phi_j)\nabla \phi\cdot \chi_{[-1+\frac{1}{j},1-\frac{1}{j}]}(\phi), \nabla \phi)\geqslant 0.
\end{align*}
Besides, it holds
$$
(F'(\phi)-F'(\phi_j), \Delta \phi)\to 0\quad \text{as}\ j\to \infty.
$$
Using integration by parts, the Cauchy--Schwarz inequality and Young's inequality, we observe that
\begin{align*}
&- (\mu,\Delta\phi) = (\nabla \mu, \nabla \phi)\leqslant \|\nabla \mu\|\|\nabla \phi\|,
\\
& - c_W(\phi+\phi^k, \Delta \phi) \leqslant \frac12\|\Delta \phi\|^2+ c_W^2(\|\phi\|^2+\|\phi^k\|^2).
\end{align*}
Inserting the above estimates into \eqref{es-H2-phi-a} and passing to the limit as $j\to \infty$, we can conclude
\begin{align}
\|\Delta \phi\|^2\leqslant 2 \|\nabla \mu\|\|\nabla \phi\| + C, \label{es-H2-phi-b}
\end{align}
where the positive constant $C$ depends only on $\Omega$ and $c_W$. This, combined with the elliptic theory for the Neumann problem, yields
$$
\|\phi\|_{H^2(\Omega)}^2\leqslant C(\|\Delta \phi\|+\|\phi\|)^2\leqslant C\|\nabla \mu\|\|\nabla \phi\|+C.
$$
The proof is complete.
\end{proof}

Let $N$ be an arbitrarily given positive integer. We establish the existence of a solution to the discrete problem \eqref{time-discretization-u}--\eqref{time-discretization-theta}.
\begin{proposition}\label{lemma-time-discretization}
For any $k\in \mathbb{N}$, assume that the data at the time step $k$ satisfy $\u^k \in \H_\sigma$, $\phi^k \in H^2_\mathbf{n}(\Omega)$ with $\phi^k \in [-1,1]$ almost everywhere in $\Omega$, $\overline{\phi^k}\in (-1,1)$, and $\vartheta^k \in H^2(\Omega)\cap H^1_0(\Omega)$.
\begin{enumerate}[label=\textup{(\arabic*)}]
\item The discrete problem \eqref{time-discretization-u}--\eqref{time-discretization-theta} admits a solution
$$(\u, \phi, \mu, \vartheta) \in \V \times \D(\partial \widetilde{E}) \times H_{\mathbf{n}}^2(\Omega) \times (  H^2(\Omega) \cap H_0^1(\Omega))$$
at the time step $k+1$. Moreover, $\theta= \vartheta+ \Theta_{\mathrm{b}}$ satisfies \eqref{maximumprinciple-discrete}.
\item The solution $(\u, \phi, \mu, \vartheta)$ satisfies the following discrete energy inequality:
	\begin{align}
		& (1+C_1h) E_{\mathrm{tot}}(\u,\phi, \vartheta)
        + \frac{h}{4} \left(
		 \underline{\nu}  \| \nabla \u \|^2 +
		 \lambda_0 a \underline{m}  \| \nabla \mu \|^2 +
		 \underline{\kappa} \| \nabla \vartheta \|^2 \right)
        \notag \\
		&\quad \leqslant
       E_{\mathrm{tot}}(\u^k,\phi^k, \vartheta^k) + C_2h,
    \label{discrete-energy-inequality}
	\end{align}
where the modified total energy $E_{\tot}$ is given by
   \begin{align}
		E_{\tot}(\u,\phi,\vartheta) =
		\int_\Omega \frac{1}{2}\rho|\u|^2\,\mathrm{d}x +
		\lambda_0 a \left( \frac{1}{2} \| \nabla \phi \|^2 + \int_{\Omega} W(\phi)\,\mathrm{d}x \right) +
		\frac{1}{2} \| \vartheta \|^2.
        \label{E_tot}
	\end{align}
The positive constants $C_1$, $C_2$ in \eqref{discrete-energy-inequality} depend on $\Omega$, coefficients of the system, $\|\theta_{\mathrm{b}}\|_{H^\frac{3}{2}(\partial\Omega)}$, $\|\theta^0\|_{L^\infty(\Omega)}$, but they are independent of $h$ and $k$.
\end{enumerate}
\end{proposition}
\begin{proof}
\textbf{Part A. Discrete energy inequality.} We first verify that if $(\u, \phi, \mu, \vartheta)\in \V \times \D(\partial \widetilde{E}) \times H_{\mathbf{n}}^2(\Omega) \times ( H^2(\Omega) \cap H_0^1(\Omega))$ is a solution to problem \eqref{time-discretization-u}--\eqref{time-discretization-theta} at the time step $k+1$ subject to the given data $(\u^k, \phi^k, \mu^k, \vartheta^k)$,
then it satisfies the discrete energy inequality \eqref{discrete-energy-inequality}. The key feature of \eqref{discrete-energy-inequality} is that this energy inequality is based on the maximum principle, namely, uniform $L^\infty$-estimates (with respect to $k$, $h$) for $\phi$, $\phi^k$, $\theta$, $\theta^k$ play a crucial role in the derivation.

Taking $\v = \u$ in \eqref{time-discretization-u}, using the following identities (see  \cite{JMFM})
\begin{align*}
    & \int_\Omega \left((\mathrm{div}\mathbf{J})\frac{\u}{2}+(\mathbf{J}\cdot \nabla)\u\right)\cdot\u\,\mathrm{d}x=0,\\
    & \int_\Omega \left(\mathrm{div}(\rho^k\u\otimes \u)-(\nabla \rho^k\cdot\u)\frac{\u}{2}\right)\cdot\u\,\mathrm{d}x=0,\\
    &(\rho\u-\rho^k\u^k)\cdot\u
    =\left(\rho\frac{|\u|^2}{2}-\rho^k\frac{|\u^k|^2}{2}\right) +(\rho-\rho^k)\frac{|\u|^2}{2}+\rho^k\frac{|\u-\u^k|^2}{2},
\end{align*}
we obtain
	\begin{align}
		&\int_{\Omega} \left(\frac{\rho|\u|^2}{2h} - \frac{\rho^k|\u^k|^2}{2h}
        + \frac{\rho^k|\u - \u^k|^2}{2h}\right) \dx
		+ \int_{\Omega} 2 \nu(\phi^k,\theta^k) | D \u |^2 \dx
        \notag \\
		&\quad =  \lambda_0 a (\mu \nabla \phi^k, \u)
         -\lambda_0 b \big(\theta^k (\nabla \phi \otimes \nabla \phi), \nabla \u\big)
         +\big(\bm{f}_{\mathrm{b}}(\phi^k, \theta^k), \u\big).
         \label{time-discretization-estimate-u}
	\end{align}
Testing \eqref{time-discretization-phi} by $\lambda_0 a \mu$ and \eqref{time-discretization-mu} by $\dfrac{\lambda_0 a}{h} (\phi - \phi^{k})$, adding the resultants together, by some straightforward calculations, we get
	\begin{align}
		&   \lambda_0 a \int_{\Omega} \left(\frac{| \nabla \phi |^2}{2h} - \frac{| \nabla \phi^k |^2}{2h} + \frac{| \nabla(\phi-\phi^k) |^2}{2h} \right)\dx
		+ \lambda_0 a \int_{\Omega} m(\phi^k,\theta^k) | \nabla \mu |^2 \dx \notag \\
		&\qquad
		+ \frac{\lambda_0 a}{h} \int_{\Omega} F^{\prime}(\phi) (\phi - \phi^k) \dx
        - \frac{\lambda_0 ac_W}{h} \int_{\Omega} (|\phi|^2 - |\phi^k|^2)\,\mathrm{d}x
       \notag \\
		&\quad =  - \lambda_0 a(\u\cdot\nabla \phi^k,\mu).
\label{time-discretization-estimate-ph}
	\end{align}
Finally, taking $\xi = \vartheta$ in \eqref{time-discretization-theta} yields
	\begin{align}
		& \int_{\Omega} \left(\frac{|\vartheta|^2}{2h} - \frac{|\vartheta^k|^2}{2h} + \frac{|\vartheta-\vartheta^k|^2}{2h} \right)\dx
		+ \int_{\Omega} \kappa(\phi^k,\theta^k) | \nabla \vartheta |^2 \dx \notag \\
		&\quad =  -(\u \cdot \nabla \Theta_{\mathrm{b}},\vartheta)
        - ( \kappa(\phi^k,\theta^k) \nabla \Theta_{\mathrm{b}},\nabla \vartheta).
        \label{time-discretization-estimate-th}
	\end{align}

We observe that the first term on the right-hand side of \eqref{time-discretization-estimate-u} cancels with the first term on the right-hand side of \eqref{time-discretization-estimate-ph}. In what follows, we estimate the other terms in \eqref{time-discretization-estimate-u}--\eqref{time-discretization-estimate-th}.

Let us first treat the right-hand side of \eqref{time-discretization-estimate-u}.
Using the fact $\overline{\phi}=\overline{\phi^k}\in (-1,1)$, the $L^\infty$-estimates \eqref{maximumprinciple-discrete-b}, $\|\phi^k \|_{L^{\infty}(\Omega)} \leqslant 1$, $\|\phi \|_{L^{\infty}(\Omega)} \leqslant 1$, and the estimate \eqref{dpm-es4} for $\|\phi\|_{H^2(\Omega)}$, we find
	\begin{align}
		-\lambda_0 b \big( \theta^k (\nabla \phi \otimes \nabla \phi), \nabla \u\big)
		&\leqslant |b| | \lambda_0 | \| \theta^k \|_{L^{\infty}(\Omega)} \| \nabla \phi \|_{L^4(\Omega)}^2 \| \nabla \u \| \notag \\
		&\leqslant \frac{\underline{\nu}}{6} \| \nabla \u \|^2
		+ C \| \phi \|_{H^2(\Omega)}^2 \| \phi \|_{L^{\infty}(\Omega)}^2  \notag \\
		&\leqslant \frac{\underline{\nu}}{6} \| \nabla \u \|^2
		+ C (\| \nabla \mu \| \| \nabla \phi \|  + 1) \notag \\
		&\leqslant \frac{\underline{\nu}}{6} \| \nabla \u \|^2
		+ \frac{\lambda_0 a \underline{m}}{4} \| \nabla \mu \|^2
		+ C_3\| \nabla \phi \|^2 + C.
		\label{time-discretization-estimate1}
	\end{align}
Here, we have also used Young's inequality and the following Gagliardo--Nirenberg inequality (valid for $d=2,3$)
\begin{align}
\|\nabla \phi\|_{L^4(\Omega)}\leqslant C \|\phi\|_{H^2(\Omega)}^\frac12\|\phi\|_{L^\infty(\Omega)}^\frac12, \quad \forall\,\phi\in H^2(\Omega). \label{GL-L4}
\end{align}
In order to absorb the term $ C_3\| \nabla \phi \|^2 $ on the right-hand side of \eqref{time-discretization-estimate1}, we test \eqref{time-discretization-mu} by $2C_3(\phi-\overline{\phi})$, after integration by parts, we get
    \begin{align}
    & 2C_3\|\nabla \phi\|^2 + 2C_3(F'(\phi), \phi-\overline{\phi}) = 2C_3(\mu, \phi-\overline{\phi}) + 2C_3c_W(\phi+\phi^k, \phi-\overline{\phi}).
    \label{time-discretization-estimate-mu}
    \end{align}
By Taylor's expansion, it follows from the convexity of $F$ and \eqref{mass-dis-b} that
\begin{align}
	(F^{\prime}(\phi), \phi - \overline{\phi})
	& \geqslant   \int_{\Omega} F(\phi) \, \dx - \int_{\Omega} F(\overline{\phi}) \, \dx
      \geqslant \int_{\Omega} F(\phi) \, \dx - C.
    \notag
\end{align}
On the other hand, it follows from the Poincar\'{e}--Wirtinger inequality and $\|\phi \|_{L^{\infty}(\Omega)} \leqslant 1$ that
\begin{align}
2C_3(\mu, \phi-\overline{\phi})& = 2C_3(\mu-\overline{\mu}, \phi)
\leqslant C\|\nabla \mu\|\|\phi\|  \leqslant \frac{\lambda_0 a \underline{m}}{4} \| \nabla \mu\|^2 + C,
\notag
\end{align}
and
\begin{align}
2C_3c_W(\phi+\phi^k, \phi-\overline{\phi}) \leqslant C(\|\phi\|+\|\phi^k\|)\|\phi\|\leqslant C. \notag
\end{align}
Inserting the above estimates into \eqref{time-discretization-estimate-mu}, we obtain
\begin{align}
    & 2C_3\|\nabla \phi\|^2 + 2C_3\int_{\Omega} F(\phi) \, \dx \leqslant \frac{\lambda_0 a \underline{m}}{4} \| \nabla \mu\|^2+C.
    \label{time-discretization-estimate-mu-b}
    \end{align}
Next, using $(\mathbf{A5})$, Poincar\'{e}'s inequality \eqref{Poin} and Young's inequality, we see that the last term on the right-hand side of \eqref{time-discretization-estimate-u} can be estimated as follows
\begin{align}
		\big(\bm{f}_{\mathrm{b}}(\phi^k, \theta^k), \u\big)
     \leqslant C\| \u \| \leqslant \frac{\underline{\nu}}{6} \| \nabla \u \|^2 + C.
		\label{time-discretization-estimate2}
\end{align}
Concerning \eqref{time-discretization-estimate-ph}, we apply once again the convexity of $F$ to get
\begin{align}
    \int_\Omega  F^{\prime}(\phi)(\phi-\phi^{k})\,\mathrm{d}x  \geqslant \int_\Omega F(\phi)\,\mathrm{d}x  - \int_\Omega F(\phi^{k})\,\mathrm{d}x.
     \label{time-discretization-conF}
\end{align}
Finally, let us consider \eqref{time-discretization-estimate-th}. By the definition of $\Theta_{\mathrm{b}}$, the $L^\infty$-estimate for $\theta$, $\|\phi^k\|_{L^\infty(\Omega)}\leqslant 1$, H\"{o}lder's inequality, Poincar\'{e}'s inequality and Young's inequality, we can control the two terms on the right-hand side as follows
\begin{align}
-(\u \cdot \nabla \Theta_{\mathrm{b}},\vartheta)
&\leqslant \|\u\|_{L^6(\Omega)}\|\nabla \Theta_{\mathrm{b}}\|_{L^3(\Omega)}\|\theta-\Theta_{\mathrm{b}}\|
\notag \\
&\leqslant C\|\nabla \u\| \|\Theta_{\mathrm{b}}\|_{H^2(\Omega)}(\|\theta\| + \|\Theta_{\mathrm{b}}\|)
\notag \\
&\leqslant \frac{\underline{\nu}}{6} \|\nabla \u\|^2  +C,
\label{time-discretization-estimate3}
\end{align}
and
\begin{align}
- ( \kappa(\phi^k,\theta^k) \nabla \Theta_{\mathrm{b}},\nabla \vartheta)
&\leqslant \|\kappa(\phi^k,\theta^k)\|_{L^\infty(\Omega)}\|\nabla \Theta_{\mathrm{b}}\|\|\nabla \vartheta\|
\notag \\
&\leqslant \frac{\underline{\kappa}}{2} \|\nabla \vartheta\|^2+C.
\label{time-discretization-estimate4}
\end{align}

Adding \eqref{time-discretization-estimate-u}, \eqref{time-discretization-estimate-ph}, \eqref{time-discretization-estimate-th} and \eqref{time-discretization-estimate-mu-b} together, using the estimates \eqref{time-discretization-estimate1}, \eqref{time-discretization-estimate2}, \eqref{time-discretization-conF}, \eqref{time-discretization-estimate3} and \eqref{time-discretization-estimate4}, we can deduce from the assumptions $(\mathbf{A1})$--$(\mathbf{A3})$ and Korn's inequality \eqref{Korn} that
\begin{align}
		& E_{\mathrm{tot}}(\u,\phi, \vartheta)
        + h \left(\frac{\underline{\nu}}{2}   \| \nabla \u \|^2
		+ \frac{\lambda_0 a \underline{m}}{2} \| \nabla \mu \|^2
        + \frac{\underline{\kappa}}{2} \| \nabla \vartheta \|^2 \right)
        \notag \\
		&\quad +Ch\left(\frac12 \|\nabla\phi\|^2 + \int_\Omega F(\phi)\,\mathrm{d}x\right)
          \leqslant
       E_{\mathrm{tot}}(\u^k,\phi^k, \vartheta^k) + Ch,
       \label{dis-energy1}
\end{align}
where the positive constant $C$ depends on $\Omega$, coefficients of the system, $\|\theta_{\mathrm{b}}\|_{H^\frac{3}{2}(\partial\Omega)}$, $\|\theta^0\|_{L^\infty(\Omega)}$, but not on $h$ and $k$. The assumption $(\mathbf{A4})$ together with the estimate  $\|\phi\|_{L^\infty(\Omega)}\leqslant 1$ implies that the energy functional $\int_\Omega W(\phi)\,\mathrm{d}x$ is bounded from below by a constant that depends only on $\Omega$ and $c_W$. As a consequence, we can conclude the discrete energy inequality \eqref{discrete-energy-inequality} from \eqref{dis-energy1} and Poincar\'{e}'s inequality for $\u$, $\vartheta$.
\medskip

\textbf{Part B. Solvability of the discrete problem \eqref{time-discretization-u}--\eqref{time-discretization-theta}.}
In what follows, we prove the existence of a solution $(\u,\phi,\mu,\vartheta)$ to problem \eqref{time-discretization-u}--\eqref{time-discretization-theta}. For convenience, we will use the notation $\bm{w} = (\u, \phi, \mu, \vartheta)$ in the subsequent analysis.

Using the identities
\begin{align*}
& \mathrm{div}(\u\otimes \mathbf{J})= (\mathrm{div} \,\mathbf{J})\u+ (\mathbf{J}\cdot \nabla)\u,\qquad \mathrm{div} \mathbf{J} =  -\frac{\rho - \rho^k}{h}  -
	 \u \cdot \nabla \rho^{ k },\\
& \mathrm{div}\big[\theta^k (\nabla \phi \otimes \nabla \phi)\big]= (\nabla \phi \otimes \nabla \phi)\nabla \theta^k + \theta^k\left[-\big(\mu+c_W(\phi+\phi^k)\big)\nabla \phi  + \nabla \left(\frac{|\nabla \phi|^2}{2}+ F(\phi)\right)\right],
\end{align*}
we rewrite equation \eqref{time-discretization-u} as
\begin{align}
&\left( \frac{\rho\u - \rho^k\u^k}{h} , \v \right)
    + \big(\mathrm{div}(\rho^k \u \otimes \u), \v\big)
    + \left(\Big(\mathrm{div} \mathbf{J}  -\frac{\rho - \rho^k}{h}
    - \u \cdot \nabla \rho^{k}\Big)\frac{\u}{2},\v  \right)
    \notag \\
	&\qquad  + \left((\mathbf{J}\cdot \nabla)\u,\v\right)
    + \big(2 \nu(\phi^k, \theta^k) D \u , D \v\big)
    - \lambda_0 a \big(\mu \nabla \phi^k, \v\big)
    \notag \\
    &\quad =
	\lambda_0 b \big(  (\nabla \phi \otimes \nabla \phi)\nabla (\vartheta^k+ \Theta_{\mathrm{b}}), \v\big)
    - \lambda_0 b  \big((\vartheta^k+ \Theta_{\mathrm{b}})\big[\mu+c_W(\phi+\phi^k)\big] \nabla \phi, \bm{v}\big)
    \notag \\
    &\qquad -\frac{\lambda_0 b}{2} \big(\nabla(\vartheta^k+ \Theta_{\mathrm{b}}) (|\nabla \phi|^2 + 2F(\phi)),\bm{v}\big)
    + \big(\bm{f}_{\mathrm{b}}(\phi^k, \vartheta^k+ \Theta_{\mathrm{b}}), \v\big),
    \quad \forall\, \v \in \V.
\label{time-discretization-u-b}
\end{align}
Define
\begin{align*}
&X = \V \times \D(\partial \widetilde{E}) \times H_{\mathbf{n}}^2(\Omega) \times (H_0^1(\Omega)\cap H^2(\Omega)),\\
&Y = \V^{\prime} \times L^2(\Omega) \times L^2(\Omega) \times L^2(\Omega).
\end{align*}
We note that $X$ is not a Banach space due to nonlinear constraints. In view of
\eqref{time-discretization-phi}--\eqref{time-discretization-theta} and \eqref{time-discretization-u-b}, we consider two operators $\mathcal{L}_k, \mathcal{F}_k: X \to Y$ given by
	\begin{align*}
		\mathcal{L}_{k}(\bm{w}) =
		\begin{pmatrix}
			\mathcal{L}_k^{(1)}(\u) \\
			-\mathrm{div} ( m(\phi^{k}, \theta^k) \nabla \mu) + \displaystyle{\int_{\Omega} \mu\, \mathrm{d}x}\\
			\partial \widetilde{E}(\phi) \\
			-\mathrm{div}(\kappa(\phi^k,\theta^k) \nabla \vartheta)
		\end{pmatrix},
	\end{align*}
	and
	\begin{align*}
		\mathcal{F}_k(\bm{w}) =
		\begin{pmatrix}
			 \mathcal{F}^{(1)}_{k}(\u, \phi, \mu)
             \\
			 - \dfrac{\phi - \phi^{k}}{h} - \u \cdot \nabla \phi^k
             + \displaystyle{\int_{\Omega} \mu \dx}
             \\
			 \mu + c_W(\phi+\phi^k)
             \\
			 - \dfrac{\vartheta - \vartheta^{k}}{h} - \u \cdot \nabla \vartheta
             - \u \cdot \nabla \Theta_{\mathrm{b}}
             + \mathrm{div} (\kappa(\phi^k,\theta^k) \nabla \Theta_{\mathrm{b}})
		\end{pmatrix}.	
	\end{align*}
    For convenience, we denote the $i$th $(i=1,2,3,4)$ element of $\mathcal{L}_{k}$ (resp. $\mathcal{F}_{k}$) by $\mathcal{L}^{(i)}_{k}$ (resp. $\mathcal{F}^{(i)}_{k}$).
	The linear operator $\mathcal{L}_k^{(1)} : \V \to \V^{\prime}$ is defined by
	$$
	\big \langle \mathcal{L}_k^{(1)}(\u) , \v \big \rangle_{\V} = \int_{\Omega} 2 \nu(\phi^k,\theta^k) D \u : D \v \dx,
	\quad  \forall\, \v \in \V,
	$$
    and $\mathcal{F}^{(1)}_{k}(\u, \phi, \mu)$ given by
    \begin{align*}
    \mathcal{F}^{(1)}_{k}(\u, \phi, \mu) =& - \dfrac{\rho\u - \rho^k\u^k}{h} - \mathrm{div}(\rho^k \u \otimes \u)
    -\Big(\mathrm{div}\, \mathbf{J}  -\dfrac{\rho - \rho^k}{h}
    -\u \cdot \nabla \rho\Big)\dfrac{\u}{2}
    \\
    &- (\mathbf{J}\cdot \nabla)\u
    + \lambda_0a \mu \nabla \phi^k
    - \lambda_0 b \theta^k \big[\mu+c_W(\phi+\phi^k)\big] \nabla \phi
    \\
    &+ \lambda_0 b (\nabla \phi \otimes \nabla \phi)\nabla \theta^k
    - \dfrac{\lambda_0 b}{2} \nabla\theta^k (|\nabla \phi|^2 + 2F(\phi))
    + \bm{f}_{\mathrm{b}}(\phi^k, \theta^k)
    \end{align*}
	should also be understood in the sense of a weak formulation.
    After establishing the above framework, we see that solving the discrete problem \eqref{time-discretization-u}--\eqref{time-discretization-theta} is equivalent to solving the abstract equation
	$$
	\mathcal{L}_k(\bm{w}) - \mathcal{F}_k(\bm{w}) = 0.
	$$
	
\textbf{Analysis of $\mathcal{L}_k$.} Our aim is to show that the operator $\mathcal{L}_k: X\to Y$ is invertible.

Given $(\phi^k, \theta^k)$ with described regularity properties, using Korn's inequality, Poincar\'{e}'s inequality, the Poincar\'{e}--Wirtinger inequality, and the positive lower/upper bounds on the coefficients $\nu$, $m$, $\kappa$, we can apply the Lax--Milgram theorem to conclude that
\begin{align*}
\mathcal{L}_k^{(1)}:\ \V \to \V',\quad
\mathcal{L}_k^{(2)}:\ H^1(\Omega)\to (H^1(\Omega))',\quad
\mathcal{L}_k^{(4)}:\ H^1_0(\Omega)\to H^{-1}(\Omega),
\end{align*}
are all invertible (cf., e.g., \cite{JMFM}). Moreover, it is straightforward to check that the corresponding inverse operators are continuous.

Consider the elliptic problem associated with $\mathcal{L}_k^{(2)}$ such that
\begin{align}
\begin{cases}
-\mathrm{div}( m(\phi^{k},\theta^k) \nabla \mu) + \displaystyle{\int_{\Omega} \mu\,\dx} = f,\quad \text{in}\ \Omega,\\
\partial_\mathbf{n} \mu=0,\qquad\qquad \qquad\qquad \qquad \qquad \ \  \text{on}\ \partial\Omega.
\end{cases}
\label{sol-mu}
\end{align}
We can extend the bootstrapping argument in \cite{JMFM} to show that for every $f\in L^2(\Omega)$, the unique solution $\mu$ to problem \eqref{sol-mu} satisfies $\mu\in H^2_\mathbf{n}(\Omega)$. Indeed, the solution $\mu$ satisfies $\mu\in H^1(\Omega)$ and can be viewed as a weak solution to the following equation:
$$
-\Delta \mu= \frac{1}{m(\phi^{k},\theta^k)}\left( \nabla m(\phi^{k},\theta^k)\cdot \nabla \mu - \displaystyle{\int_{\Omega} \mu\,\dx} + f\right)
$$
subject to the homogeneous Neumann boundary condition for $\mu$. From assumptions $\phi^k, \theta^k\in H^2(\Omega)$, $(\mathbf{A2})$, and Sobolev embedding theorems $H^2(\Omega)\hookrightarrow L^\infty(\Omega)$, $H^1(\Omega)\hookrightarrow L^6(\Omega)$, we have
\begin{align*}
& \left\|\frac{1}{m(\phi^{k},\theta^k)} \nabla m(\phi^{k},\theta^k)\cdot \nabla \mu\right\|_{L^\frac32{(\Omega)}}
\\
& \quad \leqslant \frac{1}{\underline{m}}\left(\|\partial_1 m\|_{L^\infty(\Omega)}\|\nabla \phi^k\|_{L^6(\Omega)}+ \|\partial_2 m\|_{L^\infty(\Omega)}\|\nabla \theta^k\|_{L^6(\Omega)}\right)\|\nabla \mu\|
\\
&\quad \leqslant \frac{C}{\underline{m}}\left(\|\phi^k\|_{H^2(\Omega)}+ \|\theta^k\|_{H^2(\Omega)}\right)\|\nabla \mu\|,
\end{align*}
which together with the assumption $f\in L^2(\Omega)$ yields $\Delta \mu\in L^\frac32(\Omega)$. Here, $\partial_i m$ denotes the partial derivative of $m$ with respect to its  $i$th component ($i=1,2$). By the elliptic theory, we get $\mu\in W^{2,\frac{3}{2}}(\Omega) \hookrightarrow W^{1,3}(\Omega)$. This property combined with the following estimate
\begin{align*}
& \left\|\frac{1}{m(\phi^{k},\theta^k)} \nabla m(\phi^{k},\theta^k)\cdot \nabla \mu\right\| \\
& \quad \leqslant \frac{C}{\underline{m}}\left(\|\partial_1 m\|_{L^\infty(\Omega)}\|\nabla \phi^k\|_{L^6(\Omega)}+ \|\partial_2 m\|_{L^\infty(\Omega)}\|\nabla \theta^k\|_{L^6(\Omega)}\right)\|\nabla \mu\|_{L^3(\Omega)},
\end{align*}
implies that $\Delta \mu\in L^2(\Omega)$. Hence, we can conclude that $\mu\in H^2_\mathbf{n}(\Omega)$.
Analogously, we consider the Dirichlet problem
\begin{align}
\begin{cases}
-\mathrm{div}( \kappa(\phi^{k},\theta^k) \nabla \vartheta) = f,\quad \text{in}\ \Omega,\\
\vartheta=0,\qquad\qquad \qquad \qquad  \ \  \text{on}\ \partial\Omega.
\end{cases}
\label{sol-varth}
\end{align}
From assumptions $\phi^k, \theta^k\in H^2(\Omega)$, $(\mathbf{A3})$, we see that for every $f\in L^2(\Omega)$, the unique solution $\vartheta$ to the problem \eqref{sol-varth} satisfies $\vartheta\in H^2(\Omega)\cap H^1_0(\Omega)$. Hence, the operators
\begin{align*}
\mathcal{L}_k^{(2)}:\ H^2_\mathbf{n}(\Omega) \to L^2(\Omega),\quad
\mathcal{L}_k^{(4)}:\ H^2(\Omega)\cap H^1_0(\Omega) \to L^2(\Omega),
\end{align*}
are both invertible (cf. \cite{JMFM}), and their corresponding inverse operators are continuous.

Finally, we consider $\mathcal{L}_k^{(3)}= \partial \widetilde{E}$. Since $F$ is strictly convex, then the nonlinear operator $\mathcal{L}_k^{(3)}$ is maximal monotone and $\mathcal{L}_k^{(3)}:\D(\partial \widetilde{E}) \to L^2(\Omega)$ is invertible (see also Lemma \ref{lem:sing-phi}). Concerning the inverse operator $(\mathcal{L}_k^{(3)})^{-1}=(\partial \widetilde{E})^{-1} : L^2(\Omega) \to \D(\partial \widetilde{E})$, the estimate \eqref{es-pE-1} implies that $\sup_{f\in U}\|F'\big((\mathcal{L}_k^{(3)})^{-1}(f)\big)\|$ is bounded, provided that $U$ is a bounded set in $L^2(\Omega)$.
Moreover, it has been shown in \cite[Lemma 4.3]{JMFM} that for $s \in (0,\frac{1}{4})$, its inverse $(\mathcal{L}_k^{(3)})^{-1} : L^2(\Omega) \to H^{2-s}(\Omega)$ is continuous and compact.

In summary, we have shown that $\mathcal{L}_k: X\to Y$ is invertible with the inverse $\mathcal{L}_k^{-1}: Y\to X$.
Define the Banach spaces
	\begin{align*}
	&\widetilde{X} = \V \times H^{2-s}(\Omega) \times H_\mathbf{n}^2(\Omega) \times (H^2(\Omega)\cap H_0^1(\Omega)),\quad \text{for some}\ s\in \left(0,\frac14\right),\\
	&\widetilde{Y} = L^{ \frac{3}{2} }(\Omega) \times W^{1,\frac{3}{2}} (\Omega) \times H^1(\Omega) \times W^{1,\frac{3}{2}}(\Omega),
	\end{align*}
Then we infer that the mapping $\mathcal{L}_k^{-1}: Y\to \widetilde{X}$ is continuous.
Furthermore, due to the compact embedding $\widetilde{Y}\hookrightarrow\hookrightarrow Y$,
the restriction $\mathcal{L}_k^{-1}: \widetilde{Y}\to \widetilde{X}$ is also a compact operator.
\medskip
	
\textbf{Analysis of $\mathcal{F}_k$.} Define
$$
\widehat{X} = \V \times H^{2-s}_F(\Omega) \times H_\mathbf{n}^2(\Omega) \times (H^2(\Omega)\cap H_0^1(\Omega)) \subset \widetilde{X},
$$
where
$$
H^{2-s}_F(\Omega) = \{\phi\in H^{2-s}(\Omega)\ :\  F'(\phi) \in L^2(\Omega)\big\}.
$$
\begin{remark}\label{rem:Lk3}\rm
We note that $H^{2-s}_F(\Omega)$ is not a Banach space. Nevertheless, it has the following property. Let $\{\phi_j\}_{j\geqslant 1}$ be a sequence in $H^{2-s}_F(\Omega)$ such that $\phi_j \to \phi$ strongly in $H^{2-s}(\Omega)$ as $j\to \infty$ and $\{\|F'(\phi_j)\|\}_{j\geqslant 1}$ be uniformly bounded with respect to $j$. By the Sobolev embedding $H^{2-s}(\Omega)\hookrightarrow C(\overline{\Omega})$ for $s\in (0,\frac14)$, we have $\phi_j \to \phi$ point-wisely in $\Omega$.
Since $-1<\phi_j<1$ almost everywhere in $\Omega$, then $\phi\in L^\infty(\Omega)$ and $-1\leqslant \phi\leqslant 1$ almost everywhere in $\Omega$.
From $(\mathbf{A4})$, it follows that $F'(\phi_j)\to \widetilde{F'}(\phi)$ almost everywhere in $\Omega$,
where $\widetilde{F'}(s)=F'(s)$ if $s\in (-1,1)$ and $\widetilde{F'}(\pm 1)=\pm \infty$. By Fatou's lemma, we get $\|\widetilde{F'}(\phi)\|^2\leqslant \liminf_{j\to \infty}\|F'(\phi_j)\|^2<\infty$, which implies that $\widetilde{F'}(\phi)\in L^2(\Omega)$.
Thus, $-1<\phi<1$ almost everywhere in $\Omega$, implying that $\widetilde{F'}(\phi)=F'(\phi)$ almost everywhere in $\Omega$. Consequently, the limit function $\phi \in H^{2-s}_F(\Omega)$.
\end{remark}

In the following, we show that $\mathcal{F}_k : \widehat{X} \to \widetilde{Y}$ is continuous and maps bounded sets into bounded sets.

For this purpose, we derive estimates for every term in the expression of $\mathcal{F}_k(\bm{w})$. Using the assumptions $\phi^k, \theta^k \in H^2(\Omega)\hookrightarrow L^\infty(\Omega)$ (and thus $\rho^k\in H^2(\Omega)$), $(\mathbf{A2})$--$(\mathbf{A5})$, H\"{o}lder's inequality, the Sobolev embedding theorems and \eqref{es-wpphi}, we can deduce that
	\begin{align*}
		&\| \rho\u \|_{L^{\frac32 }(\Omega)} \leqslant \|\rho\| \|\u\|_{L^6(\Omega)}\leqslant  C(1+\|\phi\|) \| \u \|_{\V},
        \\
        & \| \rho^k\u^k \|_{L^{\frac32 }(\Omega)} \leqslant C (1+\|\phi^k\|)\| \u^k \|_{\V},
        \qquad \| \rho^k\u \|_{L^{\frac32 }(\Omega)} \leqslant C (1+\|\phi^k\|)\| \u \|_{\V},
        \\
		&\| \mathrm{div}(\rho^k(\u \otimes \u) )\|_{L^{\frac32 }(\Omega)}
           \leqslant \|\rho^k\|_{L^\infty(\Omega)}\|\nabla \u\|\|\u\|_{L^6(\Omega)}
           + \|\nabla \rho^k\|_{L^6(\Omega)} \|\u\|_{L^4(\Omega)}^2
        \\
        &\qquad\qquad \qquad \qquad \qquad  \leqslant C(1+\|\phi^k\|_{H^2(\Omega)}) \| \u \|_{\V}^2,
        \\
        & \|(\mathrm{div}\,\mathbf{J})\u\|_{L^{\frac32 }(\Omega)}
        \leqslant C(\|\partial_1 m\|_{L^\infty(\Omega)}\|\nabla \phi^k\|_{L^6(\Omega)}
        + \|\partial_2 m\|_{L^\infty(\Omega)}\|\nabla \theta^k\|_{L^6(\Omega)}) \|\nabla \mu\|_{L^4(\Omega)}\|\u\|_{L^4(\Omega)}
        \\
        &\qquad \qquad\qquad \qquad  + C \|m\|_{L^\infty(\Omega)} \|\Delta\mu\|\|\u\|_{L^6(\Omega)}
        \\
        & \qquad \qquad\qquad \qquad \leqslant C(1+\|\phi^k\|_{H^2(\Omega)} + \|\theta^k\|_{H^2(\Omega)})\|\mu\|_{H^2(\Omega)}\| \u \|_{\V},
        \\
        & \|(\mathbf{J}\cdot\nabla) \u\|_{L^{\frac32 }(\Omega)}
        \leqslant C\|m\|_{L^\infty(\Omega)} \|\nabla \mu\|_{L^6(\Omega)} \| \nabla \u \|
        \leqslant C\|\mu\|_{H^2(\Omega)}\| \u \|_{\V},
        \\
        & \|(\u\cdot\nabla \rho)\u\|_{L^{\frac32 }(\Omega)}
        \leqslant \|\nabla \rho\|_{L^4(\Omega)} \|\u\|_{L^6(\Omega)}\|\u\|_{L^4(\Omega)} \leqslant
        C \big(1+\|\phi\|_{H^\frac{7}{4}(\Omega)}\big)  \| \u \|_{\V}^2,
        \\
		&\| \mu \nabla \phi^k \|_{L^{\frac32 }(\Omega)}
		\leqslant C  \| \mu \|_{L^6(\Omega)} \| \nabla \phi^k \|
        \leqslant C \| \mu \|_{H^1(\Omega)}\|\phi^k\|_{H^1(\Omega)},
        \\
        &\| \theta^k  \big[\mu+c_W(\phi+\phi^k)\big] \nabla \phi \|_{L^{\frac32 }(\Omega)}
        \\
		&\qquad \leqslant C \|\theta^k\|_{L^\infty(\Omega)}  (\| \mu \|_{L^6(\Omega)} + \| \phi \|_{L^6(\Omega)} +\| \phi^k \|_{L^6(\Omega)}) \| \nabla \phi \|
        \\
        &\qquad \leqslant C \|\theta^k\|_{H^2(\Omega)} ( \| \mu \|_{H^1(\Omega)} + \|\phi\|_{H^1(\Omega)}+ \|\phi^k\|_{H^1(\Omega)})\|\phi\|_{H^1(\Omega)},
        \\
		& \| (\nabla \phi \otimes \nabla \phi)\nabla \theta^k\|_{L^{\frac32 }(\Omega)}
          + \| \nabla\theta^k (|\nabla \phi|^2 + 2F(\phi))\|_{L^{\frac32 }(\Omega)}
        \\
        &\qquad \leqslant C \| \nabla \theta^k \|_{L^6(\Omega)} \| \nabla \phi \|_{L^4(\Omega)}^2
        + C \| \nabla \theta^k \|_{L^6(\Omega)} \|F(\phi)\|
		\\
        &\qquad  \leqslant C \|\theta^k \|_{H^2(\Omega)} \big(1+\| \phi \|_{H^{\frac74}(\Omega)}^2\big),
        \\
		&\|\bm{f}_{\mathrm{b}}(\phi^k, \theta^k)\|_{L^{\frac32 }(\Omega)} \leqslant C(1+ \|\phi^k\|_{L^6(\Omega)})(1+ \|\theta^k\|)
        \leqslant C(1+ \|\phi^k\|_{H^1(\Omega)})(1+ \|\theta^k\|),
        \\
		&\| \u \cdot \nabla \phi^k \|_{W^{1,\frac32} (\Omega)}
		\leqslant \| \nabla \u \| \|\nabla \phi^k\|_{L^6(\Omega)}
        + C \| \u \|_{L^6(\Omega)}\|\phi^k\|_{H^2(\Omega)}
        \\
		&\qquad \qquad \qquad\quad\ \, \leqslant C \| \u \|_{\V} \| \phi^k \|_{H^2(\Omega)},
        \\
		&\left \| \int_{\Omega} \mu \,\dx \right\|_{W^{1,\frac32} (\Omega)}
        \leqslant C\left\| \int_{\Omega} \mu \,\dx \right \|
        \leqslant C \| \mu \|,
        \\
        & \|\mu + c_W(\phi+\phi^k)\|_{H^1(\Omega)}\leqslant \|\mu\|_{H^1(\Omega)}+c_W(\|\phi\|_{H^1(\Omega)} +\|\phi^k\|_{H^1(\Omega)}),
        \\
		&\| \u \cdot \nabla \theta \|_{W^{1,\frac32}(\Omega)}
		\leqslant  \|\nabla \u \| \| \nabla \theta \|_{L^6(\Omega)}
        + C\| \u \|_{L^6(\Omega)} \| \theta\|_{H^2(\Omega)}
        \\
        &\qquad \qquad \qquad\quad \leqslant C \| \u \|_{\V} (\| \vartheta \|_{H^2(\Omega)} + \|\Theta_{\mathrm{b}}\|_{H^2(\Omega)}),
	\end{align*}
and finally,
\begin{align}
& \| \mathrm{div} ( \kappa(\phi^k, \theta^k) \nabla \Theta_{\mathrm{b}} ) \|_{W^{1,\frac{3}{2}}(\Omega)}
 \leqslant \| \kappa(\phi^k, \theta^k) \nabla \Theta_{\mathrm{b}} \|_{W^{2,\frac{3}{2}}(\Omega)}
        \notag \\
&\quad \leqslant \| \kappa \|_{L^{\infty}(\Omega)} \| \Theta_{\mathrm{b}} \|_{W^{3,\frac{3}{2}}(\Omega)}
        + \big( \| \partial_1^2 \kappa \|_{L^{\infty}(\Omega)} \| \nabla \phi^k \|_{L^{3}(\Omega)}^2
        + \| \partial_1 \kappa \|_{L^{\infty}(\Omega)} \| \phi^k \|_{W^{2,\frac{3}{2}}(\Omega)}
        \notag \\
&\qquad + \| \partial_2^2 \kappa \|_{L^{\infty}(\Omega)} \| \nabla \theta^k \|_{L^{3}(\Omega)}^2
        + \| \partial_2 \kappa \|_{L^{\infty}(\Omega)} \| \theta^k \|_{W^{2,\frac{3}{2}}(\Omega)}
        \notag \\
&\qquad + \| \partial_1 \partial_2 \kappa \|_{L^{\infty}(\Omega)} \| \nabla \phi^k \|_{L^3(\Omega)} \| \nabla \theta^k \|_{L^3(\Omega)}
        \big) \| \nabla \Theta_{\mathrm{b}} \|_{L^{\infty}(\Omega)}
        \notag \\
&\qquad + (\|\partial_1\kappa\|_{L^\infty(\Omega)}\|\nabla \phi^k\|+ \|\partial_2\kappa\|_{L^\infty(\Omega)}\|\nabla \theta^k\|) \| \Theta_{\mathrm{b}} \|_{W^{2,6}(\Omega)}
        \notag \\
&\quad \leqslant C(1+\|\phi^k\|_{H^2(\Omega)}+ \|\theta^k\|_{H^2(\Omega)})\|\Theta_{\mathrm{b}}\|_{H^{3}(\Omega)}.
\label{need-H5/2}
\end{align}
Here, $\partial_i \kappa$ denotes the partial derivative of $\kappa$ with respect to its $i$th component ($i=1,2$).
Therefore, $\mathcal{F}_k: \widehat{X}\to \widetilde{Y}$ maps bounded sets into bounded sets. We note that in the above reasoning, $\|F(\phi)\|$ is controlled using the assumption $(\mathbf{A4})$ and the estimate $\|\phi\|_{L^\infty(\Omega)}\leqslant 1$ (since $F'(\phi)\in L^2(\Omega)$).

To verify the continuity of $\mathcal{F}_k:\widehat{X}\to \widetilde{Y}$, below we only treat the most difficult term $F(\phi) \nabla \theta^k $, since the continuous dependence estimate for the other terms in $\mathcal{F}_k(\bm{w})$ are straightforward in view of the above argument for boundedness.

Let $\phi_1, \phi_2\in H^{2-s}_F(\Omega)$. Using H\"{o}lder's inequality and the Sobolev embedding theorem, we obtain
   \begin{align*}
   \|(F(\phi_1)-F(\phi_2))\nabla \theta^k \|_{L^\frac43(\Omega)}
    &  \leqslant \|\nabla \theta^k \|_{L^6(\Omega)}\|F(\phi_1)-F(\phi_2)\|_{L^\frac{12}{7}(\Omega)}
    \\
    &  \leqslant C\|\theta^k \|_{H^2(\Omega)}\|F(\phi_1)-F(\phi_2)\|_{L^\frac{12}{7}(\Omega)}.
   \end{align*}
   Thus, it remains to control $\|F(\phi_1)-F(\phi_2)\|_{L^\frac{12}{7}(\Omega)}$.
   From the assumption $F'(\phi_1), F'(\phi_2)\in L^2(\Omega)$ and the singularity of $F'$ at $\pm 1$, we infer that $\phi_1, \phi_2\in (-1,1)$ almost everywhere in $\Omega$. Therefore, for almost all $x\in \Omega$ and any $r\in [0,1]$, it holds
   $$
   -1<\min\{\phi_1(x),\phi_2(x)\}\leqslant r\phi_1(x)+(1-r)\phi_2(x)\leqslant \max\{\phi_1(x),\phi_2(x)\}<1.
   $$
   On the other hand, from $(\mathbf{A4})$ it follows that $F'\in C^1(-1,1)$ is strictly increasing.
   Thus, we find
   $$
   |F'(r\phi_1+(1-r)\phi_2)|\leqslant |F'(\phi_1)|+|F'(\phi_2)|, \quad \text{a.e. in}\ \Omega,\ \ \forall\, r\in [0,1].
   $$
   As a consequence, $F'(r\phi_1+(1-r)\phi_2)\in L^2(\Omega)$ for every $r\in[0,1]$.
   Using Taylor's expansion
   \begin{align*}
   F(\phi_1)-F(\phi_2)=\int_0^1 F'(r\phi_1+(1-r)\phi_2)(\phi_1-\phi_2)\,\mathrm{d}r,
   \end{align*}
   we infer that
   \begin{align*}
   \|F(\phi_1)-F(\phi_2)\|_{L^\frac{12}{7}(\Omega)}& \leqslant \int_0^1 \|F'(r\phi_1+(1-r)\phi_2)\|\|\phi_1-\phi_2\|_{L^{12}(\Omega)}\,\mathrm{d}r\\
   &\leqslant C(\|F'(\phi_1)\|+\|F'(\phi_2)\|)\|\phi_1-\phi_2\|_{H^\frac{7}{4}(\Omega)},
   \end{align*}
   where the positive constant $C$ depends only on $\Omega$. This yields the continuity of $F(\phi) \nabla \theta^k $ from $H^{2-s}_F(\Omega)$ to $L^\frac43(\Omega)$ for any $s\in (0,\frac14)$.
   Besides, thanks to the fact $\theta^k \in H^2(\Omega)$ and $\mathbf{(A4)}$, it is straightforward to check that
   $ \|(F(\phi_1)-F(\phi_2))\nabla \theta^k \|_{L^6(\Omega)} \leqslant C \| \theta^k \|_{H^2(\Omega)}$, for some universal positive constant $C$. By interpolation, the continuity actually holds from $H^{2-s}_F(\Omega)$ to $L^\frac32(\Omega)$. Hence, we can conclude that the mapping $\mathcal{F}_k : \widehat{X} \to \widetilde{Y}$ is continuous.
   \medskip
	
	\textbf{Existence via the Leray--Schauder fixed point theorem.}
	Like in \cite{JMFM},  in order to find a solution $\bm{w}\in X$ satisfying the equation $\mathcal{L}_k(\bm{w}) - \mathcal{F}_k(\bm{w}) = 0$,
    we set $\bm{g}= \mathcal{L}_k(\bm{w})$ and consider the following equation
	$$\bm{g} - \mathcal{F}_k \circ \mathcal{L}_k^{-1} (\bm{g}) = 0.$$
Because $\mathcal{L}_k^{-1} : Y \to X$ is well defined and $\widetilde{Y}\subset Y$, so is the restriction $\mathcal{L}_k^{-1} : \widetilde{Y} \to X \subset \widehat{X}\subset \widetilde{X}$. Since $\mathcal{L}_k^{-1} : \widetilde{Y} \to \widetilde{X}$ is compact, $\mathcal{F}_k : \widehat{X}\subset \widetilde{X} \to \widetilde{Y}$ is continuous and maps bounded sets into bounded sets, recalling the above mentioned properties of $(\mathcal{L}_k^{(3)})^{-1}$ and Remark \ref{rem:Lk3}, we can check that the operator
$$\mathcal{K}_k := \mathcal{F}_k \circ \mathcal{L}_k^{-1} : \widetilde{Y} \to \widetilde{Y}$$
is well defined and also compact.

The problem then reduces to finding a fixed point for the operator $\mathcal{K}_k$ in $\widetilde{Y}$, which can be done by applying the Leray--Schauder fixed point theorem (see, for instance, \cite[Theorem 6.A]{Ze92}). To this aim, we only need to verify the following assertion: there exists some $R>0$ such that
\begin{align}
\textit{if }\bm{g} \in \widetilde{Y}\textit{ and }\ 0 \leqslant \Lambda  \leqslant 1\textit{ fulfills }\ \bm{g} = \Lambda \mathcal{K}_k(\bm{g}),\textit{ then }\| \bm{g} \|_{\widetilde{Y}} \leqslant R.
\label{claim}
\end{align}

Let $\bm{g} \in \widetilde{Y}$ and $ 0 \leqslant \Lambda  \leqslant 1$ satisfy $\bm{g} = \Lambda \mathcal{K}_k(\bm{g})$. For $(\u,\phi,\mu, \vartheta)= \bm{w}= \mathcal{L}_k^{-1}(\bm{g})\in X$, we have
	$$ \mathcal{L}_k(\bm{w}) = \Lambda \mathcal{F}_k(\bm{w}). $$
	Rewriting the above equation as a system of partial differential equations, we obtain
\begin{align}
&\Lambda\left( \frac{\rho\u - \rho^k\u^k}{h} , \v \right) +
	\Lambda\big(\mathrm{div}(\rho^k \u \otimes \u), \v\big) + \Lambda(\mathrm{div}(\u\otimes \mathbf{J}), \v)
    \notag \\
	&\qquad  +
	\big(2 \nu(\phi^k, \vartheta^k+ \Theta_{\mathrm{b}}) D \u , D \v\big) - \Lambda \lambda_0 a (\mu \nabla \phi^k, \v)
    \notag \\
    &\quad = -\Lambda
	\lambda_0 b \big((\vartheta^k+ \Theta_{\mathrm{b}}) (\nabla \phi \otimes \nabla \phi), \nabla \v\big) +
	\Lambda\big(\bm{f}_{\mathrm{b}}(\phi^k, \vartheta^k+ \Theta_{\mathrm{b}}), \v\big),
\quad &&\forall\, \v \in \V,
\label{Leray-Schauder-u}
\end{align}
with
\begin{align}
\mathbf{J}= -\frac{\rho_2-\rho_1}{2}m(\phi^k,\vartheta^k+ \Theta_{\mathrm{b}})\nabla \mu, \qquad  \text{a.e. in } \Omega, \label{Leray-Schauder-J}
\end{align}
and
\begin{align}
	&  \Lambda\frac{\phi - \phi^k}{h}  +
	 \Lambda\u \cdot \nabla \phi^k  - \Lambda \int_\Omega \mu\,\dx = \mathrm{div}(m(\phi^k,\vartheta^k+ \Theta_{\mathrm{b}}) \nabla \mu) -  \int_\Omega \mu\,\dx, \quad &&\text{a.e. in } \Omega,
\label{Leray-Schauder-phi}\\
	& \Lambda\mu  + \Lambda c_W (\phi+\phi^k) = - \Delta \phi + F^{\prime}(\phi), \quad &&\text{a.e. in } \Omega,
\label{Leray-Schauder-mu}  \\
	&  \Lambda\frac{\vartheta - \vartheta^k}{h}  +
	  \Lambda\u \cdot \nabla \vartheta - \mathrm{div}  ( \kappa(\phi^k,\vartheta^k+ \Theta_{\mathrm{b}}) \nabla \vartheta)
\notag \\
&\quad =  - \Lambda\u \cdot \nabla \Theta_{\mathrm{b}} +  \Lambda\mathrm{div} ( \kappa(\phi^k,\vartheta^k+ \Theta_{\mathrm{b}}) \nabla \Theta_{\mathrm{b}}), \quad &&\text{a.e. in } \Omega.
\label{Leray-Schauder-theta}
\end{align}

Taking $\v = \u$ in \eqref{Leray-Schauder-u}, testing \eqref{Leray-Schauder-phi} by $\lambda_0 a \mu$, \eqref{Leray-Schauder-mu} by $\dfrac{\lambda_0 a}{h}(\phi - \phi^k)$, and \eqref{Leray-Schauder-theta} by $\vartheta$, summing up the resultants, we get the energy identity
	\begin{align}
		& \Lambda \int_{\Omega} \frac{\rho|\u|^2}{2h} - \frac{\rho^k|\u^k|^2}{2h}
        + \Lambda \frac{\rho^k|\u - \u^k|^2}{2h} \dx
		+ \int_{\Omega} 2 \nu(\phi^k,\theta^k) | D \u |^2 \dx
        \notag \\
        & \qquad  +  \lambda_0 a \int_{\Omega} \frac{| \nabla \phi |^2}{2h} - \frac{| \nabla \phi^k |^2}{2h} + \frac{| \nabla(\phi-\phi^k) |^2}{2h} \dx
		+ \lambda_0 a \int_{\Omega} m(\phi^k,\theta^k) | \nabla \mu |^2 \dx
        \notag \\
        & \qquad + \Lambda \int_{\Omega} \frac{|\vartheta|^2}{2h} - \frac{|\vartheta^k|^2}{2h} + \frac{|\vartheta-\vartheta^k|^2}{2h} \dx
		+ \int_{\Omega} \kappa(\phi^k,\theta^k) | \nabla \vartheta |^2 \dx
        \notag \\
        &\qquad + (1-\Lambda) \lambda_0 a \left( \int_{\Omega} \mu \dx \right)^2
         + \frac{\lambda_0 a}{h} \int_{\Omega} F^{\prime}(\phi) (\phi - \phi^k) \dx
        - \Lambda \frac{\lambda_0 ac_W}{h} \int_{\Omega} (|\phi|^2 - |\phi^k|^2)\,\mathrm{d}x
       \notag \\
		&\quad =
         -\Lambda \lambda_0 b \big(\theta^k (\nabla \phi \otimes \nabla \phi), \nabla \u\big)
         +\Lambda \big(\bm{f}_{\mathrm{b}}(\phi^k, \theta^k), \u\big)-\Lambda (\u \cdot \nabla \Theta_{\mathrm{b}},\vartheta)
         \notag \\
         &\qquad - \Lambda ( \kappa(\phi^k,\theta^k) \nabla \Theta_{\mathrm{b}},\nabla \vartheta).
         \label{dis-energyLa}
	\end{align}

Since $\bm{w}\in X$, $\|\phi\|_{L^\infty(\Omega)}\leqslant 1$ holds for any $\Lambda\in [0,1]$. Concerning the $L^\infty$-estimate for $\theta$, for any given $\Lambda\in (0,1]$, we can apply Lemma \ref{Stampacchia-truncation-method-remark} to conclude \eqref{maximumprinciple-discrete}. When $\Lambda=0$, \eqref{Leray-Schauder-theta} becomes an elliptic equation. In this case, we infer from \cite[Lemma 3.1]{LB96} that
$$
\min_{\partial\Omega}\theta_{\mathrm{b}}\leqslant \theta(x) \leqslant \max_{\partial\Omega}\theta_{\mathrm{b}},\quad\text{a.e. in}\ \overline{\Omega}.
$$
Hence, we obtain $L^\infty$-estimates for both $\phi$, $\theta$ (as well as $\vartheta$) that are independent of the parameter $\Lambda\in [0,1]$.

Based on the above observation, using the fact $\Lambda\in [0,1]$ and accordingly modifying the arguments in Part A, we can deduce from \eqref{dis-energyLa} that
\begin{align}
		& \int_\Omega \frac{\Lambda}{2}\rho|\u|^2\,\mathrm{d}x +
		\lambda_0 a \left( \frac{1}{2} \| \nabla \phi \|^2 + \int_{\Omega} W(\phi)\,\mathrm{d}x \right)
        + \frac{\Lambda}{2} \| \vartheta \|^2
        \notag\\
        &\qquad
        + h \left(\frac{\underline{\nu}}{2}   \| \nabla \u \|^2
		+ \frac{\lambda_0 a \underline{m}}{2} \| \nabla \mu \|^2
        + \frac{\underline{\kappa}}{2} \| \nabla \vartheta \|^2 \right)
        + C'h\left(\frac12 \|\nabla\phi\|^2 + \int_\Omega F(\phi)\,\mathrm{d}x\right)
        \notag \\
        &\qquad + (1-\Lambda) \lambda_0 a \left( \int_{\Omega} \mu \dx \right)^2
        \notag \\
		&\quad \leqslant
       \int_\Omega \frac{\Lambda}{2}\rho^k|\u^k|^2\,\mathrm{d}x +
		\lambda_0 a \left( \frac{1}{2} \| \nabla \phi^k \|^2 + \int_{\Omega} W(\phi^k)\,\mathrm{d}x \right)
        + \frac{\Lambda}{2} \| \vartheta^k \|^2 + C''h,
       \label{dis-energy1-La}
\end{align}
where the positive constants $C'$, $C''$ are independent of $\Lambda \in [0,1]$.
Since $h\in (0,1]$, it follows that
	\begin{align*}
		\sqrt{1-\Lambda} \left| \int_{\Omega} \mu \dx \right|
		+ \| \u \|_{\V} + \| \nabla \mu \| + \| \phi \|_{H^1(\Omega)} + \| \vartheta \|_{H_0^1(\Omega)} \leqslant C_{k,h},
	\end{align*}
where the positive constant $C_{k,h}$ is independent of $\Lambda \in [0,1]$.
As in \cite{JMFM}, the estimate for $| \int_{\Omega} \mu \dx |$ should be distinguished in two cases. For $\Lambda \in [0,\frac12)$, it easily follows from the above estimate that $| \int_{\Omega} \mu \dx |\leqslant \sqrt{2}C_{k,h}$. For $\Lambda \in [\frac12,1]$, we can apply \eqref{dpm-es2} to conclude $| \int_{\Omega} \mu \dx |\leqslant C(\|\nabla \mu\|+1)\leqslant C$ for some positive constant $C$ depending on $C_{k,h}$, $\overline{\phi}$, but not on $\Lambda$.
Finally, using the $H^2$-estimate for elliptic problems \eqref{sol-mu}, \eqref{sol-varth}, we can deduce that
	\begin{align*}
		\| \u \|_{\V} + \| \mu \|_{H^2(\Omega)} + \| \phi \|_{H^1(\Omega)} + \| \theta \|_{H^2(\Omega)} \leqslant C_{k,h}.
	\end{align*}
In view of equation \eqref{Leray-Schauder-mu}, we infer from the fact $\Lambda \in [0,1]$ and a similar argument as in the proof for \eqref{dpm-es3} that $\|\phi\|_{H^2(\Omega)}+\|F'(\phi)\|\leqslant C_{k,h}$. As a consequence, it holds
	\begin{align}
		\| \bm{w} \|_{\widetilde{X}} + \|F'(\phi)\|\leqslant C_{k,h},\label{es-w-1}
	\end{align}
where the positive constant $C_{k,h}$ is independent of $\Lambda \in [0,1]$ and $\bm{g}$.
Recalling that $\bm{g}$ satisfies $\bm{g} = \Lambda \mathcal{K}_k(\bm{g}) =  \Lambda \mathcal{F}_k(\bm{w})$,
we infer from \eqref{es-w-1} and the above verified properties of $\mathcal{F}_k: \widehat{X}\to \widetilde{Y}$ that
	\begin{align*}
		\|\bm{g}\|_{\widetilde{Y}} = \|\Lambda \mathcal{F}_k(\bm{w})\|_{\widetilde{Y}} \leqslant C_k (\| \bm{w} \|_{\widetilde{X}}) \leqslant \widetilde{C}_{k,h},
	\end{align*}
where the positive constant $\widetilde{C}_{k,h}$ is independent of $\Lambda \in [0,1]$ and $\bm{g}$.
Choosing $R = \widetilde{C}_{k,h}$, we thus successfully verify the assertion \eqref{claim}.
This enables us to apply the Leray--Schauder fixed point theorem and establish the existence of a weak solution of the time discrete problem \eqref{time-discretization-u}--\eqref{time-discretization-theta}.

The proof of Proposition \ref{lemma-time-discretization} is complete.
\end{proof}

\subsection{Existence of weak solutions to the continuous problem}
\label{dis-existence}
The proof of Theorem \ref{weaksolution-3d} consists of several steps.
\medskip

\textbf{Step 1. Construction of the initial data for induction.}
Let $N$ be a given positive integer and $h=\frac{1}{N}$. We first construct the initial data $(\u^0,\phi^0,\vartheta^0)$ for the induction associated with the discrete problem \eqref{time-discretization-u}--\eqref{time-discretization-theta}.

Given $\phi_0 \in H^{1}(\Omega)$ satisfying $\|\phi_0\|_{L^{\infty}(\Omega)} \leqslant 1$, $|\overline{\phi_0}|<1$, as in \cite{JMFM}, we take $$\phi_0^N=\varphi\left(\frac{1}{N}\right),$$
where $\varphi$ is the (unique) solution of the following linear parabolic problem
\begin{equation}
\begin{cases}
\partial_t \varphi = \Delta \varphi, \qquad \ \,\text{in}\ \Omega\times (0,\infty),\\
\partial_\mathbf{n}\varphi=0,\qquad\ \ \ \  \text{on}\ \partial\Omega\times (0,\infty),\\
\varphi|_{t=0}=\phi_0,\qquad \ \text{in}\ \Omega.
\end{cases}
\notag
\end{equation}
By the classical theory of linear parabolic equations, we easily find $\varphi\in C([0,\infty);H^1(\Omega))\cap C((0,\infty);H^2_\mathbf{n}(\Omega))$ such that $\|\varphi(t)\|_{L^\infty(\Omega)}\leqslant 1$, and $\overline{\varphi(t)}=\overline{\phi_0}\in (-1,1)$ for all $t\geqslant 0$. As a consequence, we have
$$\phi_0^N\to \phi_0\quad \text{in}\ H^1(\Omega)\ \ \text{as}\ N\to \infty,$$
and due to the convexity of $F$,
$$F(\phi_0^N)\to F(\phi_0) \quad \text{in}\ \ L^1(\Omega).$$

Given $\theta_0\in L^\infty(\Omega)$, we choose
$$\theta_0^N=\varpi^N, \qquad \vartheta_0^N=\theta_0^N- \Theta_{\mathrm{b}}, $$
where $\varpi^N$ is the (unique) solution to the following linear elliptic problem
\begin{equation}
\begin{cases}
-\dfrac{1}{N} \Delta \varpi^N +\varpi^N = \theta_0,\quad \text{in}\ \Omega,\\
\varpi^N=\theta_\mathrm{b},\qquad \qquad \qquad \ \ \text{on}\ \partial \Omega,
\end{cases}
\notag
\end{equation}
and the function $\Theta_\mathrm{b}$ is given by \eqref{Theta_b}.
Applying the classical theory for linear elliptic equations, we have $\varpi^N-\Theta_\mathrm{b}\in H^2(\Omega)\cap H^1_0(\Omega)$. In addition, an argument similar to that for Lemma \ref{Stampacchia-truncation-method-remark} yields
$$
\min\Big\{\operatorname*{ess\,inf}_{\Omega}\theta_0,\,\min_{\partial\Omega}\theta_{\mathrm{b}}\Big\} \leqslant \varpi^N(x) \leqslant \max\Big\{\operatorname*{ess\,sup}_{\Omega}\theta_0,\,\max_{\partial\Omega}\theta_{\mathrm{b}}\Big\},\quad \text{a.e. in}\ \overline{\Omega},
$$
which is consistent with \eqref{maximumprinciple-discrete}.
Testing the equation for $\varpi^N$ by $\varpi^N-\Theta_\mathrm{b}$, using the Cauchy--Schwarz inequality and Poincar\'{e}'s inequality, we obtain
\begin{align*}
\frac{1}{N}\|\nabla (\varpi^N-\Theta_\mathrm{b})\|^2+ \| \varpi^N-\Theta_\mathrm{b}\|^2
\leqslant
\|\theta_0-\Theta_\mathrm{b}\|\|\varpi^N-\Theta_\mathrm{b}\|
\leqslant C\|\theta_0-\Theta_\mathrm{b}\|\|\nabla (\varpi^N-\Theta_\mathrm{b})\|,
\end{align*}
which implies
$$
\| \varpi^N-\Theta_\mathrm{b}\|\leqslant \|\theta_0-\Theta_\mathrm{b}\|,\qquad \frac{1}{N}\|\nabla (\varpi^N-\Theta_\mathrm{b})\| \leqslant C\|\theta_0-\Theta_\mathrm{b}\|.
$$
Next, testing the equation for $\varpi^N$ by $(-\Delta_D)^{-1}(\varpi^N-\theta_0)\in H^2(\Omega)\cap H^1_0(\Omega)$ (here $(-\Delta_D)^{-1}$ denotes the inverse Laplacian subject to the homogeneous Dirichlet boundary condition), we infer from the fact $\Delta \Theta_{\mathrm{b}} = 0$ in $\Omega$,
Cauchy--Schwarz inequality and the above estimates that
\begin{align*}
\|\varpi^N-\theta_0\|_{H^{-1}(\Omega)}^2
& = - \frac{1}{N} (\varpi^N-\Theta_\mathrm{b},\varpi^N-\theta_0)
\\
&\leqslant \frac{1}{N} \|\varpi^N-\Theta_\mathrm{b}\| (\|\varpi^N-\Theta_\mathrm{b}\|+ \|\Theta_\mathrm{b}-\theta_0\|)\\
&\leqslant \frac{2}{N}\|\theta_0-\Theta_\mathrm{b}\|^2.
\end{align*}
As a consequence, it holds
$$
\varpi^N\to \theta_0\quad \text{in}\ H^{-1}(\Omega)\ \ \text{as}\ \ N\to \infty.
$$
By the uniform $L^2$-bound for $\varpi^N$ and the uniqueness of the limit, we also have
$$
\varpi^N\rightharpoonup \theta_0\quad \text{in}\ L^2(\Omega)\ \ \text{as}\ \ N\to \infty.
$$
Hence, we observe that
\begin{align*}
\|\varpi^N-\theta_0\|^2
&= (\varpi^N-\Theta_\mathrm{b}, \varpi^N-\theta_0)- (\theta_0-\Theta_\mathrm{b}, \varpi^N-\theta_0)
\\
& = \frac{1}{N}(\varpi^N-\Theta_\mathrm{b}, \Delta (\varpi^N-\Theta_\mathrm{b}))- (\theta_0-\Theta_\mathrm{b}, \varpi^N-\theta_0)
\\
& = \underbrace{-\frac{1}{N}\|\nabla (\varpi^N-\Theta_\mathrm{b})\|^2}_{\leqslant 0}
-  (\theta_0-\Theta_\mathrm{b}, \varpi^N-\theta_0).
\end{align*}
This together with the weak convergence of $\varpi^N$ in $L^2(\Omega)$ yields
\begin{align}
0\leqslant \limsup_{N\to \infty} \|\varpi^N-\theta_0\|^2
\leqslant  \limsup_{N\to \infty} \big[- (\theta_0-\Theta_\mathrm{b}, \varpi^N-\theta_0)\big]
=0,
\notag
\end{align}
that is, the strong convergence of $\varpi^N$ in $L^2(\Omega)$. Hence, we have
$$
\theta^N_0\to \theta_0 \quad \text{and}\quad \vartheta^N_0\to \theta_0-\Theta_\mathrm{b}\quad \text{in}\ L^2(\Omega)\ \ \text{as}\ \ N\to \infty.
$$
Therefore, for the discrete problem \eqref{time-discretization-u}--\eqref{time-discretization-theta},
we set the initial data
$$(\u^0,\phi^0,\vartheta^0)=(\u_0, \phi^N_0,\vartheta^N_0).$$

\textbf{Step 2. Construction of approximate solutions and uniform estimates.} Like in \cite{JMFM}, we define the piecewise constant interpolant $f^{N}(t)$ on $[-h,\infty )$ through $f^{N}(t)=f^{k}$ for $t\in [(k-1)h,kh)$, where $k\in \mathbb{N}$ and $f\in
\{\u,\phi,\mu,\vartheta\}$. We also define $\rho^{N}:=\rho( \phi^{N})$.
Due to the implicit formulation for $\mu$, we simply set $\mu^N(t)=0$ for $t\in [-h,0)$.
Moreover, we introduce the time shifts $f_{h}:=f(t-h)$ and the notation for time differences as well as difference quotients:
\begin{align*}
\left( \Delta _{h}^{+}f\right) (t):=f(t+h)-f(t),\qquad &
\partial_{t,h}^{+}f(t):=\frac{1}{h}\left( \Delta _{h}^{+}f\right)(t), \\
\left( \Delta _{h}^{-}f\right) (t):=f(t)-f(t-h),\qquad &
\partial_{t,h}^{-}f(t):=\frac{1}{h}\left( \Delta _{h}^{-}f\right)(t).
\end{align*}

From \eqref{time-discretization-u}--\eqref{time-discretization-theta}, we can derive the corresponding time-continuous equations.
For an arbitrary vector $\bm{v}\in C_{0}^{\infty }\big( 0,\infty;\mathbf{C}_{0,\sigma}^{\infty}(\Omega)\big)$, we choose $\widetilde{\bm{v}}:=\int_{kh}^{(k+1)h}\bm{v}\,\mathrm{d}t$ as a test function in the weak formulation \eqref{time-discretization-u} and sum over $k\in \mathbb{Z}^+$ to get
\begin{align}
&\int_0^\infty\!\!\int_\Omega \left[- (\rho^N\u^N) \cdot \partial_{t,h}^{+} \v
- (\rho_h^N \u^N \otimes \u^N):\nabla \v
- (\u^N\otimes \mathbf{J}^N): \nabla \v \right]\,\mathrm{d}x\mathrm{d}t
    \notag \\
	&\qquad  +
	 \int_0^\infty\!\!\int_\Omega 2\nu(\phi_h^N, \vartheta_h^N+ \Theta_{\mathrm{b}}) D \u^N : D \v\,\mathrm{d}x\mathrm{d}t
-  \int_0^\infty\!\!\int_\Omega \lambda_0 a \mu^N \nabla \phi_h^N\cdot \v\,\mathrm{d}x\mathrm{d}t
    \notag \\
    &\quad = - \int_0^\infty\!\!\int_\Omega \lambda_0 b  (\vartheta_h^N+ \Theta_{\mathrm{b}}) (\nabla \phi^N \otimes \nabla \phi^N): \nabla \v  \,\mathrm{d}x\mathrm{d}t
    \notag\\
    &\qquad
    + \int_0^\infty\!\!\int_\Omega \bm{f}_{\mathrm{b}}(\phi_h^N, \vartheta_h^N+ \Theta_{\mathrm{b}})\cdot \v\,\mathrm{d}x\mathrm{d}t,
    \label{time-conti-u}
\end{align}
where
\begin{align}
& \mathbf{J}^N= -\frac{\rho_2-\rho_1}{2}m(\phi_h^N, \vartheta_h^N+ \Theta_{\mathrm{b}})\nabla \mu^N,\quad &&\text{a.e. in } \Omega\times (0,\infty),
 \label{time-conti-J}\\
& \mu^N  + c_W (\phi^N+\phi_h^N) = - \Delta \phi^N + F^{\prime}(\phi^N), \quad &&\text{a.e. in } \Omega\times (0,\infty).
\label{time-conti-mu}
\end{align}
Concerning the first term of \eqref{time-conti-u}, we have also used the following integration by parts with respect to time
\begin{align}
\int_0^\infty\!\!\int_\Omega \partial_{t,h}^{-} (\rho^N\u^N )\cdot \v\,\mathrm{d}x\mathrm{d}t=
-\int_0^\infty\!\!\int_\Omega (\rho^N\u^N) \cdot \partial_{t,h}^{+} \v\,\mathrm{d}x\mathrm{d}t.
\label{rhou-intgra}
\end{align}
Analogously, we have
\begin{align}
  &  \int_0^\infty\!\!\int_\Omega \left(\partial_{t,h}^{-}  \phi^N \zeta
  -  (\phi_h^N \u^N) \cdot \nabla \zeta  \right)\,\mathrm{d}x\mathrm{d}t
  + \int_0^\infty\!\!\int_\Omega m(\phi_h^N,\vartheta_h^N
  + \Theta_{\mathrm{b}}) \nabla \mu^N \cdot \nabla \zeta\,\mathrm{d}x\mathrm{d}t
  = 0,
\label{time-conti-phi} \\
	& \int_0^\infty\!\!\int_\Omega \left(\partial_{t,h}^{-} \vartheta^N  \xi
  -  (\vartheta^N \u^N) \cdot \nabla \xi  \right)\,\mathrm{d}x\mathrm{d}t
  + \int_0^\infty\!\!\int_\Omega \kappa(\phi_h^N,\vartheta_h^N
  + \Theta_{\mathrm{b}}) \nabla \vartheta^N \cdot \nabla \xi\,\mathrm{d}x\mathrm{d}t
\notag \\
&\quad = \int_0^\infty\!\!\int_\Omega \left[ (\Theta_{\mathrm{b}} \u^N) \cdot \nabla \xi - \kappa(\phi_h^N,\vartheta_h^N+ \Theta_{\mathrm{b}}) \nabla \Theta_{\mathrm{b}}  \cdot \nabla \xi \right]\,\mathrm{d}x\mathrm{d}t,
\label{time-conti-theta}
\end{align}
for all $\zeta\in C_0((0,\infty); C^\infty(\overline{\Omega}))$, $\xi \in C_0((0,\infty);C_0^\infty(\overline{\Omega}))$.

Define
$$
\widehat{E}_{\mathrm{tot}}(\u^{k},\phi^{k},\vartheta^{k})= (1+C_1 h)^k
E_{\mathrm{tot}}(\u^{k},\phi^{k},\vartheta^{k}),\quad k\in \mathbb{N}.
$$
From the discrete energy inequality \eqref{discrete-energy-inequality}, we obtain
\begin{align}
&\frac{\widehat{E}_{\mathrm{tot}}(\u^{k+1},\phi^{k+1},\vartheta^{k+1})- \widehat{E}_{\mathrm{tot}}(\u^{k},\phi^{k},\vartheta^{k})}{h}
\notag \\
&\quad =
(1+C_1 h)^k\left( \frac{E_{\mathrm{tot}}(\u^{k+1},\phi^{k+1},\vartheta^{k+1})- E_{\mathrm{tot}}(\u^{k},\phi^{k},\vartheta^{k})}{h}
+ C_1E_{\mathrm{tot}}(\u^{k+1},\phi^{k+1},\vartheta^{k+1})\right)
\notag\\
&\quad \leqslant C_2(1+C_1 h)^k.
\label{discre-ineq-1}
\end{align}
Summation over $k$ yields the following result
\begin{align}
\widehat{E}_{\mathrm{tot}}(\u^{k},\phi^{k},\vartheta^{k})
&\leqslant \widehat{E}_{\mathrm{tot}}(\u^{0},\phi^{0},\vartheta^{0}) +
C_2h \sum_{j=0}^{k-1} (1+C_1 h)^j,\quad \forall\, k\in \mathbb{Z}^+,
\notag
\end{align}
which implies
\begin{align}
E_{\mathrm{tot}}(\u^{k},\phi^{k},\vartheta^{k})
&\leqslant (1+C_1 h)^{-k} E_{\mathrm{tot}}(\u^{0},\phi^{0},\vartheta^{0}) +
\frac{C_2}{C_1}\left[1-(1+C_1 h)^{-k}\right]\notag \\
&\leqslant (1+C_1 h)^{-k} E_{\mathrm{tot}}(\u^{0},\phi^{0},\vartheta^{0}) +
\frac{C_2}{C_1},\quad \forall\, k\in \mathbb{Z}^+.
\label{discre-ineq-2}
\end{align}
Let $E^N(t)$ be the piecewise linear interpolant of $E_{\mathrm{tot}}(\u^k, \phi^k, \vartheta^k)$ at $t_k= kh$ given by
$$
E^N(t)= \frac{(k+1)h-t}{h} E_{\mathrm{tot}}(\u^k, \phi^k, \vartheta^k)
+ \frac{t-kh}{h} E_{\mathrm{tot}}(\u^{k+1}, \phi^{k+1}, \vartheta^{k+1})
$$
for $t\in [kh,(k+1)h)$, $k\in \mathbb{N}$. For $t\in (kh, (k+1)h)$, we also define the dissipation function
$$
D^N(t)= \frac{1}{4} \int_\Omega \left(
		 \underline{\nu} | \nabla \u^{k+1} |^2
        + \lambda_0 a \underline{m}  | \nabla \mu^{k+1} |^2
        + \underline{\kappa}  | \nabla \vartheta^{k+1} |^2 \right)\,\mathrm{d}x.
$$
Then it follows from \eqref{discrete-energy-inequality} that
\begin{align}
\frac{\mathrm{d}}{\mathrm{d} t} E^N(t) + D^N(t)
& \leqslant - C_1 E_{\mathrm{tot}}(\u^{k+1},\phi^{k+1}, \vartheta^{k+1})
 + C_2,
\label{conti-energy-inequality}
\end{align}
for all $t\in (kh,(k+1)h)$. Recalling that for all $k\in \mathbb{N}$, $E_{\mathrm{tot}}(\u^{k+1},\phi^{k+1}, \vartheta^{k+1})\geqslant -C_4$ for some positive constant $C_4$ that depends only on $\Omega$, we infer from \eqref{conti-energy-inequality} that
\begin{align}
E_{\mathrm{tot}}(\u^{j},\phi^{j}, \vartheta^{j})
+ \int_{t_i}^{t_j} D^N(t)\,\mathrm{d}\tau
\leqslant E_{\mathrm{tot}}(\u^{i},\phi^{i}, \vartheta^{i}) + (C_1C_4+C_2)(t_j-t_i),
\label{discre-ineq-3}
\end{align}
for all $t_i=ih$, $t_j=jh$ with $i, j\in \mathbb{N}$ and $j\geqslant i$.

From \eqref{discre-ineq-2}, \eqref{discre-ineq-3}, the construction of the initial data $(\u^{0},\phi^{0}, \vartheta^{0})$ and the $L^\infty$-estimates for $\phi^k$, $\theta^k$, we can deduce that
\begin{align*}
&\u^N\ \ \text{is bounded in}\ \ L^\infty(0,\infty; \H_\sigma)\cap L^2_{\mathrm{uloc}}([0,\infty);\V),
\\
&\phi^N \ \ \text{is bounded in}\ \ L^\infty(0,\infty;H^1(\Omega)),\quad \|\phi^N\|_{L^\infty(\Omega)}\leqslant 1\ \ \text{for all}\ t\geqslant 0,
\\
&F(\phi^N) \ \ \text{is bounded in}\ L^\infty(0,\infty; L^1(\Omega)),
\\
&\nabla \mu^N \ \ \text{is bounded in}\ \ L^2_{\mathrm{uloc}}([0,\infty);\bm{L}^2(\Omega)),
\\
&\mathbf{J}^N \ \ \text{is bounded in}\ \ L^2_{\mathrm{uloc}}([0,\infty);\bm{L}^2(\Omega)),
\\
& \vartheta^N\ \ \text{is bounded in}\ \ L^\infty(0,\infty; L^\infty(\Omega))\cap L^2_{\mathrm{uloc}}([0,\infty);H^1(\Omega)).
\end{align*}
Recalling  \eqref{notation-theta}, for $\theta^N:=\vartheta^N+\Theta_\mathrm{b}$, we have
\begin{align*}
& \theta^N\ \ \text{is bounded in}\ \ L^\infty(\Omega\times (0,\infty))\cap L^2_{\mathrm{uloc}}([0,\infty);H^1(\Omega)),
\end{align*}
and moreover,
\begin{align}
&  \min\Big\{\operatorname*{ess\,inf}_{\Omega}\theta^0,\,\min_{\partial\Omega}\theta_{\mathrm{b}}\Big\} \leqslant \theta^N(x) \leqslant \max\Big\{\operatorname*{ess\,sup}_{\Omega}\theta^0,\,\max_{\partial\Omega}\theta_{\mathrm{b}}\Big\},\quad \text{a.e. in}\ \overline{\Omega}.
\label{maximumprinciple-theta-N}
\end{align}
Applying Lemma \ref{lem:mu-es-k}, we can further deduce that
\begin{align*}
&\overline{\mu^N}\ \ \text{is bounded in}\ \ L^2_{\mathrm{uloc}}([0,\infty)),
\\
&  \mu^N \ \ \text{is bounded in}\ \ L^2_{\mathrm{uloc}}([0,\infty);H^1(\Omega)),
\\
&\phi^N \ \ \text{is bounded in}\ \ L^4_{\mathrm{uloc}}([0,\infty);H^2(\Omega))\cap  L^2_{\mathrm{uloc}}([0,\infty);W^{2,p}(\Omega)),
\\
&F'(\phi^N) \ \ \text{is bounded in}\ \ L^2_{\mathrm{uloc}}([0,\infty);L^p(\Omega)),
\end{align*}
for $p\in [2,6]$ if $d=3$, $p\in[2,\infty)$ if $d=2$.

In order to pass to the limit for all nonlinearities, we need to show strong
convergence of $\u^N$, $\phi^N$, $\vartheta^N$ (and also $\theta^N$), by investigating time derivatives of suitable approximate solutions.

Let $\widetilde{\phi}^{N}$ be the piecewise linear interpolant of $\phi^{N}(t^{k})$, where $t^{k}=kh$, $k\in \mathbb{N}$, that is, $%
\widetilde{\phi}^{N}=\frac{1}{h}\chi _{\lbrack 0,h]}\ast _{t}\phi^{N} $, where the convolution is taken only with respect to the time
variable $t$. Similarly, we define  $\widetilde{\vartheta}^{N}=\frac{1}{h}\chi _{\lbrack 0,h]}\ast
_{t}\vartheta^{N} $. Then it follows that
\begin{equation*}
\partial _{t}\widetilde{\phi }^{N}=\partial _{t,h}^{-}\phi ^{N},\quad
\partial _{t}\widetilde{\vartheta }^{N}=\partial _{t,h}^{-}\vartheta^{N},
\quad \text{for almost all}\ t\in (0,\infty),
\end{equation*}%
and
\begin{equation}
\| \widetilde{\phi }^{N}-\phi ^{N}\|_{(H^{1}(\Omega ))'}
\leqslant h\| \partial _{t}\widetilde{\phi }^{N}\|_{(H^{1}(\Omega))'},\quad
\| \widetilde{\vartheta }^{N}-\vartheta ^{N}\|_{H^{-1}(\Omega)}
\leqslant h\| \partial _{t}\widetilde{\vartheta}^{N}\|_{H^{-1}(\Omega)}.
\label{tilde-error}
\end{equation}%
From \eqref{time-conti-phi}, \eqref{time-conti-theta}, we can obtain
\begin{align*}
&\partial _{t,h}^{-}\phi ^{N} \ \ \text{is bounded in}\ \ L^2_{\mathrm{uloc}}([0,\infty);(H^1(\Omega))'),
\\
&\partial _{t,h}^{-}\vartheta^{N} \ \ \text{is bounded in}\ \
L^2_{\mathrm{uloc}}([0,\infty); H^{-1}(\Omega)),
\end{align*}
the first follows from the boundedness of $\phi^N\u^N$, $\nabla \mu^N$ in $L^2_{\mathrm{uloc}}([0,\infty);\bm{L}^2(\Omega))$, while the second follows from the boundedness of $\vartheta^N\u^N$ and $\nabla \vartheta^N$ in $L^2_{\mathrm{uloc}}([0,\infty);\bm{L}^2(\Omega))$.
By definition, we easily find
	\begin{align*}
		&\partial _{t}\widetilde{\phi }^{N} \ \ \text{is bounded in}\ \ L^2_{\mathrm{uloc}}([0,\infty);(H^1(\Omega))'),
		\\
		&\partial _{t}\widetilde{\vartheta }^{N} \ \ \text{is bounded in}\ \
		L^2_{\mathrm{uloc}}([0,\infty); H^{-1}(\Omega)).
\end{align*}

Next, let $\widetilde{\rho \u}^N$ be the piecewise linear interpolant of $\rho^N\u^N(t^k)$, where $t^{k}=kh$, $k\in \mathbb{N}$.
We can check that $\bm{P}(\widetilde{\rho \u}^N)$ is bounded in $L^\frac{4}{3}_{\mathrm{uloc}}([0,\infty);\V)$, which is a consequence of
the boundedness of $\u^N$ in $L^\infty(0,\infty;\H_\sigma)\cap L^2_{\mathrm{uloc}}([0,\infty); \V)$ and the boundedness of $\phi^N$ in $L^2_{\mathrm{uloc}}([0,\infty); W^{2,6}(\Omega))\cap L^\infty(\Omega\times (0,\infty))$.
Recalling the estimates (cf. \cite{JMFM})
\begin{align*}
&\rho_h^N \u^N\otimes\u^N \ \ \text{is bounded in}\ \ L^2_{\mathrm{uloc}}([0,\infty);\bm{L}^\frac32(\Omega)),
\\
& \u^N\otimes \mathbf{J}^N \ \ \text{is bounded in}\ \ L^\frac{8}{7}_{\mathrm{uloc}}([0,\infty);\bm{L}^\frac43(\Omega)),
\\
& \mu^N\nabla \phi_h^N \ \ \text{is bounded in}\ \ L^2_{\mathrm{uloc}}([0,\infty);\bm{L}^\frac32(\Omega)),
\end{align*}
and the following fact
\begin{align}
\|(\vartheta_h^N+ \Theta_{\mathrm{b}})(\nabla \phi^N\otimes \nabla \phi^N)\|_{L^2(\Omega)}
 & \leqslant \|\theta_h^N\|_{L^\infty(\Omega)}\|\nabla \phi^N\|_{L^\infty(\Omega)}\| \nabla \phi^N\|\notag
 \\
 &\leqslant C\|\theta_h^N\|_{L^\infty(\Omega)}\|\phi^N\|_{W^{2,4}(\Omega)}\|\phi^N\|_{H^1(\Omega)},\notag
\end{align}
which implies
\begin{align*}
(\vartheta_h^N+ \Theta_{\mathrm{b}})(\nabla \phi^N\otimes \nabla \phi^N)\ \ \text{is bounded in}\ \ L^2_{\mathrm{uloc}}([0,\infty);\bm{L}^2(\Omega)),
\end{align*}
then from \eqref{time-conti-u} and \eqref{rhou-intgra}, we can deduce that
\begin{align*}
\partial_t(\widetilde{\rho \u}^N)=\partial^{-}_{t,h}(\rho^N\u^N)\ \ \text{is bounded in}\ \ L^\frac{8}{7}_{\mathrm{uloc}}([0,\infty);\bm{W}^{-1,4}(\Omega)).
\end{align*}

\textbf{Step 3. Passage to the limit as $N\to \infty$.}
In Step 2, we have collected all the ingredients necessary to draw the conclusion of Theorem \ref{weaksolution-3d}. In particular, all the estimates obtained above are independent of the parameters $N$, $h$ and $k$. Hence, applying the same compactness argument as in \cite{JMFM,Wangxiaoming} with minor modifications, we can pass to the limit as $N \to \infty$ (always understood in the sense of a convergent subsequence) to obtain a global weak solution $(\u,\phi,\mu,\vartheta)$ (and thus $\theta=\vartheta+\Theta_\mathrm{b}$) of  problem \eqref{NSCHM}--\eqref{initial} on $[0,\infty)$ with required regularity properties.

\begin{remark}\rm
When applying the compactness argument in \cite{JMFM,Wangxiaoming}, since iteration step $k$ can be taken arbitrarily large, for any given final time $L\in \mathbb{Z}^+$, we can extract a convergent subsequence $\{(\u^N,\phi^N,\mu^N,\theta^N)\}$ on the finite interval $[0,L]$ to find a weak solution $(\u, \phi, \mu, \theta)$ to problem \eqref{NSCHM}--\eqref{initial} defined on $[0,L]$.
However, since the uniqueness of weak solutions is not known, with uniform-in-time estimates for the approximate solutions, we cannot simply ``extend'' the weak solution from $[0,L]$ to the whole interval $[0,\infty)$ by letting $L\to \infty$. This issue can be solved by applying a diagonal argument as in \cite{BF2013} for the classical Navier--Stokes equations. Roughly speaking, whenever we obtain a convergent subsequence $\{(\u^N,\phi^N,\mu^N,\theta^N)\}$ on $[0,L]$, $L\in \mathbb{Z}^+$, we extract a further subsequence that converges on $[0,L+1]$. The diagonal argument then enables us to obtain a sequence $\{(\u^N,\phi^N,\mu^N,\theta^N)\}$ (not relabeled for simplicity) that converges on $[0,T]$ for any $T>0$, as $N\to \infty$.
\end{remark}

In the following, we will not repeat the limiting process as $N\to \infty$ mentioned above. To finish the proof, we just sketch the convergence for the nonlinear term $\theta_h^N(\nabla \phi^N\otimes \nabla \phi^N)$ due to the Marangoni effect. Below the convergence as $N\to\infty$ will be understood in the sense of a subsequence, and the convergent subsequence will not be relabeled for simplicity.

With the uniform estimates obtained in Step 2, we recall that the following weak and strong convergence results can be obtained for $\phi^N$ as $N\to \infty$ (see \cite[Section 5.1]{JMFM})
\begin{align*}
& \phi^N \to \phi, \qquad  \text{weakly star in}\ \ L^\infty(0,\infty;H^1(\Omega)),\\
& \phi^N \to \phi, \qquad  \text{strongly in}\ \ L^2(0,T;H^1(\Omega)), \ \ \text{for any} \ T\in(0,\infty),
\end{align*}
for some function $\phi\in L^\infty(0,\infty;H^1(\Omega))\cap C_w([0,\infty);H^1(\Omega))$. Moreover, the boundedness of $\phi^N$ in $L^4_{\mathrm{uloc}}([0,\infty);H^2(\Omega))\cap L^2_{\mathrm{uloc}}([0,\infty);W^{2,6}(\Omega))$ also implies
\begin{align*}
& \phi^N \to \phi, \qquad  \text{weakly in}\ \ L^4(0,T;H^2(\Omega))\cap L^2(0,T;W^{2,6}(\Omega)), \ \ \text{for any} \ T\in(0,\infty),
\end{align*}
with $\phi\in L^4_{\mathrm{uloc}}([0,\infty);H^2(\Omega))\cap L^2_{\mathrm{uloc}}([0,\infty);W^{2,6}(\Omega))$.
The linear dependence of $\rho$ on $\phi$ enables us to derive analogous results for $\rho^N$. Similar arguments also apply to $\vartheta^N$. First, we easily find that as $N\to \infty$, it holds
\begin{align*}
& \vartheta^N \to \vartheta, \qquad  \text{weakly star in}\ \ L^\infty(\Omega \times (0,\infty) ),\\
& \vartheta^N \to \vartheta, \qquad  \text{weakly in}\ \ L^2(0,T;H^1(\Omega)),\ \ \text{for any} \ T\in(0,\infty),
\end{align*}
for some function $\vartheta\in L^\infty(\Omega \times (0,\infty)) \cap L^2_{\mathrm{uloc}}([0,\infty);H^1(\Omega))$.
Next, by the uniform estimates for $\widetilde{\vartheta}^N$ and the Aubin--Lions compactness lemma, we get
\begin{align*}
& \widetilde{\vartheta}^N \to \widetilde{\vartheta}, \qquad  \text{strongly in}\ \ L^2(0,T;L^2(\Omega)), \ \ \text{for any} \ T\in(0,\infty),
\end{align*}
for some function $\widetilde{\vartheta}\in L^2_{\mathrm{uloc}}([0,\infty);H^1(\Omega)) \cap H^1_{\mathrm{uloc}}([0,\infty); H^{-1}(\Omega))$.
Then it follows from the Lions--Magenes lemma (see e.g., \cite[Lemma 7.1]{Roubicek2005}, or \cite[Theorem 4.10.2, Chapter III]{Amann}) that the limit  $\widetilde{\vartheta}$ satisfies  $\widetilde{\vartheta} \in C([0,\infty);L^2(\Omega))\cap L^\infty(0,\infty;L^2(\Omega))$.
On the other hand, \eqref{tilde-error} implies that as $N\to \infty$,
$$
\vartheta^N-\widetilde{\vartheta}^N \to 0 \qquad \text{strongly in}\ \ L^2(0,T;H^{-1}(\Omega)), \ \ \text{for any} \ T\in(0,\infty),
$$
since $\partial _{t}\widetilde{\vartheta }^{N}$ is uniformly bounded in $L^2_{\mathrm{uloc}}([0,\infty); H^{-1}(\Omega))$, and $h \to 0$.
As a result, we have $\vartheta=\widetilde{\vartheta}$. From the fact
\begin{align*}
\|\vartheta^N -\vartheta\|
&\leqslant \| \vartheta^N-\widetilde{\vartheta}^N\|+\|\widetilde{\vartheta}^N - \widetilde{\vartheta}\|\\
&\leqslant \| \vartheta^N-\widetilde{\vartheta}^N\|_{H^1(\Omega)}^\frac12
\| \vartheta^N-\widetilde{\vartheta}^N\|_{H^{-1}(\Omega)}^\frac12+\|\widetilde{\vartheta}^N - \widetilde{\vartheta}\|,
\end{align*}
we arrive at
$$
\vartheta^N  \to \vartheta  \qquad  \text{strongly in}\ \ L^2(0,T;L^2(\Omega)), \ \ \text{for any} \ T\in(0,\infty).
$$
By the $L^\infty$-boundedness of $\vartheta^N$, $\vartheta$ and the interpolation, for any $q\in (2,\infty)$, we get
$$
\vartheta^N  \to \vartheta  \qquad  \text{strongly in}\ \ L^q(0,T;L^q(\Omega)), \ \ \text{for any} \ T\in(0,\infty).
$$
Recalling the definition of $\vartheta_h^N$, we also find that
$$
\vartheta_h^N  \to \vartheta  \qquad  \text{strongly in}\ \ L^q(0,T;L^q(\Omega)), \ \ \text{for any} \ T\in(0,\infty).
$$

Let $T\in (0,\infty)$ be given. With the above convergence results, for any $\bm{v}\in L^4(0,T;\bm{W}^{1,3}(\Omega))$, we consider the difference
\begin{align*}
& \int_0^T \!\! \int_\Omega \big[\theta_h^N(\nabla \phi^N\otimes \nabla \phi^N)- \theta(\nabla \phi\otimes \nabla \phi)\big] : \nabla \bm{v}\,\mathrm{d}x\mathrm{d}t
\\
&\quad =
 \int_0^T\!\!\int_\Omega  (\vartheta_h^N-\vartheta) (\nabla \phi^N\otimes \nabla \phi^N): \nabla \bm{v}\,\mathrm{d}x\mathrm{d}t
+ \int_0^T\!\!\int_\Omega  \big[\theta  \nabla \phi^N\otimes (\nabla \phi^N-\nabla \phi)\big]: \nabla \bm{v}\,\mathrm{d}x\mathrm{d}t\\
&\qquad + \int_0^T\!\!\int_\Omega \big[ \theta( \nabla \phi^N -\nabla \phi)\otimes \nabla \phi\big]: \nabla \bm{v}\,\mathrm{d}x\mathrm{d}t.
\end{align*}
The first term on the right-hand side can be estimated as follows
\begin{align*}
&\left|\int_0^T\!\!\int_\Omega  (\vartheta_h^N-\vartheta) (\nabla \phi^N\otimes \nabla \phi^N): \nabla \bm{v}\,\mathrm{d}x\mathrm{d}t\right|
\\
&\quad \leqslant  \int_0^T \|\vartheta_h^N-\vartheta\|_{L^6(\Omega)} \|\nabla \phi^N\|_{L^4(\Omega)}^2\|\nabla \bm{v}\|_{L^3(\Omega)}\,\mathrm{d}t\\
&\quad
\leqslant  C \int_0^T \|\vartheta_h^N-\vartheta\|_{L^6(\Omega)} \|\phi^N\|_{H^2(\Omega)}^\frac{3}{2}\|\phi^N\|_{H^1(\Omega)}^\frac12\|\nabla \bm{v}\|_{L^3(\Omega)}\,\mathrm{d}t
\\
&\quad \leqslant \|\vartheta_h^N-\vartheta\|_{L^6(0,T;L^6(\Omega))} \|\phi^N\|_{L^4(0,T;H^2(\Omega))}^\frac{3}{2} \|\phi^N\|_{L^\infty(0,T;H^1(\Omega))}^\frac12
\|\nabla \bm{v}\|_{L^\frac{24}{11}(0,T;L^3(\Omega))}
\\
&\quad \to 0\quad \text{as} \ N\to \infty.
\end{align*}
Concerning the second term, we see that
\begin{align*}
& \left| \int_0^T\!\!\int_\Omega  \big[\theta  \nabla \phi^N\otimes (\nabla \phi^N-\nabla \phi)\big]: \nabla \bm{v}\,\mathrm{d}x\mathrm{d}t\right|
\\
&\quad \leqslant \|\theta\|_{L^\infty(\Omega\times(0,T))}
\int_0^T \|\nabla \phi^N\|_{L^6(\Omega)}\|\nabla \phi^N-\nabla \phi\|\|\nabla \bm{v}\|_{L^3(\Omega)}\,\mathrm{d}t\\
&\quad \leqslant C \|\theta\|_{L^\infty(\Omega\times(0,T))}
\|\phi^N\|_{L^4(0,T;H^2(\Omega))}\|\phi^N-\phi\|_{L^2(0,T;H^1(\Omega))}\|\nabla \bm{v}\|_{L^4(0,T;L^3(\Omega))}\\
&\quad \to 0\quad \text{as} \ N\to \infty.
\end{align*}
Finally, noticing that $\nabla \phi^N\to \nabla \phi$ weakly in $L^4(0,T; L^6(\Omega))$ and
\begin{align*}
\|\theta\partial_i\phi\partial_j\bm{v}\|_{L^\frac{4}{3}(0,T;L^\frac{6}{5}(\Omega))}
& \leqslant \|\theta\|_{L^\infty(\Omega\times(0,T))}
\|\partial_i\phi\|_{L^4(0,T;L^2(\Omega))} \|\partial_j\bm{v}\|_{L^2(0,T;L^3(\Omega))},
\end{align*}
we obtain the convergence for the third term
\begin{align*}
\int_0^T\!\!\int_\Omega \big[ \theta( \nabla \phi^N -\nabla \phi)\otimes \nabla \phi\big]: \nabla \bm{v}\,\mathrm{d}x\mathrm{d}t \to 0\quad \text{as} \ N\to \infty.
\end{align*}
Hence, for any test function $\bm{v}\in L^4(0,T;\bm{W}^{1,3}(\Omega))$, it holds
\begin{align*}
& \int_0^T \!\! \int_\Omega \big[\theta_h^N(\nabla \phi^N\otimes \nabla \phi^N)- \theta(\nabla \phi\otimes \nabla \phi)\big] : \nabla \bm{v}\,\mathrm{d}x\mathrm{d}t \to 0\quad \text{as} \ N\to \infty.
\end{align*}

Using the convergence results above, we can pass to the limit as $N\to \infty$ in
\eqref{time-conti-u}--\eqref{time-conti-theta} to obtain \eqref{weaku}--\eqref{weakmu}. In addition, following the argument in \cite{JMFM}, we see that the initial data $(\u_0,\phi_0,\theta_0)$ can be attained.
 	
The proof of Theorem \ref{weaksolution-3d} is complete. \qed

\section{Uniqueness in Two Dimensions} \label{section-of-weak-2d}
This section is devoted to the proof of Theorem \ref{weaksolution-2d}. We first investigate an auxiliary problem for the convective heat equation with a pair of given data $(\u, \phi)$. By analyzing the regularity of solutions to the auxiliary problem, we establish the existence of global weak solutions to the two-dimensional problem \eqref{NSCHM}--\eqref{initial} with improved regularity properties for $\theta$, provided that the initial temperature satisfies $\theta_0 \in C^{\gamma}(\overline{\Omega}) \cap H_0^1(\Omega)$ for some $\gamma \in (0,1)$.
After that, we prove the uniqueness result in two dimensions under additional structural assumptions.

\subsection{An auxiliary problem}
Let $(\bm{u},\phi)$ be given functions. We consider the following initial boundary value problem of a convective heat equation:
\begin{align}
&\partial_t \theta + \u \cdot \nabla \theta =
\mathrm{div} ( \kappa(\phi, \theta) \nabla \theta ), \quad \,\text{in}~ \Omega \times (0,\infty),  \label{uphitotheta1} \\
&\theta |_{\partial \Omega} = \theta_{\mathrm{b}},\qquad\qquad \qquad  \qquad \qquad\, \text{on}\ \partial\Omega\times (0,\infty),
\label{uphitotheta2}\\
&\theta |_{t = 0} = \theta_0, \qquad \qquad \qquad \qquad \quad \ \ \  \text{in}\ \Omega.
\label{uphitotheta3}
\end{align}

\begin{proposition} \label{uphitotheta}
Suppose that $\Omega \subset \mathbb{R}^2$ is a bounded domain with $C^2$-boundary $\partial \Omega$, and the assumption $\mathbf{(A3)}$ is  satisfied.
\begin{enumerate}[label=\textup{(\roman*)}]
\item Assume that $\u \in L^2_{\mathrm{uloc}}([0,\infty) ; \bm{H}_{\sigma})$, $\phi\in L^\infty(\Omega \times (0,\infty))$, $\theta_0 \in L^{\infty}(\Omega)$, $\theta_{\mathrm{b}} \in H^{\frac{3}{2}}(\partial \Omega)$.
Then problem \eqref{uphitotheta1}--\eqref{uphitotheta3} admits a global weak solution that satisfies
$$
\theta \in L^\infty(\Omega\times (0,\infty)) \cap
L^2_{\mathrm{uloc}} ([0, \infty) ; H^1(\Omega))
\cap
H^1_{\mathrm{uloc}} ([0, \infty) ; H^{-1}(\Omega)),
$$
and for almost all $(x,t) \in \Omega \times (0, \infty)$,
\begin{equation}\label{theta_bound}
		\min\left\{
		\operatorname*{ess\,inf}_{\Omega}\theta_0,\,
		\min_{\partial\Omega}\theta_{\mathrm{b}}
		\right\}
        \leqslant
		\theta(x,t) 
        \leqslant
		\max\left\{
		\operatorname*{ess\,sup}_{\Omega}\theta_0,\,
		\max_{\partial\Omega}\theta_{\mathrm{b}}
		\right\}.
\end{equation}
\item Assume that
\begin{align*}
& \u \in L^{\infty}(0, \infty ; \bm{H}_{\sigma}) \cap L^2_{\mathrm{uloc}}([0,\infty) ; \V),\\
& \phi \in L^{\infty}(0, \infty ; H^1(\Omega)) \cap L^2_{\mathrm{uloc}}([0, \infty); H^2_\mathbf{n}(\Omega))\cap L^\infty(\Omega\times (0,\infty)),\\
& \theta_0 \in C^{\gamma}(\overline{\Omega}) \cap H^1(\Omega)\ \ \text{for some}\ \gamma \in (0,1),\quad \theta_0|_{\partial\Omega}=\theta_{\mathrm{b}} \in H^{\frac{3}{2}}(\partial \Omega).
\end{align*}
Then problem \eqref{uphitotheta1}--\eqref{uphitotheta3} admits a global strong solution that satisfies
\begin{equation}
\theta \in L^{\infty} (0, \infty ; H^1(\Omega) \cap C^{\beta}(\overline{\Omega}) )
\cap
L_{\mathrm{uloc}}^2([0, \infty) ; H^2(\Omega))
\cap
H_{\mathrm{uloc}}^1([0, \infty); L^2(\Omega)),
\label{theta_regularity}
\end{equation}
for any $\beta \in (0, \gamma]$.

\item Let the assumptions in (ii) be satisfied.
Suppose that $\theta_1$ is a global strong solution and $\theta_2$ is a global weak solution to problem \eqref{uphitotheta1}--\eqref{uphitotheta3} subject to the same boundary and initial data.
Then we have $\theta_1(t) = \theta_2(t)$ for all $t\geqslant 0$.
\end{enumerate}
\end{proposition}
\begin{proof}
The existence of a global weak/strong solution to problem \eqref{uphitotheta1}--\eqref{uphitotheta3} can be proven by a standard argument as in \cite{LB99}. In the following, we simply perform the necessary \emph{a priori} estimates.

(i) The Stampacchia method easily yields the estimate \eqref{theta_bound}, cf. e.g., \cite[Lemma 3.5]{LB99}. Next, using \eqref{Theta_b} with $\theta_\mathrm{b}\in H^\frac32(\partial\Omega)$, we rewrite \eqref{uphitotheta1} as (cf. \eqref{Eq:vartheta})
\begin{align}
& \partial_t \vartheta + \u \cdot \nabla \vartheta - \mathrm{div}\,(\kappa(\phi,\vartheta+ \Theta_{\mathrm{b}}) \nabla \vartheta)
= -\u \cdot \nabla \Theta_{\mathrm{b}} + \mathrm{div}\,(\kappa(\phi,\vartheta+ \Theta_{\mathrm{b}}) \nabla \Theta_{\mathrm{b}}),
\label{2D:vartheta}
\end{align}
where $\vartheta=\theta-\Theta_\mathrm{b}$ with $\Theta_\mathrm{b}\in H^2(\Omega)$. Multiplying \eqref{2D:vartheta} by $\vartheta$ and integrating over $\Omega$, we have
\begin{align}
& \frac{1}{2} \frac{\d}{\dt} \| \vartheta \|^2 + \int_{\Omega} \kappa(\phi, \theta) | \nabla \vartheta |^2 \dx \notag  \\
&\quad = -\int_\Omega (\u \cdot \nabla \Theta_{\mathrm{b}})\vartheta\,\mathrm{d}x -\int_\Omega \kappa(\phi,\theta) \nabla \Theta_{\mathrm{b}}\cdot \nabla \vartheta\,\mathrm{d}x \notag \\
&\quad \leqslant \|\bm{u}\|\|\nabla \Theta_\mathrm{b}\|_{L^4(\Omega)}\|\vartheta\|_{L^4(\Omega)} +\|\kappa\|_{L^\infty(\Omega)}\|\nabla \Theta_\mathrm{b}\|\|\nabla \vartheta\| \notag \\
&\quad \leqslant \frac{\underline{\kappa}}{2} \|\nabla \vartheta\|^2 + C\|\Theta_\mathrm{b}\|_{H^2(\Omega)}^2\|\bm{u}\|^2  + C\|\Theta_\mathrm{b}\|_{H^1(\Omega)}^2.
\label{2D-L2a}
\end{align}
In the above estimate, we have used the $L^\infty$-bounds of $\phi$, $\theta$ and the assumption $\mathbf{(A3)}$. By Poincar\'{e}'s inequality and the elliptic estimate, we get
\begin{align}
\frac{\d}{\dt} \| \vartheta \|^2 + C\underline{\kappa}\|\vartheta \|^2 \leqslant C\|\theta_\mathrm{b}\|_{H^\frac32(\partial\Omega)}^2(\|\bm{u}\|^2+1),\label{2D-L2b}
\end{align}
where the positive constant $C$ depends on $\Omega$, $\kappa$ and $L^\infty$-bounds of $\phi$, $\theta$. Since $\u \in L^2_{\mathrm{uloc}}([0,\infty) ; \bm{H}_{\sigma})$, we can apply a Gronwall-type inequality (see \cite[Lemma 2.5]{GGP}) and obtain
\begin{align}
\|\vartheta(t)\|^2\leqslant 2 \|\vartheta(0)\|^2e^{-C\underline{\kappa}t} + \frac{2e^{C\underline{\kappa}}}{1-e^{-C\underline{\kappa}}}\sup_{t\geqslant 0}\int_t^{t+1}\|\bm{u}(\tau)\|^2\,\mathrm{d}\tau.
\label{2D-L2c}
\end{align}
Integrating \eqref{2D-L2a} with respect to time, we infer from \eqref{2D-L2c} that
\begin{align}
\int_t^{t+1}\|\nabla \vartheta(\tau)\|^2\,\mathrm{d}\tau\leqslant C, \quad \forall\, t\geqslant 0,
\label{2D-L2d}
\end{align}
where $C>0$ is independent of $t$. The estimates \eqref{2D-L2c} and \eqref{2D-L2d} imply that $\theta \in L^2_{\mathrm{uloc}}([0, \infty) ; H^1(\Omega))$. By comparison in the equation for $\theta$, we further get  $\partial_t\theta\in L^2_\mathrm{uloc}([0,\infty);H^{-1}(\Omega))$.
\smallskip

(ii) Keeping in mind that $\Theta_\mathrm{b}$ is independent of time and $\Delta \Theta_\mathrm{b}=0$, we test \eqref{2D:vartheta} by $-\Delta \vartheta$ and integrate over $\Omega$. This gives
\begin{align}
&\frac{1}{2} \frac{\d}{\dt} \| \nabla \vartheta \|^2
+ \int_{\Omega} \kappa(\phi, \theta) | \Delta \vartheta |^2 \dx \notag \\
&\quad = \int_{\Omega} (\u \cdot \nabla \theta) \Delta \vartheta \dx
- \int_{\Omega} \partial_1 \kappa (\nabla \phi \cdot \nabla \theta) \Delta \vartheta \dx
- \int_{\Omega} \partial_2 \kappa |\nabla \theta|^2 \Delta \vartheta \dx.
\label{2D-H1a}
\end{align}
The first term on the right-hand side can be estimated using \eqref{GL-L4} for $\theta$ and the elliptic estimate, namely,
\begin{align*}
\int_{\Omega} (\u \cdot \nabla \theta) \Delta \vartheta \dx
& =\int_{\Omega} (\u \cdot \nabla \theta) \Delta \theta \dx
= - \int_{\Omega} \nabla \u : (\nabla \theta \otimes \nabla \theta) \dx \\
&\leqslant C \| \nabla \u \| \| \nabla \theta \|_{L^4(\Omega)}^2 \\
&\leqslant C \| \nabla \u \| \| \theta \|_{L^{\infty}(\Omega)} \big( \| \Delta \theta \| + \| \theta \| + \|\theta_\mathrm{b}\|_{H^\frac32(\partial\Omega)} \big) \\
&\leqslant \frac{\underline{\kappa}}{6} \| \Delta \vartheta \|^2
+ C (\| \nabla \u \|^2 + 1 ).
\end{align*}
Since $\u \in L^{\infty}(0, \infty ; \bm{H}_{\sigma}) \cap L^2_{\mathrm{uloc}}([0,\infty) ; \V)$, by interpolation we have
$ \u \in L^4_{\mathrm{uloc}}([0,\infty) ; \bm{L}^4(\Omega) )$.
For initial datum $\theta_0\in C^{\gamma}(\overline{\Omega}) \cap H^1(\Omega)$, we can apply the same argument as in \cite[Lemma 3.2]{ZZF2013} to conclude that
\begin{align}
	\|\theta\|_{L^{\infty} (0,\infty; C^{\beta}(\overline{\Omega}))} \leqslant C, \label{thetaalpha}
\end{align}
for some $\beta \in(0, \gamma]$. By the H\"{o}lder estimate \eqref{thetaalpha} and the elliptic estimate, we deduce that
\begin{align*}
- \int_{\Omega} \partial_2 \kappa |\nabla \theta|^2 \Delta \vartheta \dx
&\leqslant \|\partial_2 \kappa\|_{L^{\infty}(\Omega)} \|\nabla \theta\|^2_{L^4(\Omega)} \| \Delta \vartheta\|    \\
&\leqslant C \|\theta\|_{C^{\beta}(\overline{\Omega})}^{2 \xi} \| \theta\|_{H^2(\Omega)}^{2-2\xi} \| \Delta \vartheta\|  \\
&\leqslant C \|\theta\|_{C^{\beta}(\overline{\Omega})}^{2 \xi} \big(\| \Delta \theta \| + \| \theta \| + \|\theta_\mathrm{b}\|_{H^\frac32(\partial\Omega)} \big)^{2-2\xi}
\| \Delta \vartheta\| \\
&\leqslant \frac{\underline{\kappa}}{6} \| \Delta \vartheta\|^2 + C.
\end{align*}	
Here, we have used the following interpolation inequality (see \cite[Lemma 2.2]{me})
\begin{align} \label{GiorginiHolder}
	\| \theta \|_{W^{1,4}(\Omega)} \leqslant C \| \theta \|_{C^{\beta}(\overline{\Omega})}^{\xi} \| \theta \|_{H^2(\Omega)}^{1 - \xi},\quad
	\text{for some } \xi \in \left(\frac{1}{2},1\right).
\end{align}
Finally, applying the Gagliardo--Nirenberg inequality, H\"{o}lder's inequality and Young's inequality, we have
\begin{align*}
- \int_{\Omega} \partial_1 \kappa (\nabla \phi \cdot \nabla \theta) \Delta \vartheta \dx
&\leqslant \|\partial_1 \kappa\|_{L^{\infty}(\Omega)} \| \nabla \phi \|_{L^4(\Omega)} \| \nabla \theta \|_{L^4(\Omega)} \| \Delta \vartheta \| \\
&\leqslant C \| \phi \|_{H^1(\Omega)}^{\frac12} \| \phi \|_{H^2(\Omega)}^{\frac12}
\|\theta\|_{C^{\beta}(\overline{\Omega})}^{\xi} \| \theta\|_{H^2(\Omega)}^{1-\xi} \| \Delta \vartheta\| \\
&\leqslant C \| \phi \|_{H^1(\Omega)}^{\frac12} \| \phi \|_{H^2(\Omega)}^{\frac12}
\|\theta\|_{C^{\beta}(\overline{\Omega})}^{\xi} \big(\| \Delta \theta\| + \| \theta \| +  \|\theta_\mathrm{b}\|_{H^\frac32(\partial\Omega)} \big)^{1-\xi} \| \Delta \vartheta\| \\
&\leqslant \frac{\underline{\kappa}}{6} \| \Delta \vartheta\|^2
+ C ( \| \phi \|_{H^2(\Omega)}^{\frac{1}{\xi}} + 1) \\
&\leqslant \frac{\underline{\kappa}}{6} \| \Delta \theta\|^2
+ C ( \| \phi \|_{H^2(\Omega)}^2 + 1).
\end{align*}
Combining the above estimates, from \eqref{2D-H1a} we infer that
\begin{align}
\frac{\d}{\dt} \| \nabla \vartheta \|^2 + \underline{\kappa} \| \Delta \vartheta \|^2
\leqslant C ( \| \nabla \u \|^2 + \| \phi \|_{H^2(\Omega)}^2 + 1 ).
\label{2D-H1b}
\end{align}
Using arguments similar to those for \eqref{2D-L2c}, \eqref{2D-L2d}, we can deduce from \eqref{2D-H1b}
that $$\vartheta \in L^{\infty} (0, \infty ; H^1(\Omega)) \cap L_{\mathrm{uloc}}^2([0, \infty) ; H^2(\Omega)\cap H^1_0(\Omega)).$$
Again, a comparison in equation \eqref{2D:vartheta} yields $\partial_t\vartheta\in L^2_{\mathrm{uloc}}([0,\infty);H^{-1}(\Omega))$. Using $\theta=\vartheta+\Theta_\mathrm{b} $, we arrive at the conclusion \eqref{theta_regularity}.
\smallskip

(iii) Now, we investigate the weak-strong uniqueness. The difference of solutions $\theta_1-\theta_2$ satisfies
\begin{equation}
\partial_t (\theta_1 - \theta_2) + \u \cdot \nabla (\theta_1 - \theta_2)
= \mathrm{div} \big(
( \kappa(\phi, \theta_1) - \kappa(\phi, \theta_2) ) \nabla \theta_1
+ \kappa(\phi, \theta_2) \nabla (\theta_1 - \theta_2 )
\big),
\label{eq-diff}
\end{equation}
in the sense of distributions. Since $\theta_1$ (resp. $\theta_2$) is a strong (resp. weak) solution, $\theta_1 - \theta_2$ admits the same regularity property as the weak solution $\theta_2$ and thus is valid as a test function.
Multiplying \eqref{eq-diff} by $\theta_1 - \theta_2$, integrating over $\Omega$, we get
\begin{align*}
\frac{1}{2} \frac{\d}{\dt} \| \theta_1 - \theta_2 \|^2
+ \int_{\Omega} \kappa(\phi, \theta_2) | \nabla(\theta_1 - \theta_2) |^2 \dx
= -\int_{\Omega} ( \kappa(\phi, \theta_1) - \kappa(\phi, \theta_2) ) \nabla \theta_1 \cdot \nabla(\theta_1 - \theta_2 ) \dx.
\end{align*}
Using the Gagliardo--Nirenberg inequality, the term on the right-hand side can be estimated as follows
\begin{align}
& -\int_{\Omega} ( \kappa(\phi, \theta_1) - \kappa(\phi, \theta_2) ) \nabla \theta_1 \cdot \nabla(\theta_1 - \theta_2 ) \dx \notag \\
&\quad \leqslant
\| \partial_2 \kappa \|_{L^{\infty}(\Omega)} \| \theta_1 - \theta_2 \|_{L^4(\Omega)} \| \nabla \theta_1 \|_{L^4(\Omega)} \| \nabla(\theta_1 - \theta_2) \| \notag \\
&\quad \leqslant
C \| \theta_1 - \theta_2 \|^{\frac12} \| \nabla (\theta_1 - \theta_2) \|^{\frac{3}{2}}
\| \theta_1 \|_{L^{\infty}(\Omega)}^{\frac12} \| \theta_1 \|_{H^2(\Omega)}^{\frac12}
\notag \\
&\quad \leqslant \frac{\underline{\kappa}}{2} \| \nabla(\theta_1 - \theta_2) \|^2
+ C \| \theta_1 \|_{H^2(\Omega)}^2 \| \theta_1 - \theta_2 \|^2. \label{theta-uniqueness}
\end{align}
This leads to the following inequality
\begin{align*}
\frac{\d}{\dt} \| \theta_1 - \theta_2 \|^2
+  \underline{\kappa} \| \nabla(\theta_1 - \theta_2) \|^2
\leqslant C \| \theta_1 \|_{H^2}^2 \| \theta_1 - \theta_2 \|^2.
\end{align*}
Since $\theta_1 \in L_{\mathrm{uloc}}^2([0,\infty);H^2(\Omega))$, we can apply Gronwall's Lemma to conclude that $\theta_1(t) = \theta_2(t)$ in $[0, T]$ for any $T>0$.

The proof of Proposition \ref{uphitotheta} is complete.
\end{proof}	

\subsection{Proof of Theorem \ref{weaksolution-2d}}
We are ready to prove Theorem \ref{weaksolution-2d}.
\smallskip

\textbf{Step 1. Existence.} Let $(\bm{u}^*, \phi^*, \mu^*, \theta^*)$ be a global weak solution to problem \eqref{NSCHM}--\eqref{initial} constructed in Theorem \ref{weaksolution-3d}. Under the additional assumption $\theta_0 \in C^{\gamma}(\overline{\Omega}) \cap H^1(\Omega)$ for some $\gamma \in (0,1)$, $\theta_0|_{\partial\Omega}=\theta_{\mathrm{b}} \in H^{\frac{3}{2}}(\partial \Omega)$, we consider the auxiliary problem \eqref{uphitotheta1}--\eqref{uphitotheta3} with the given data $\bm{u}=\bm{u}^*$, $\phi=\phi^*$. According to Proposition \ref{uphitotheta}-(ii), (iii), problem \eqref{uphitotheta1}--\eqref{uphitotheta3} admits a unique global strong solution $\theta$. Noticing that $\theta^*$ can be regarded as a global weak solution of problem \eqref{uphitotheta1}--\eqref{uphitotheta3} with the same structural data, then by the weak-strong uniqueness result Proposition \ref{uphitotheta}-(iii), we have $\theta^*(t)=\theta(t)$ for $t\geqslant 0$.

Therefore, under the additional assumption on $\theta_0$ mentioned above, $(\bm{u}^*, \phi^*, \mu^*, \theta^*)$ indeed gives a global weak solution to problem \eqref{NSCHM}--\eqref{initial} with improved regularity properties for the temperature as stated in Theorem \ref{weaksolution-2d}.
\smallskip

\textbf{Step 2. Uniqueness.}
In the case of unmatched densities $\rho_1\neq \rho_2$, the uniqueness of weak solutions is beyond reach even in the absence of temperature coupling (recall Remark \ref{uniqueness-AGG}). On the other hand, if the mobility and thermal diffusivity depend both on the order parameter and on the temperature, additional regularity assumptions on $\theta_1$ seem necessary to guarantee the uniqueness (see Remark \ref{weaksolution-2d'}), as will become apparent in the subsequent proof. Therefore, in the following, we focus on the special case such that
$$
\rho_1=\rho_2, \quad
m(\phi, \theta) \equiv m(\phi), \quad
\kappa(\phi, \theta) \equiv \kappa(\theta).
$$

Let $(\u_i,\phi_i,\mu_i,\theta_i)$, $i=1,2$, be two global weak solutions to problem \eqref{NSCHM}--\eqref{initial} constructed in Step 1 satisfying the same boundary and initial conditions. For simplicity, we denote their difference by
$$
\u = \u_1 - \u_2, \quad
\phi = \phi_1 - \phi_2, \quad
\mu = \mu_1 - \mu_2, \quad
\theta = \theta_1 - \theta_2.
$$

We first consider the equation satisfied by $\u$. For every $\v \in \V$, it holds
\begin{align}
&\left\langle \partial_t \u, \v\right\rangle_{\V}
+ \int_{\Omega} (\u_1 \cdot \nabla \u  + \u \cdot \nabla \u_2) \cdot \v \dx
+ \int_{\Omega} 2 \nu(\phi_1, \theta_1) D \u: D \v \dx \notag \\
&\qquad + \int_{\Omega} 2 (\nu(\phi_1, \theta_1) - \nu(\phi_2, \theta_2)) D \u_2: D \v \dx
\notag \\
&\quad = \int_{\Omega} \big(
\lambda(\theta_1) \nabla \phi_1 \otimes \nabla \phi_1 - \lambda(\theta_2) \nabla \phi_2 \otimes \nabla \phi_2 \big) : \nabla \v \dx
+ \int_{\Omega} (\bm{f}_{\mathrm{b}}(\phi_1, \theta_1) - \bm{f}_{\mathrm{b}}(\phi_2, \theta_2) ) \v \dx. \label{2D-dif-ua}
\end{align}
Taking $\v = \bm{S}^{-1} \u$ in \eqref{2D-dif-ua} yields
\begin{align}
\frac{1}{2} \frac{\d}{\dt}  \| \nabla (\bm{S}^{-1} \u) \|^2
+ \int_{\Omega} 2 \nu(\phi_1, \theta_1) D \u: D (\bm{S}^{-1} \u) \dx
= I_1 + I_2 + I_3 + I_4,
\label{2D-dif-u}
\end{align}
where
\begin{align*}
&I_1 = - \int_{\Omega} 2 (\nu(\phi_1, \theta_1) - \nu(\phi_2, \theta_2)) D \u_2:D (\bm{S}^{-1} \u) \dx, \\
&I_2 = - \big[ (\u_1 \otimes \u, \nabla (\bm{S}^{-1} \u) ) + (\u \otimes \u_2, \nabla (\bm{S}^{-1} \u) ) \big], \\
&I_3 = \int_{\Omega} \big(
\lambda(\theta_1) \nabla \phi_1 \otimes \nabla \phi_1 - \lambda(\theta_2) \nabla \phi_2 \otimes \nabla \phi_2 \big) : \nabla (\bm{S}^{-1} \u) \dx, \\
&I_4 = \int_{\Omega} (\bm{f}_{\mathrm{b}}(\phi_1, \theta_1) - \bm{f}_{\mathrm{b}}(\phi_2, \theta_2) ) \bm{S}^{-1} \u \dx.
\end{align*}
By the definition of the Stokes operator, we have $-\Delta(\bm{S}^{-1} \u) + \nabla \pi = \u$. Then it follows that
\begin{align*}
&\int_{\Omega} 2 \nu(\phi_1, \theta_1) D \u: D (\bm{S}^{-1} \u) \dx \\
&\quad = -2 \int_{\Omega} \mathrm{div} ( \nu(\phi_1,\theta_1) D (\bm{S}^{-1} \u) ) \cdot \u \dx \\
&\quad = - \int_{\Omega} \nu(\phi_1,\theta_1) (\Delta \bm{S}^{-1} \u) \cdot \u \dx
- 2 \int_{\Omega}  \big[D (\bm{S}^{-1} \u) (\partial_1 \nu \nabla \phi_1 + \partial_2 \nu \nabla \theta_1)\big]\cdot \u \dx \\
&\quad = \int_{\Omega} \nu(\phi_1,\theta_1) | \u |^2 \dx
- I_5 - I_6 - I_7,
\end{align*}
where
\begin{align*}
&I_5 = \int_{\Omega} \nu(\phi_1,\theta_1) \nabla \pi \cdot \u \dx, \\
&I_6 =  2 \int_{\Omega} \big[D (\bm{S}^{-1} \u) \partial_1 \nu \nabla \phi_1\big] \cdot \u \dx, \\
&I_7 =  2 \int_{\Omega} \big[ D (\bm{S}^{-1} \u) \partial_2 \nu \nabla \theta_1\big] \cdot \u \dx.
\end{align*}
Therefore, we arrive at
\begin{align*}
\frac{1}{2} \frac{\d}{\dt}  \| \nabla (\bm{S}^{-1} \u) \|^2
+ \int_{\Omega} \nu(\phi_1,\theta_1) | \u |^2 \dx
= \sum\limits_{k=1}^7 I_k.
\end{align*}
We write $I_1 = I_{1a} + I_{1b}$, where
\begin{align*}
&I_{1a} = - \int_{\Omega} 2 (\nu(\phi_1, \theta_1) - \nu(\phi_1, \theta_2)) D \u_2:D (\bm{S}^{-1} \u) \dx, \\
&I_{1b} = - \int_{\Omega} 2 (\nu(\phi_1, \theta_2) - \nu(\phi_2, \theta_2)) D \u_2:D (\bm{S}^{-1} \u) \dx.
\end{align*}
Using a similar reasoning as in \cite[Section 3]{me}, we have the following estimates
\begin{align*}
&I_{1a} \leqslant \frac{\underline{\nu}}{20} \| \u \|^2 + \frac{\underline{\kappa}}{10} \| \nabla \theta \|^2
+ C \| \nabla \u_2 \|^2 (\| \theta \|^2 + \| \nabla (\bm{S}^{-1} \u) \|^2), \\
&I_2 \leqslant \frac{\underline{\nu}}{20} \| \u \|^2 + C ( \| \u_1 \|_{\V}^2 + \| \u_2 \|_{\V}^2 )
\| \nabla (\bm{S}^{-1} \u) \|^2, \\
&I_3 \leqslant \frac{1}{40} \| \nabla \phi \|^2 + \frac{\underline{\kappa}}{10} \| \nabla \theta \|^2
+ C ( \| \nabla \phi_1 \|_{L^{\infty}(\Omega)}^2 + \| \nabla \phi_2 \|_{L^{\infty}(\Omega)}^2 + \| \phi_1 \|_{H^2(\Omega)}^2 )
\| \nabla (\bm{S}^{-1} \u) \|^2, \\
&I_4 \leqslant C \| \theta \|^2 + \| \nabla (\bm{S}^{-1} \u) \|^2, \\
&I_7 \leqslant \frac{\underline{\nu}}{20} \| \u \|^2 + C \| \theta_1 \|_{H^2}^2 \| \nabla (\bm{S}^{-1} \u) \|^2.
\end{align*}
Let us recall the $L^2$-product estimate given in \cite[Proposition C.1]{Giorgini2019}:
$$
\| \phi D (\bm{S}^{-1} \u) \| \leqslant C \| \phi \|_{H^1(\Omega)} \| \nabla (\bm{S}^{-1} \u) \| \ln^{\frac12}\left( e \frac{\| \nabla (\bm{S}^{-1} \u) \|_{H^1(\Omega)}}{\| \nabla (\bm{S}^{-1} \u) \|} \right).
$$
Using the Poincar\'{e}--Wirtinger inequality, Poincar\'{e}'s inequality and Young's inequality, we can deduce that
\begin{align*}
I_{1b} &\leqslant 2 \|\partial_1\nu\|_{L^\infty(\Omega)} \| D \u_2 \| \| \phi D (\bm{S}^{-1} \u) \|  \\
&\leqslant C \| \nabla \u_2 \| \| \phi \|_{H^1(\Omega)} \| \nabla (\bm{S}^{-1} \u) \| \ln^{\frac12}\left( e \frac{\| \nabla (\bm{S}^{-1} \u) \|_{H^1(\Omega)}}{\| \nabla (\bm{S}^{-1} \u) \|} \right) \\
&\leqslant \frac{1}{40} \| \nabla \phi \|^2
+ C \| \nabla \u_2 \|^2
\mathcal{Y}(t) \ln \left( \frac{C}{\mathcal{Y}(t)} \right),
\end{align*}
where
\begin{align}
\mathcal{Y}(t)= \| \nabla (\bm{S}^{-1} \u)(t) \|^2 + \int_{\Omega} m(\phi_1(t)) | \nabla \G_{\phi_1} \phi (t)|^2 \dx
+ \| \theta(t)\|^2.
\label{2D-dif-Y}
\end{align}
In the above estimate for $I_{1b}$, we have used the fact that $\mathcal{Y}(t) \in L^\infty(0,\infty)$ and the function $s\ln\left(\frac{C}{s}\right)$ is increasing for $s\in (0,L]$ if $C\geqslant eL$. As far as the term $I_5$ is concerned, using the Gagliardo--Nirenberg inequality and the $L^4$-estimate for the pressure $\pi$ (see \cite[Lemma 3.1]{pressure}), we have
\begin{align*}
I_5 &= - \int_{\Omega} (\partial_1 \nu \nabla \phi_1 + \partial_2 \nu \nabla \theta_1 ) \cdot \u \pi
\dx \\
&\leqslant ( \| \partial_1 \nu\|_{L^{\infty}(\Omega)} \| \nabla \phi_1 \|_{L^4(\Omega)}
+ \| \partial_2 \nu\|_{L^{\infty}(\Omega)} \| \nabla \theta_1 \|_{L^4(\Omega)} )
\| \u \| \| \pi \|_{L^4(\Omega)} \notag \\
&\leqslant C (\| \phi_1\|_{L^{\infty}(\Omega)}^{\frac{1}{2}} \| \phi_1 \|_{H^2(\Omega)}^{\frac{1}{2}} +
\| \theta_1\|_{L^{\infty}(\Omega)}^{\frac{1}{2}} \| \theta_1 \|_{H^2(\Omega)}^{\frac{1}{2}}  )
\| \u \|^{\frac{3}{2}} \| \nabla \bm{S}^{-1} \u \|^{\frac{1}{2}} \notag \\
&\leqslant \frac{\underline{\nu}}{20} \| \u \|^2 +
C ( \| \phi_1\|_{H^2(\Omega)}^2 + \| \theta_1\|_{H^2(\Omega)}^2 ) \| \nabla \bm{S}^{-1} \u\|^2.
\end{align*}
Finally, the term $I_6$ can be estimated as in \cite{Giorgini2019}:
\begin{align*}
I_6 &\leqslant C \| \partial_1 \nu \|_{L^{\infty}} \| \nabla \phi_1 \|_{L^4(\Omega)} \| D \bm{S}^{-1} \u \|_{L^4(\Omega)} \| \u \| \\
&\leqslant C  \| \phi_1 \|_{L^{\infty}(\Omega)}^{\frac{1}{2}} \| \phi_1 \|_{H^2(\Omega)}^{\frac{1}{2}}\| \nabla \bm{S}^{-1} \u \|^{\frac12} \| \u \|^{\frac32} \\
&\leqslant \frac{\underline{\nu}}{20} \| \u \|^2 + C \| \phi_1\|_{H^2(\Omega)}^2 \| \nabla \bm{S}^{-1} \u\|^2.
\end{align*}

Next, we consider the equations for $(\phi, \mu)$,  which read as follows
\begin{align}
& \partial_t \phi + \u_1 \cdot \nabla \phi + \u \cdot \nabla \phi_2 =
\mathrm{div} (m(\phi_1) \nabla \mu)
+ \mathrm{div} [(m(\phi_1) - m(\phi_2)) \nabla \mu_2], \label{2D-dif-p} \\
&\mu  = -\Delta \phi + W^{\prime}(\phi_1) - W^{\prime}(\phi_2). \label{2D-dif-m}
\end{align}
Multiply \eqref{2D-dif-p} by $\G_{\phi_1} \phi$ and then integrate over $\Omega$. Recalling the definition of $\G_{\phi_1}$ and the fact $W^{\prime \prime} \geqslant -c_W$, we have
\begin{align*}
-( \mathrm{div} (m(\phi_1) \nabla \mu), \G_{\phi_1} \phi)
	&= (\nabla \mu, m(\phi_1) \nabla \G_{\phi_1} \phi) \notag \\
	&= (\mu, \phi) \notag \\
    & \geqslant \| \nabla \phi \|^2 - c_W\|\phi\|^2\\
	&\geqslant \frac{1}{2} \| \nabla \phi \|^2 - \frac{c_W^2}{2} \| \phi \|_{V_{(0)}^{\prime}}^2.
\end{align*}
This implies that
\begin{align}
(\partial_t \phi , \G_{\phi_1} \phi) + \frac{1}{2} \| \nabla \phi \|^2 & \leqslant
\frac{c_W^2}{2} \| \phi \|_{V_{(0)}^{\prime}}^2
+ (\phi \u_1, \nabla \G_{\phi_1} \phi)
+ (\phi_2 \u, \nabla \G_{\phi_1} \phi)  \notag \\
&\quad - \big((m(\phi_1) - m(\phi_2)) \nabla \mu_2 , \nabla \G_{\phi_1} \phi\big).
\label{2D-dif-phi}
\end{align}
Using the Gagliardo--Nirenberg inequality, the Sobolev embedding theorem, Poincar\'{e}'s inequality and Young's inequality, we get
\begin{align*}
- ((m(\phi_1) - m(\phi_2)) \nabla \mu_2 , \nabla \G_{\phi_1} \phi) &\leqslant
\|m'\|_{L^\infty(\Omega)} \| \phi \|_{L^4(\Omega)} \| \nabla \mu_2 \| \| \nabla \G_{\phi_1} \phi \|_{L^4(\Omega)} \\
&\leqslant C\| \nabla \mu_2 \| \| \nabla \G_{\phi_1} \phi \|^{\frac12}\|\phi\| \| \nabla \phi \|^{\frac12}
\\
&\leqslant  C\| \nabla \mu_2 \|  \| \nabla \G_{\phi_1} \phi \| \| \nabla \phi \| \\
&\leqslant \frac{1}{40} \| \nabla \phi \|^2 + C\| \nabla \mu_2 \|^2 \| \nabla \G_{\phi_1} \phi \|^2,
\end{align*}
and
\begin{align*}
(\phi \u_1, \nabla \G_{\phi_1} \phi) &\leqslant \| \phi \|_{L^4(\Omega)} \| \u_1 \|_{L^4(\Omega)} \| \nabla \G_{\phi_1} \phi \| \\
&\leqslant \frac{1}{40} \| \nabla \phi \|^2 + C \| \u_1 \|_{L^4(\Omega)}^2 \| \nabla \G_{\phi_1} \phi \|^2, \\
(\phi_2 \u, \nabla \G_{\phi_1} \phi) &\leqslant \| \phi_2 \|_{L^{\infty}(\Omega)} \| \u \| \| \nabla \G_{\phi_1} \phi \| \\
&\leqslant \frac{\underline{\nu}}{20} \| \u \|^2 + C \| \nabla \G_{\phi_1} \phi \|^2.
\end{align*}
The term $(\partial_t \phi , \G_{\phi_1} \phi)$ can be treated as in \cite{New_Results_Mobility} such that
\begin{align}
(\partial_t \phi , \G_{\phi_1} \phi) &=
\frac{\d}{\dt} \frac{1}{2} \int_{\Omega} m(\phi_1) | \nabla \G_{\phi_1} \phi |^2 \dx
+ \frac{1}{2} \int_{\Omega} \nabla \G \partial_t{\phi_1} \cdot m^{\prime \prime}(\phi_1) \nabla \phi_1 |\nabla \G_{\phi_1} \phi |^2 \dx \notag \\
&\quad + \int_{\Omega} \nabla \G \partial_t{\phi_1} \cdot m^{\prime}(\phi_1) (\nabla^2 \G_{\phi_1} \phi \nabla \G_{\phi_1} \phi) \dx,
\label{transfer-phi-1}
\end{align}
with the following estimates (see \cite[(3.47), (3.50)]{New_Results_Mobility})
\begin{align*}
\left|\frac{1}{2} \int_{\Omega} \nabla \G \partial_t{\phi_1} \cdot m^{\prime \prime}(\phi_1) \nabla \phi_1 |\nabla \G_{\phi_1} \phi |^2 \dx\right|
\leqslant
\frac{1}{40} \| \nabla \phi \|^2
+ C \left( \| \nabla \G \partial_t{\phi_1} \|^2 + \| \phi_1 \|_{H^2(\Omega)}^4 \right)
\| \nabla \G_{\phi_1} \phi \|^2,
\end{align*}
\begin{align*}
\left|\int_{\Omega} \nabla \G \partial_t{\phi_1} \cdot m^{\prime}(\phi_1) (\nabla^2 \G_{\phi_1} \phi \nabla \G_{\phi_1} \phi) \dx\right|
\leqslant
\frac{1}{40} \| \nabla \phi \|^2
+ C \left( \| \nabla \G \partial_t{\phi_1} \|^2 + \| \phi_1 \|_{H^2(\Omega)}^4 \right)
\| \nabla \G_{\phi_1} \phi \|^2.
\end{align*}
\medskip

Finally, we investigate the equation for $\theta$, that is,
\begin{equation}
\partial_t \theta + \u_1 \cdot \nabla \theta + \u \cdot \nabla \theta_2
= \mathrm{div} \big(
( \kappa(\theta_1) - \kappa(\theta_2) ) \nabla \theta_1
+ \kappa(\theta_2) \nabla \theta
\big).\label{2D-dif-the}
\end{equation}
Multiplying \eqref{2D-dif-the} by $\theta$ and integrating over $\Omega$, we have
\begin{align}
\frac{1}{2} \frac{\d}{\dt} \| \theta \|^2
+ \int_{\Omega} \kappa(\theta_2) | \nabla \theta |^2 \dx
= -\int_{\Omega} (\u \cdot \nabla \theta_2) \theta \dx
-\int_{\Omega} ( \kappa(\theta_1) - \kappa(\theta_2) ) \nabla \theta_1 \cdot \nabla \theta \dx.
\label{2D-dif-theb}
\end{align}
The second term on the right-hand side can be estimated in the same way as in \eqref{theta-uniqueness}:
\begin{align*}
-\int_{\Omega} ( \kappa(\theta_1) - \kappa(\theta_2) ) \nabla \theta_1 \cdot \nabla \theta \dx
\leqslant \frac{\underline{\kappa}}{10} \| \nabla \theta \|^2 + C \| \theta_1 \|_{H^2(\Omega)}^2 \| \theta \|^2.
\end{align*}
Concerning the first term on the right-hand side, we apply the Gagliardo--Nirenberg inequality, H\"{o}lder's inequality and Young's inequality to obtain
\begin{align*}
-\int_{\Omega} (\u \cdot \nabla \theta_2) \theta \dx &\leqslant
\| \u \| \| \nabla \theta_2 \|_{L^4(\Omega)} \| \theta \|_{L^4(\Omega)} \\
&\leqslant \frac{\underline{\nu}}{20} \| \u \|^2
+ C \| \nabla \theta_2 \| \| \theta_2 \|_{H^2(\Omega)} \| \theta \| \| \nabla \theta \| \\
&\leqslant \frac{\underline{\nu}}{20} \| \u \|^2 + \frac{\underline{\kappa}}{10} \| \nabla \theta \|^2
+ C \| \theta_2 \|_{H^2(\Omega)}^2 \| \theta \|^2.
\end{align*}

Combining all the above estimates, we deduce from \eqref{2D-dif-u}, \eqref{2D-dif-phi} and \eqref{2D-dif-theb} that
\begin{align}
&\frac{\d}{\dt} \mathcal{Y}(t)
+ \underline{\nu}\| \u \|^2
+ \frac{1}{2} \| \nabla \phi \|^2
+ \underline{\kappa} \| \nabla \theta \|^2
 \leqslant C\mathcal{H}(t) \mathcal{Y}(t)
\ln \left( \frac{C}{\mathcal{Y}(t)} \right),
\label{2D-Diff}
\end{align}
where $C>0$ is sufficiently large, $\mathcal{Y}(t)$ is defined as in \eqref{2D-dif-Y} and
\begin{align*}
\mathcal{H}(t) & = 1 + \| \u_1(t) \|_{\V}^2
+ \| \u_2(t) \|_{\V}^2 + \| \phi_1(t) \|_{W^{2,3}(\Omega)}^2
+ \| \phi_2(t) \|_{W^{2,3}(\Omega)}^2
+ \| \phi_1(t) \|_{H^2(\Omega)}^4 \\
&\quad + \|\nabla \mu_2(t) \|^2
 + \| \nabla \G \partial_t{\phi_1}(t) \|^2
 + \| \theta_1(t) \|_{H^2(\Omega)}^2
 + \| \theta_2(t) \|_{H^2(\Omega)}^2.
\end{align*}
Noticing that $\mathcal{H}(t) \in L^1(0,T)$ for any $T>0$, applying Osgood's Lemma (see \cite[Lemma 3.4]{Osgood}), we can deduce from \eqref{2D-Diff} that, if $\mathcal{Y}(0)=0$, then $\mathcal{Y}(t) =0$ for all $t\in [0,T]$. Since $T>0$ is arbitrary, we thus prove the uniqueness of global weak solutions on $[0,\infty)$.

The proof of Theorem \ref{weaksolution-2d} is complete. \qed

\begin{remark} \rm
In the more general case where $m$, $\kappa$ depend on both $\phi$ and $\theta$, additional difficulties may occur. For example, the identity \eqref{transfer-phi-1} becomes (at least formally)
\begin{align*}
	\big(\partial_t \phi , \G_{(\phi_1, \theta_1)} \phi\big)
    &=
	\frac{\d}{\dt} \frac{1}{2} \int_{\Omega} m(\phi_1, \theta_1) | \nabla \G_{(\phi_1, \theta_1)} \phi |^2 \dx
	+ \frac{1}{2} \int_{\Omega} \nabla \G \partial_t{\phi_1} \cdot \partial_1^2 m \nabla \phi_1
	|\nabla \G_{(\phi_1, \theta_1)} \phi |^2 \dx \\
	&\quad + \int_{\Omega} \nabla \G \partial_t{\phi_1} \cdot \partial_1 m (\nabla^2 \G_{(\phi_1, \theta_1)} \phi \nabla \G_{(\phi_1, \theta_1)} \phi) \dx
	+ \frac{1}{2} \int_{\Omega} \partial_2 m \partial_t \theta_1 |\nabla \G_{(\phi_1, \theta_1)} \phi |^2 \dx \\
	&\quad + \frac{1}{2} \int_{\Omega} \nabla \G \partial_t{\phi_1} \cdot \partial_{1} \partial_2 m \nabla \theta_1 |\nabla \G_{(\phi_1, \theta_1)} \phi |^2 \dx .
\end{align*}
Here, the notation $\G_{(\phi_1, \theta_1)}$ is defined analogously to \eqref{Pre-Gq}, with the function $q$ replaced by a pair of functions $(p, q)$. The second and third terms on the right-hand side have been analyzed as above. The fourth term on the right-hand side can be handled via the improved regularity $\partial_t \theta_1 \in L^2_{\mathrm{uloc}}([0,\infty);L^2(\Omega))$ and the Gagliardo--Nirenberg inequality:
\begin{align*}
\frac{1}{2} \int_{\Omega} \partial_2 m \partial_t \theta_1 |\nabla \G_{(\phi_1, \theta_1)} \phi |^2 \dx
\leqslant \frac{1}{40} \| \nabla \phi \|^2
+ C \big(1+\| \partial_t \theta_1 \|^2\big) \| \nabla \G_{(\phi_1, \theta_1)} \phi \|^2.
\end{align*}
However, the last term presents some challenges, necessitating the additional assumption posed in Remark \ref{weaksolution-2d'}. In particular, this term is structurally analogous to the second term, whose estimation requires $\phi_1 \in L^4_\mathrm{uloc}([0,\infty);H^2(\Omega))$.
\end{remark}

\section*{Declarations}
\noindent
\textbf{Conflict of interest.} The authors have no competing interests to declare that are relevant to the content of this article.
\smallskip
\\
\noindent
\textbf{Funding.} The research of H. Wu was partially supported by Natural Science Foundation of Shanghai (Grant number 25ZR1401023).
\smallskip
\\
\noindent
\textbf{Acknowledgments.}
The authors are grateful to the reviewer for several valuable comments. Part of the work was done during H. Wu's participation in the Thematic Program on Free Boundary Problems at the Erwin Schr\"{o}dinger International Institute for Mathematics and Physics (ESI), whose hospitality is gratefully acknowledged. H. Wu is a member of the Key Laboratory of Mathematics for Nonlinear Sciences (Fudan University), Ministry of Education of China.



\begin{thebibliography}{99}

	\bibitem{Abels2009}
	H. Abels,
	{\it On a diffuse interface model for two-phase flows of viscous, incompressible fluids with matched densities,}
	Arch. Ration. Mech. Anal., 194 (2009), 463--506.
	
	\bibitem{JMFM}
	H. Abels, D. Depner, and H. Garcke,
	{\it Existence of weak solutions for a diffuse interface model for two-phase flows of incompressible fluids with different densities,}
	J. Math. Fluid Mech., 15 (2013), 453--480.
	
	\bibitem{AGG-P}
	H. Abels, D. Depner, and H. Garcke,
	{\it On an incompressible Navier--Stokes/Cahn--Hilliard system with degenerate mobility,}
	Ann. Inst. H. Poincar\'e C Anal. Non Lin\'eaire, 30 (2013), 1175--1190.
	
	\bibitem{longtime_ann}
	H. Abels, H. Garcke, and A. Giorgini,
	{\it Global regularity and asymptotic stabilization for the incompressible Navier--Stokes--Cahn--Hilliard model with unmatched densities,}
	Math. Ann., 389 (2024), 1267--1321.
	
	\bibitem{AGGmodel}
	H. Abels, H. Garcke, and G. Gr\"un,
	{\it Thermodynamically consistent, frame indifferent diffuse interface models for incompressible two-phase flows with different densities,}
	Math. Models Methods Appl. Sci., 22 (2012), 1150013.
	
	\bibitem{mobility-multi-phase}
	H. Abels, H. Garcke, and A. Poiatti,
	{\it Mathematical analysis of a diffuse interface model for multi-phase flows of incompressible viscous fluids with different densities,}
	J. Math. Fluid Mech., 26 (2024), 29.
	
	\bibitem{Abels2025}
	H. Abels, A. Marveggio, and A. Poiatti,
	{\it Well-posedness and sharp interface limit of a non-isothermal Navier--Stokes/Allen--Cahn model,}
	Math. Models Methods Appl. Sci., to appear, arXiv:2511.11892, 2025.
	
	\bibitem{AW21}
	H. Abels and J. Weber,
	{\it Local well-posedness of a quasi-incompressible two-phase flow,}
	J. Evol. Equ., 21 (2021), 3477--3502.
	
	\bibitem{Abels2007}
	H. Abels and M. Wilke,
	{\it Convergence to equilibrium for the Cahn--Hilliard equation with a logarithmic free energy,}
	Nonlinear Anal., 67 (2007), 3176--3193.
	
	\bibitem{Amann}
	H. Amann,
	{\it Linear and Quasilinear Parabolic Problems. Vol. I. Abstract Linear Theory,}
	Monogr. Math., 89, Birkh\"auser Boston, Boston, MA, 1995.
	
	\bibitem{phasefield1}
	D. M. Anderson, G. B. McFadden, and A. A. Wheeler,
	{\it Diffuse-interface methods in fluid mechanics,}
	Annu. Rev. Fluid Mech., 30 (1998), 139--165.
	
	\bibitem{Osgood}
	H. Bahouri, J.-Y. Chemin, and R. Danchin,
	{\it Fourier Analysis and Nonlinear Partial Differential Equations,}
	Springer-Verlag, Heidelberg, 2011.
	
	\bibitem{Benard1901}
	H. B\'enard,
	{\it Les tourbillons cellulaires dans une nappe liquide propageant de la chaleur par convection: en r\'egime permanent,}
	Gauthier-Villars, Paris, 1901.
	
	\bibitem{phasefield3}
	J. L. Boldrini,
	{\it Phase field: a methodology to model complex material behavior,}
	in {\it Advances in Mathematics and Applications}, Springer, Cham, 2018, 67--103.
	
	\bibitem{BF2013}
	F. Boyer and P. Fabrie,
	{\it Mathematical Tools for the Study of the Incompressible Navier--Stokes Equations and Related Models,}
	Appl. Math. Sci., 183, Springer, New York, 2013.
	
	\bibitem{sharp2}
	G. Caginalp and W. Xie,
	{\it Phase-field and sharp-interface alloy models,}
	Phys. Rev. E, 48 (1993), 1897--1909.
	
	\bibitem{Cahn1958}
	J. W. Cahn and J. E. Hilliard,
	{\it Free energy of a nonuniform system. I. Interfacial free energy,}
	J. Chem. Phys., 28 (1958), 258--267.
	
	\bibitem{Charles1960}
	G. E. Charles and S. G. Mason,
	{\it The coalescence of liquid drops with flat liquid/liquid interfaces,}
	J. Colloid Sci., 15 (1960), 236--267.
	
	\bibitem{me}
	L.-X. Chen,
	{\it Global well-posedness for a two-dimensional Navier--Stokes--Cahn--Hilliard--Boussinesq system with singular potential,}
	Commun. Math. Sci., 23 (2025), 509--540.
	
	\bibitem{New_Results_Mobility}
	M. Conti, P. Galimberti, S. Gatti, and A. Giorgini,
	{\it New results for the Cahn--Hilliard equation with non-degenerate mobility: well-posedness and long-time behavior,}
	Calc. Var. Partial Differential Equations, 64 (2025), 87.
	
	\bibitem{Conti2020}
	M. Conti and A. Giorgini,
	{\it Well-posedness for the Brinkman--Cahn--Hilliard system with unmatched viscosities,}
	J. Differential Equations, 268 (2020), 6350--6384.
	
	\bibitem{Liuextend}
	F. De Anna, C. Liu, A. Schl\"omerkemper, and J.-E. Sulzbach,
	{\it Temperature dependent extensions of the Cahn--Hilliard equation,}
	Nonlinear Anal. Real World Appl., 77 (2024), 104056.
	
	\bibitem{phasefield4}
	Q. Du and X.-B. Feng,
	{\it The phase field method for geometric moving interfaces and their numerical approximations,}
	Handb. Numer. Anal., 21 (2020), 425--508.
	
	\bibitem{Eggers1997}
	J. Eggers,
	{\it Nonlinear dynamics and breakup of free-surface flows,}
	Rev. Modern Phys., 69 (1997), 865--939.
	
	\bibitem{phasefield2}
	P. C. Fife,
	{\it Models for phase separation and their mathematics,}
	Electron. J. Differential Equations, 2000 (2000), 1--26.
	
	\bibitem{pressure}
	C. G. Gal, A. Giorgini, M. Grasselli, and A. Poiatti,
	{\it Global well-posedness and convergence to equilibrium for the Abels--Garcke--Gr\"un model with nonlocal free energy,}
	J. Math. Pures Appl. (9), 178 (2023), 46--109.
	
	\bibitem{GGP}
	C. Giorgi, M. Grasselli, and V. Pata,
	{\it Uniform attractors for a phase-field model with memory and quadratic nonlinearity,}
	Indiana Univ. Math. J., 48 (1999), 1395--1445.
	
	\bibitem{AGG2dstrong}
	A. Giorgini,
	{\it Well-posedness of the two-dimensional Abels--Garcke--Gr\"un model for two-phase flows with unmatched densities,}
	Calc. Var. Partial Differential Equations, 60 (2021), 100.
	
	\bibitem{AGG3dstrong}
	A. Giorgini,
	{\it Existence and stability of strong solutions to the Abels--Garcke--Gr\"un model in three dimensions,}
	Interfaces Free Bound., 24 (2022), 565--608.
	
	\bibitem{HwuHeleshaw}
	A. Giorgini, M. Grasselli, and H. Wu,
	{\it The Cahn--Hilliard--Hele--Shaw system with singular potential,}
	Ann. Inst. H. Poincar\'e C Anal. Non Lin\'eaire, 35 (2018), 1079--1118.
	
	\bibitem{Giorgini2019}
	A. Giorgini, A. Miranville, and R. Temam,
	{\it Uniqueness and regularity for the Navier--Stokes--Cahn--Hilliard system,}
	SIAM J. Math. Anal., 51 (2019), 2535--2574.
	
	\bibitem{MBoussinesq}
	M. Grasselli and A. Poiatti,
	{\it The Cahn--Hilliard--Boussinesq system with singular potential,}
	Commun. Math. Sci., 20 (2022), 897--946.
	
	\bibitem{GP2025}
	M. Grasselli and A. Poiatti,
	{\it Convergence to equilibrium of weak solutions to the Cahn--Hilliard equation with non-degenerate mobility and singular potential,}
	preprint, 2025. arXiv:2510.17296.
	
	\bibitem{GL2015}
	Z. Guo and P. Lin,
	{\it A thermodynamically consistent phase-field model for two-phase flows with thermocapillary effects,}
	J. Fluid Mech., 766 (2015), 226--271.
	
	\bibitem{EnVarA2}
	Z. Guo, P. Lin, and Y. Wang,
	{\it Continuous finite element schemes for a phase field model in two-layer fluid B\'enard--Marangoni convection computations,}
	Comput. Phys. Commun., 185 (2014), 63--78.
	
	\bibitem{Sobolev-manifold}
	E. Hebey,
	{\it Sobolev Spaces on Riemannian Manifolds,}
	Lecture Notes in Math., 1635, Springer-Verlag, Berlin, 1996.
	
	\bibitem{Marangoni-coffee}
	H. Hu and R. G. Larson,
	{\it Marangoni effect reverses coffee-ring depositions,}
	J. Phys. Chem. B, 110 (2006), 7090--7094.
	
	\bibitem{Marangoni_review}
	D. Johnson and R. Narayanan,
	{\it A tutorial on the Rayleigh--Marangoni--B\'enard problem with multiple layers and side-wall effects,}
	Chaos, 9 (1999), 124--140.
	
	\bibitem{viscosity-phi}
	J. S. Kim,
	{\it Phase-field models for multi-component fluid flows,}
	Commun. Comput. Phys., 12 (2012), 613--661.
	
	\bibitem{Marangoni_electron}
	P. Lee, P. Quested, and M. McLean,
	{\it Modelling of Marangoni effects in electron beam melting,}
	Philos. Trans. Roy. Soc. A, 356 (1998), 1027--1043.
	
	\bibitem{lopes18}
	J. H. Lopes and G. Planas,
	{\it Well-posedness for a non-isothermal flow of two viscous incompressible fluids,}
	Commun. Pure Appl. Anal., 17 (2018), 2455--2477.
	
	\bibitem{lopes-AC}
	J. H. Lopes and G. Planas,
	{\it On a non-isothermal incompressible Navier--Stokes--Allen--Cahn system,}
	Monatsh. Math., 195 (2021), 687--715.
	
	\bibitem{lopes}
	J. H. Lopes and G. Planas,
	{\it Existence of solutions for a non-isothermal Navier--Stokes--Allen--Cahn system with thermo-induced coefficients,}
	Electron. J. Differential Equations, 2022 (2022), 1--22.
	
	\bibitem{LB96}
	S. A. Lorca and J. L. Boldrini,
	{\it Stationary solutions for generalized Boussinesq models,}
	J. Differential Equations, 124 (1996), 389--406.
	
	\bibitem{LB99}
	S. A. Lorca and J. L. Boldrini,
	{\it The initial value problem for a generalized Boussinesq model,}
	Nonlinear Anal., 36 (1999), 457--480.
	
	\bibitem{Marangoni}
	C. Marangoni,
	{\it Sull'espansione delle goccie d'un liquido galleggianti sulla superfice di altro liquido,}
	Fratelli Fusi, 1865.
	
	\bibitem{Marangoni_welding}
	K. Mills, B. Keene, R. Brooks, and A. Shirali,
	{\it Marangoni effects in welding,}
	Philos. Trans. Roy. Soc. A, 356 (1998), 911--925.
	
	\bibitem{inequality}
	A. Miranville and S. Zelik,
	{\it Robust exponential attractors for Cahn--Hilliard type equations with singular potentials,}
	Math. Methods Appl. Sci., 27 (2004), 545--582.
	
	\bibitem{NB20}
	B. A. Nerger, P.-T. Brun, and C. M. Nelson,
	{\it Marangoni flows drive the alignment of fibrillar cell-laden hydrogels,}
	Sci. Adv., 6 (2020), eaaz7748.
	
	\bibitem{Oberbeck}
	G. Peralta,
	{\it Weak and very weak solutions to the viscous Cahn--Hilliard--Oberbeck--Boussinesq phase-field system on two-dimensional bounded domains,}
	J. Evol. Equ., 22 (2022), 12.
	
	\bibitem{Marangoni_crystal}
	A. Pimpinelli and J. Villain,
	{\it Physics of Crystal Growth,}
	Cambridge Univ. Press, Cambridge, 1999.
	
	\bibitem{Roubicek2005}
	T. Roub\'{\i}$\check{\mathrm{c}}$ek,
	{\it Nonlinear Partial Differential Equations with Applications,}
	Birkh\"auser, Basel, 2005.
	
	\bibitem{sharp1}
	L. Rubinshtein,
	{\it The Stefan Problem,}
	Amer. Math. Soc., Providence, RI, 1971.
	
	\bibitem{Sohr}
	H. Sohr,
	{\it The Navier--Stokes Equations: An Elementary Functional Analytic Approach,}
	Birkh\"auser, Basel, 2012.
	
	\bibitem{Marangoni_turbulence}
	C. V. Sternling and L. Scriven,
	{\it Interfacial turbulence: hydrodynamic instability and the Marangoni effect,}
	AIChE J., 5 (1959), 514--523.
	
	\bibitem{Wang2020}
	S. Sun, J. Li, J. Zhao, and Q. Wang,
	{\it Structure-preserving numerical approximations to a non-isothermal hydrodynamic model of binary fluid flows,}
	J. Sci. Comput., 83 (2020), 50.
	
	\bibitem{EnVarA1}
	P. Sun, C. Liu, and J. Xu,
	{\it Phase field model of thermo-induced Marangoni effects in mixtures and its numerical simulations with a mixed finite element method,}
	Commun. Comput. Phys., 6 (2009), 1095--1119.
	
	\bibitem{Sun2024}
	Y. Sun, J. Wu, M. Jiang, S. M. Wise, and Z. Guo,
	{\it A thermodynamically consistent phase-field model and an entropy stable numerical method for simulating two-phase flows with thermocapillary effects,}
	Appl. Numer. Math., 206 (2024), 161--189.
	
	\bibitem{ZZF2013}
	Y. Sun and Z. Zhang,
	{\it Global regularity for the initial-boundary value problem of the 2D Boussinesq system with variable viscosity and thermal diffusivity,}
	J. Differential Equations, 255 (2013), 1069--1085.
	
	\bibitem{Marangoni_interface}
	M. G. Velarde and R. K. Zeytounian,
	{\it Interfacial Phenomena and the Marangoni Effect,}
	Springer-Verlag, Vienna/New York, 2002.
	
	\bibitem{Wangxiaoming}
	X.-M. Wang and H. Wu,
	{\it Global weak solutions to the Navier--Stokes--Darcy--Boussinesq system for thermal convection in coupled free and porous media flows,}
	Adv. Differential Equations, 26 (2021), 1--44.
	
	\bibitem{HW2017}
	H. Wu,
	{\it Well-posedness of a diffuse-interface model for two-phase incompressible flows with thermo-induced Marangoni effect,}
	European J. Appl. Math., 28 (2017), 380--434.
	
	\bibitem{CHreview}
	H. Wu,
	{\it A review on the Cahn--Hilliard equation: classical results and recent advances in dynamic boundary conditions,}
	Electron. Res. Arch., 30 (2022), 2788--2832.
	
	\bibitem{HWUXX2013}
	H. Wu and X. Xu,
	{\it Analysis of a diffuse-interface model for binary viscous incompressible fluids with thermo-induced Marangoni effects,}
	Commun. Math. Sci., 11 (2013), 603--633.
	
	\bibitem{Ze92}
	E. Zeidler,
	{\it Nonlinear Functional Analysis and Its Applications I,}
	Springer, New York, 1992.
	
	\bibitem{Zhao_regularity}
	K. Zhao,
	{\it Global regularity for a coupled Cahn--Hilliard--Boussinesq system on bounded domains,}
	Quart. Appl. Math., 69 (2011), 331--356.
	
	\bibitem{Zhao_longtime1}
	K. Zhao,
	{\it Large time behavior of a Cahn--Hilliard--Boussinesq system on a bounded domain,}
	Electron. J. Differential Equations, 2011 (2011), 1--21.
	
	\bibitem{Zhao_longtime2}
	K. Zhao,
	{\it Long-time dynamics of a coupled Cahn--Hilliard--Boussinesq system,}
	Commun. Math. Sci., 10 (2012), 735--749.
	
\end{thebibliography}
\end{document}